\documentclass[11pt]{elegantpaper}

\usepackage[mathscr]{eucal} 

\graphicspath{{figures/}} 

\crefname{equation}{Eq.}{Eqs.}
\crefname{assumption}{Assumption}{Assumptions}
\crefname{algocf}{Algorithm}{Algorithms}
\newcommand{\crefrangeconjunction}{--}

\makeatletter
\renewenvironment{proof}[1][\proofname]{\par
  \pushQED{\qed}%
  \normalfont \topsep6\p@\@plus6\p@\relax
  \trivlist
  \item[\hskip\labelsep
        \normalfont\bfseries 
    #1\@addpunct{.}]\ignorespaces
}{
  \popQED\endtrivlist\@endpefalse
}
\makeatother

\crefalias{enumi}{pureitem} 
\crefalias{enumii}{pureitem}
\crefalias{enumiii}{pureitem}
\crefalias{enumiv}{pureitem}
\crefformat{pureitem}{#2\textcolor{black}{#1}#3}
\crefrangeformat{pureitem}{#2\textcolor{black}{#1}#3\crefrangeconjunction#2\textcolor{black}{#1}#3}
\crefmultiformat{pureitem}{#2\textcolor{black}{#1}#3}%
  {, #2\textcolor{black}{#1}#3}%
  {, #2\textcolor{black}{#1}#3}%
  {, #2\textcolor{black}{#1}#3}

\newcommand{\vx}{\bm{x}}
\newcommand{\vy}{\bm{y}}
\newcommand{\vz}{\bm{z}}
\newcommand{\vc}{\bm{c}}

\newcommand{\vv}{\bm{v}}
\newcommand{\vk}{\bm{k}}

\newcommand{\mX}{\bm{X}}
\newcommand{\mD}{\bm{D}}
\newcommand{\mK}{\bm{K}}
\newcommand{\mI}{\bm{I}}
\newcommand{\mL}{\bm{L}}

\newcommand{\fX}{\mathcal{X}}
\newcommand{\fC}{\mathcal{C}}
\newcommand{\fN}{\mathcal{N}}
\newcommand{\fI}{\mathcal{I}}
\newcommand{\fT}{\mathcal{T}}

\newcommand{\fE}{\mathcal{E}}
\newcommand{\fF}{\mathcal{F}}

\newcommand{\hf}{\hat{f}}
\renewcommand{\hm}{\hat{m}}
\newcommand{\hell}{\hat{\ell}}
\newcommand{\hsigma}{\hat{\sigma}}
\newcommand{\hbeta}{\hat{\beta}}
\newcommand{\hvarSigma}{\hat{\varSigma}}
\newcommand{\hvx}{\hat{\bm{x}}}
\newcommand{\hmX}{\hat{\bm{X}}}

\DeclareMathOperator{\EE}{\mathbb{E}}
\DeclareMathOperator{\PP}{\mathbb{P}}
\DeclareMathOperator{\Var}{\mathrm{Var}}
\DeclareMathOperator*{\argmax}{arg\,max}
\DeclareMathOperator*{\argmin}{arg\,min}
\DeclareMathOperator{\1}{\mathds{1}}
\DeclareMathOperator{\vol}{vol}

\DeclareMathOperator{\diag}{diag}
\DeclareMathOperator{\diam}{diam}
\DeclareMathOperator{\Beta}{Beta}

\newcommand{\RR}{\mathbb{R}}
\newcommand{\set}[1]{\left\{#1\right\}}
\newcommand{\transpose}{{\mkern-1.0mu\mathsf{T}}}
\newcommand{\dif}{\mathop{}\!\mathrm{d}}
\newcommand{\closure}[1]{\overline{#1}}
\newcommand{\card}[1]{\# #1}
\newcommand{\bigO}{O}
\newcommand{\littleo}{o}
\newcommand{\bigOmega}{\Omega}
\newcommand{\bigTheta}{\Theta}

\newcommand{\ptogiven}[1]{\overset{\PP[\cdot|#1]}{\longrightarrow}}
\newcommand{\pto}{\overset{\PP}{\longrightarrow}}
\newcommand{\asto}{\overset{\mathrm{a.s.}}{\longrightarrow}}
\newcommand{\given}{\;\middle|\;}
\newcommand{\unitball}{\mathcal{B}}

\newcommand{\ellmax}{\ell_{\max}}
\newcommand{\hellmax}{\hell_{\max}}
\newcommand{\filldist}{h}
\newcommand{\Nacq}{N_{\mathrm{acq}}}
\newcommand{\CumReg}{\mathscr{R}}
\newcommand{\SimReg}{\mathscr{S}}

\newcommand{\proofparagraph}[1]{\par\medskip\noindent{\bfseries\textsf{#1}}\quad}

\title{Bayesian Optimization by Kernel Regression and Density-based Exploration}
\author{
    Tansheng Zhu\textsuperscript{1}  \and 
    Hongyu Zhou\textsuperscript{2} \and 
    Ke Jin\textsuperscript{2} \and 
    Xusheng Xu\textsuperscript{3} \and 
    Qiufan Yuan\textsuperscript{3} \and 
    Lijie Ji\textsuperscript{4,5}\thanks{Corresponding author: \email{lijieji@shu.edu.cn}}
}
\institute{
    \textsuperscript{1}Zhiyuan College, Shanghai Jiao Tong University, Shanghai 200240, P. R. China. \\
    \textsuperscript{2}School of Mathematical Sciences, Shanghai Jiao Tong University, Shanghai 200240, P. R. China. \\
    \textsuperscript{3}Shanghai Institute of Aerospace Systems Engineering, Shanghai 201109, P. R. China. \\
    \textsuperscript{4}Department of Mathematics, Shanghai University, Shanghai 200444, P. R. China. \\
    \textsuperscript{5}Newtouch Center for Mathematics of Shanghai University, Shanghai University, Shanghai 200444, P. R. China.
}
\version{}
\date{}

\begin{document}

\maketitle

\begin{abstract}
Bayesian optimization is highly effective for optimizing expensive-to-evaluate black-box functions, but it faces significant computational challenges due to the cubic per-iteration cost of Gaussian processes, which results in a total time complexity that is quartic with respect to the number of iterations.
To address this limitation, we propose
a novel algorithm, Bayesian optimization by kernel regression and density-based exploration (BOKE).
BOKE uses kernel regression for efficient function approximation, kernel density for exploration, and integrates them into the confidence bound criteria to guide the optimization process, thus reducing computational costs to quadratic.
Our theoretical analysis rigorously establishes the global convergence of BOKE under noisy evaluations. 
Through extensive numerical experiments on both synthetic and real-world optimization tasks, we demonstrate that BOKE not only performs competitively compared to Gaussian process-based methods and several other baseline methods but also exhibits superior computational efficiency. 
These results highlight BOKE's effectiveness in resource-constrained environments, providing a practical approach for optimization problems in engineering applications.

\keywords{Bayesian optimization, kernel regression, kernel density estimation, global convergence.}
\end{abstract}

\section{Introduction}
\label{sec:introduction}

Bayesian optimization (BO) is a sequential optimization algorithm known for its efficiency in minimizing the number of iterations required to optimize objective functions \cite{bergstra2011algorithms,hutter2011sequential,garnett2023bayesian}. This is particularly useful for problems where the objective function lacks an explicit expression, has no derivative information, is non-convex, or is expensive to evaluate \cite{brochu2010tutorial,mockus1978application,mockus2002bayesian}. Notably, BO remains effective even when only noisy observations of the objective function's sampling points are available.
Given its robustness in such scenarios, BO has been widely applied to hyperparameter tuning in machine learning models \cite{snoek2012practical,li2018hyperband} and experimental design \cite{chaloner1995bayesian,press2009bandit,pourmohamad2021bayesian}, and Monte Carlo planning \cite{mern2021bayesian,kim2020monte}.

The sampling criterion of the BO algorithm for selecting the next sampling point is determined by an acquisition function, such as probability of improvement (PI) \cite{kushner1964new}, expected improvement (EI) \cite{mockus1978application,jones1998efficient}, or Gaussian process upper confidence bound (GP-UCB) \cite{srinivas2010gaussian}. The latter two functions balance exploitation and exploration: exploitation focuses on maximizing the mean of the surrogate model for the objective function, while exploration minimizes the surrogate model's variance to encourage sampling in uncertain regions.
However, evaluating these acquisition functions traditionally relies on a Gaussian process (GP) surrogate, which requires inverting the kernel covariance matrix to compute the posterior predictive distribution \cite{rasmussen2006gaussian}.
This incurs a per-iteration cost of $\bigO(t^3)$ for $t$ observations.
Scaling to $T$ iterations, this yields an overall quartic complexity, rapidly becoming prohibitive when the algorithmic overhead dominates the time required to evaluate the objective \cite{hase2018phoenics}.

A direct approach to mitigate this cubic bottleneck is to accelerate the underlying GP computations.
From an algorithmic perspective, a prominent line of research relies on sparse and variational approximations to the GP posterior \cite{liu2020gaussian,mutny2018efficient,titsias2009variational,hensman2015mcmc,tautvaivsas2024scalable,jimenez2023scalable}.
Complementing these algorithmic developments, significant efforts have focused on leveraging hardware acceleration and distributed computing.
For instance, exact and low-rank GPs have been scaled to massive datasets via multi-GPU parallelization and iterative solvers \cite{wang2019exact,zhao2025nugpr}.
Within the broader BO pipeline, this computational scalability is often realized at the system level through GPU-native frameworks \cite{balandat2020botorch} and asynchronous, decentralized architectures \cite{egele2023asynchronous}.
Beyond sheer computational constraints, successfully scaling GP-BO to real-world applications also requires overcoming the exponential complexity of high-dimensional search spaces \cite{eriksson2019scalable,de2021greed,hvarfner2024vanilla} and modeling the non-stationarity and heteroskedasticity typical of practical hyperparameter tuning landscapes \cite{cowen2022hebo}.

Instead of modifying the GP, an alternative paradigm bypasses the surrogate model entirely to directly estimate the acquisition function.
The tree-structured Parzen estimator (TPE) reformulates the EI function as a density ratio estimation problem under specific conditions, avoiding the need for direct computation of the GP mean. This method performs well as the number of observations increases \cite{bergstra2011algorithms}.
Bayesian optimization by density-ratio estimation (BORE) transforms density-ratio estimation into a weighted classification problem, enabling estimation through likelihood-free inference \cite{tiao2021bore}. Building on this approach, likelihood-free Bayesian optimization (LFBO) generalizes the framework to accommodate any acquisition function expressed in the form of expected utility \cite{song2022general,gaudrie2024empirical}.

A third approach retains the surrogate-based framework but replaces the GP with inherently more scalable models.
For example, sequential model-based algorithm configuration (SMAC) employs a random forest regressor and uses the standard deviation among trees to estimate uncertainty \cite{hutter2011sequential,hutter2012parallel}. Phoenics leverages Bayesian neural networks to model the sample density function and proposes an acquisition function based on kernel densities \cite{hase2018phoenics}.
Global optimization with learning and interpolation surrogates (GLIS) and its extensions adopt inverse distance weighting and radial basis functions to construct the exploration term, achieving competitive performance with reduced computational overhead \cite{bemporad2020global,previtali2023glisp}.
Pseudo-Bayesian optimization (PseudoBO) introduces an axiomatic framework that combines a surrogate predictor, an uncertainty quantifier, and an acquisition function, ensuring algorithmic convergence \cite{chen2023pseudo}.

In this paper, we propose a novel Bayesian optimization algorithm, Bayesian optimization by kernel regression and density-based exploration (BOKE). This algorithm comprises three main components: (i) a kernel regression surrogate for efficient function approximation; (ii) a kernel density–based exploration strategy; and (iii) an improved kernel regression upper confidence bound (IKR-UCB) acquisition function that balances exploitation and exploration. The IKR-UCB function can be viewed as an extension of the classical upper confidence bound (UCB) \cite{auer2002finite} by replacing the mean and counting terms with their kernel regression and kernel density counterparts.
We compare the vanishing-bandwidth limits of IKR-UCB's components with those of GP-UCB, and demonstrate the effectiveness of our density-based exploration through space-filling design experiments.
BOKE attains convergence guarantees comparable to GP-based methods while substantially reducing computational cost.
From a theoretical perspective, we derive prediction error bounds for kernel regression formulated in terms of kernel density. 
Crucially, the analytical structure of these bounds naturally motivates the exact form of the IKR-UCB acquisition function.
We further establish the algorithmic consistency of BOKE in noisy settings, showing that, with probability one, the query points generated by BOKE are dense in the decision set. Building on this consistency, we prove global convergence of the simple regret. Additionally, we derive upper bounds on the cumulative regret for both general and finite decision sets, which in turn inform a principled choice of the trade-off parameter.
To mitigate potential over-exploration, we develop the BOKE+ variant, which enhances exploitation via an $\epsilon$-greedy scheme \cite{de2021greed}.
Finally, we validate the superior performance and computational efficiency of BOKE and BOKE+ across a range of synthetic and real-world optimization tasks.

The paper is organized as follows. In \cref{sec:background}, we introduce notations, problem formulation, and provide a brief overview of Bayesian optimization. \cref{sec:methodology}  details the proposed BOKE algorithm, covering its kernel regression surrogate, density-based exploration strategy, extensions, and computational complexity.
Theoretical results, including the prediction error bounds of kernel regression, the consistency and regret analysis of the proposed algorithms, are discussed in \cref{sec:convergence analysis}.
\cref{sec:numerics} presents numerical results demonstrating the superior optimization performance and computational efficiency of the proposed algorithms. Finally, \cref{sec:conclusions} concludes the paper with a discussion of future research directions.

\section{Background}
\label{sec:background}

Throughout this paper, we use $\bigO(\cdot)$, $\littleo(\cdot)$, $\bigOmega(\cdot)$, and $\bigTheta(\cdot)$ to denote standard asymptotic complexities.
The indicator function is denoted by $\1_{\set{\cdot}}$.
For a set $S$, its cardinality is denoted by $\#S$.
We denote by $\mathbb{R}^d$ the $d$-dimensional Euclidean space, $\| \cdot \|_p$ to denote the $p$-norm for $1 \le p \le \infty$ (with $p = 2$ if unspecified), and $B(\vx, r)$ represents the $d$-dimensional ball centered at $\vx$ with radius $r$.
The distance between a point $\vx \in \RR^d$ and a set $S \subset \RR^d$ is denoted by $d(\vx, S) = \inf_{\vx' \in S} \|\vx - \vx'\|$.
For any positive integer $t$, we denote by $\mX_t = \set{\vx_i}_{i=1}^{t} \subset \fX$ a sequence of observed points, $\vy_t = \set{y_i}_{i=1}^{t} \subset \mathbb{R}$ a sequence of real values, and $\mD_t = \set{(\vx_i, y_i)}_{i=1}^{t}$ the entire dataset.

\subsection{Problem formulation}
\label{sec:problem formulation}

Let $f$ be a black-box function defined over a decision set $\fX \subset \RR^d$. The optimization problem is formulated as follows:
\begin{equation}
\label{eqn:max objective}
    \vx^* \in \argmax_{\vx \in \fX} f(\vx),
\end{equation}
where $f: \fX \to \RR$ may be a non-convex function that lacks an explicit expression and can only be evaluated through noisy samples. 
In particular, the observations are of the form $y = f(\vx) + \varepsilon$, where $\varepsilon$ denotes the additive sampling noise.

Given a time horizon $T$, the performance of a sequential optimization algorithm can be evaluated by the following metrics:
\begin{itemize}
    \item The \textit{cumulative regret} is defined as
    \begin{equation}
    \label{eqn:cumulative regret}
       \CumReg_T \coloneqq \sum_{t=1}^{T} \left[f(\vx^*) - f(\vx_t)\right] = \sum_{t=1}^{T} r_t.
    \end{equation}
    where $r_t \coloneqq f(\vx^*) - f(\vx_t)$ is the \textit{instantaneous regret} at the $t$-th iteration. The procedure that selects the sequence of evaluation points $\set{\vx_t}_{1 \le t \le T}$ is called the allocation strategy \cite{bubeck2011pure}.

    \item The \textit{simple regret} is
    \begin{equation}\label{eqn:simple regret}
        \SimReg_T \coloneqq f(\vx^*) - f(\hvx^*_T),
    \end{equation}
     where $\hat{\vx}^*_T$ denotes the candidate optimizer after $T$ function evaluations, determined by a recommendation strategy.
    In empirical benchmark evaluations, this strategy is typically evaluated via an oracle that identifies the best queried point $\hvx^*_T \in \argmax_{\vx \in \mX_T} f(\vx)$, yielding $\SimReg_T = \min_{1 \le t \le T} r_t$. 
    While this standard definition suffices for empirical evaluations, our theoretical analysis under noisy observations adopts a more sophisticated recommendation strategy based on surrogate estimators \cite{bubeck2011pure,vakili2021optimal}, which we formally define in \cref{sec:regret analysis}.
\end{itemize}

\subsection{Bayesian optimization}

Bayesian optimization (BO) \cite{pourmohamad2021bayesian,garnett2023bayesian} is a sequential design strategy for efficiently solving \cref{eqn:max objective}.
BO uses an acquisition function to balance exploitation and exploration.
The acquisition function guides the selection of query points to improve the surrogate model and locate the global optimum. The pseudocode is presented in \cref{alg:BO}.

\begin{algorithm}[!htb]
    \caption{Bayesian optimization}
    \label{alg:BO}
    \KwIn{initial dataset $\mD_{T_0} = \set{(\vx_i, y_i)}_{i=1}^{T_0}$, budget $T$.}
    \For{$t = T_0, \dots, T-1$}{
        Fit a surrogate model on the dataset $\mD_t$. \;
        Select $\vx_{t+1}$ by maximizing the acquisition function over $\fX$. \;
        Evaluate the objective function $f$ at $\vx_{t+1}$ to obtain $y_{t+1}$. \;
        Augment dataset $\mD_{t+1} \gets \mD_{t} \cup \set{(\vx_{t+1}, y_{t+1})}$. \;
    }
\end{algorithm}

This algorithm mainly consists of two primary steps:
First, a surrogate model is constructed, and the corresponding acquisition function is generated.
Second, a new point is obtained by maximizing the acquisition function, followed by updating the surrogate model.
Starting with the initial design, these steps are iteratively repeated to efficiently allocate the remaining evaluation budget and minimize the cumulative regret or simple regret.

\section{The BOKE method}
\label{sec:methodology}

This section constructs the BOKE algorithm in three steps:
first, a kernel regression surrogate model; second, a density-based exploration strategy built via kernel density estimation; and third, an IKR-UCB acquisition function based on the preceding two components. 
The remainder of the section details these components, presents the complete BOKE algorithm (\cref{alg:BOKE}) and its variant BOKE+ (\cref{alg:BOKE+}), and analyzes their computational complexity relative to baseline BO algorithms.

\subsection{Kernel regression surrogate function}
\label{sec:kernel regression surrogates}
Kernel regression (KR), also known as the Nadaraya-Watson estimator \cite{nadaraya1964estimating, watson1964smooth}, is a widely recognized method in nonparametric regression. The expected value at a target point $\vx$ is estimated as a weighted average of the observed dataset values, with weights determined by a kernel function and the distances between the data points and the target point.
Let $\mD_t = \set{(\vx_i, y_i)}_{i=1}^{t} \subset \fX \times \RR$ be the dataset, the predicted expected value for any $\vx \in \fX$ is computed as:
\begin{equation}
    \label{eqn:KR}
    m_t(\vx) = m(\vx; \mD_t)
    \coloneqq \frac{\sum_{i=1}^{t} k(\vx, \vx_i) y_i}{\sum_{i=1}^{t} k(\vx, \vx_i)},
\end{equation}
where $k(\cdot,\cdot)$ is the kernel function. 
The kernel is typically assumed to be stationary, meaning that there exists a function $\Psi: \RR^d \to [0, \infty)$ such that $k(\vx, \vx') = \Psi\left( (\vx - \vx') / \ell \right)$ for all $\vx, \vx' \in \fX$, where $\ell > 0$ is the bandwidth that controls the length-scale of kernel.
Examples of kernels include (see, e.g., \cite[Section 6.2.3]{scott2015multivariate})
\[\renewcommand{\arraystretch}{1.2}
    \begin{array}{ll}
        \text{(Gaussian)} & \quad \Psi(\vx) = \exp\left(- \| \vx \|^2 / 2 \right), \\
        \text{(Triangular)} & \quad \Psi(\vx) = \left( 1 - \| \vx \| \right) \1_{\set{\| \vx \| \le 1}} , \\
        \text{(Epanechnikov)} & \quad \Psi(\vx) = \left( 1 - \| \vx \|^2 \right) \1_{\set{\| \vx \| \le 1}}, \\
        \text{(Quartic)} & \quad \Psi(\vx) = \left( 1 - \| \vx \|^2 \right)^2 \1_{\set{\| \vx \| \le 1}}, \\
    \end{array}
\]
where the coefficients are scaled such that $\Psi(\bm{0}) = 1$.

KR has a rich history in statistical learning and sequential decision-making. For instance, it has been used within the Monte Carlo tree search (MCTS), a classical reinforcement learning framework, for information sharing between tree nodes to address execution uncertainty \cite{yee2016monte}. It also serves as a foundational estimator in continuum-armed bandit problems \cite{agrawal1995continuum}. Recent work has combined KR with a randomized prior to design acquisition functions \cite{chen2023pseudo}.

\begin{figure}[!htb]
    \centering
    \includegraphics[width=.98\textwidth]{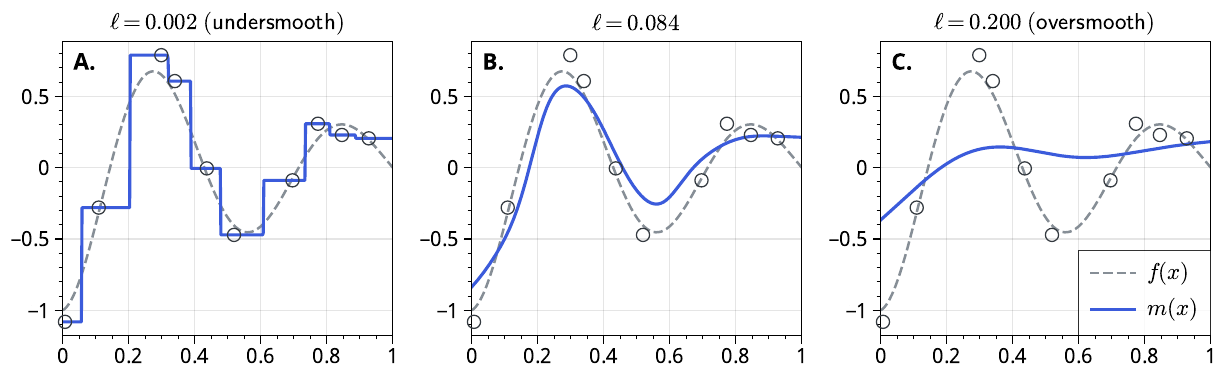}
    \caption{
        Kernel regression with Gaussian kernels at various bandwidths. 
        The objective function \cref{eqn:test} is sampled $10$ times, with each sample evaluated with additive noise $\varepsilon \sim \fN(0, 0.01)$.
    }
    \label{fig:kr}
\end{figure}

In this paper, we use KR as the surrogate model in the proposed BOKE algorithm.
In \cref{fig:kr}, we illustrate KR with a Gaussian kernel on the one-dimensional target function defined in \cref{eqn:test}:
\begin{equation}
    \label{eqn:test}
    f(x) = - e^{-1.4 x} \cos(3.5 \pi x), \quad 0 \le x \le 1.
\end{equation}
Panel B of \cref{fig:kr} shows the estimate with the bandwidth chosen by least-squares cross-validation \cite{clark1977non}.
In panel C, a large bandwidth produces an oversmoothed estimate; conversely, in panel A, a small bandwidth yields an undersmoothed, step-like estimate that approaches the nearest neighbor estimator, as formalized in \cref{pro:degenerate KR}.
This limiting behavior illustrates KR's robustness to bandwidth misspecification.
By contrast, the GP predictor converges, as the bandwidth $\ell \to 0^+$, to a weighted Dirac comb (see \cref{eqn:degenerate GPR} in \cref{app:vanishing bandwidth limits}). 
The proof generalizes \cite[Proposition 3]{kang2016kernel} by allowing repeated observations.
Note that both KR and GP estimates are linear smoothers, meaning that they can be written as $\hf(\vx) = \sum_{i=1}^{t} w_i(\vx) y_i$ with weight functions $w_i(\cdot)$ \cite{rasmussen2006gaussian}. 
Our results show that, as the bandwidth vanishes, these weight functions can differ substantially.

\begin{proposition}
    \label{pro:degenerate KR}
    Let $\Psi$ be the power exponential kernel defined by $\Psi(\vx) = \exp\left( - c \| \vx \|_p^p  \right)$ for all $\vx \in \RR^d$, with $c, p > 0$. 
    For any $\mD_t = \set{(\vx_i, y_i)}_{i=1}^{t}$ and $\vx \in \fX$,
    \begin{equation}\label{eqn:degenerate KR}
        \lim_{\ell \to 0^+} m_t(\vx)
        = \frac{\sum_{i \in \fI_t(\vx)} y_i}{\card{\fI_t(\vx)}}
        = \frac{\sum_{i=1}^{t} y_i \1_{\set{i \in \fI_t(\vx)}}}{\sum_{i=1}^{t} \1_{\set{i \in \fI_t(\vx)}}},
    \end{equation}
    where $\fI_t(\vx) \coloneqq \set{i : \vx_i \in \argmin_{\vx' \in \mX_t} \| \vx - \vx' \|_p}$ is the index set of the nearest neighbor(s) of $\vx$ with respect to the $p$-norm.
\end{proposition}

\begin{proof}
    Let $\vz \in \argmin_{\vx' \in \mX_t} \| \vx - \vx' \|_p$. 
    Dividing numerator and denominator of $m_t(\vx)$ by $\exp\left( - c \ell^{-p} \| \vx - \vz \|_p^p  \right)$ gives
    \[
        m_t(\vx) 
        = \frac{\sum_{i=1}^{t} y_i \exp\left( - c \ell^{-p} \left( \| \vx - \vx_i \|_p^p - \| \vx - \vz \|_p^p \right) \right)}{\sum_{i=1}^{t} \exp\left( - c \ell^{-p} \left( \| \vx - \vx_i \|_p^p - \| \vx - \vz \|_p^p \right) \right)}.
    \]
    For each $i \notin \fI_t(\vx)$, the difference $\|\vx - \vx_i\|_p^p - \|\vx - \vz\|_p^p$ is strictly positive, so the corresponding exponent tends to $-\infty$ as $\ell \to 0^+$ and those terms vanish. 
    Hence only indices in $\fI_t(\vx)$ contribute in the limit, which yields \cref{eqn:degenerate KR}.
\end{proof}

\subsection{Density-based exploration}
\label{sec:density-based exploration}

Kernel density estimation (KDE) \cite{rosenblatt1956remarks}, also known as Parzen's window \cite{parzen1962estimation}, is a widely-used nonparametric method for estimating probability density function. In the context of reinforcement learning, KDE is employed to design count-based exploration functions that encourage exploration of the action space \cite{bellemare2016unifying}. In Bayesian optimization, the authors of \cite{hase2018phoenics} utilized Bayesian kernel density estimation to construct an acquisition function.
Let $\mX_t = \set{\vx_1, \dots, \vx_t} \subset \fX$ be a sequence of query points, the KDE is formally expressed as:
\begin{equation}
    \hat{p}(\vx) = \frac{1}{Z} \sum_{i=1}^{t} k(\vx, \vx_i),
\end{equation}
where $Z \propto t \ell^d$ is the normalization factor such that $\int_{\fX} \hat{p}(\vx) \dif \vx = 1$, and $k$ and $\ell$ have the same meaning as in kernel regression.
From the perspective of uncertainty quantification, the component $\sum_{i=1}^{t} k(\vx, \vx_i)$ plays a central role. When no ambiguity arises, we omit the normalization factor and denote the unnormalized KDE, which scales with $t$, by
\begin{equation}
    \label{eqn:KDE}
    W_t(\vx) =W(\vx; \mX_t) \coloneqq \sum_{i=1}^{t} k(\vx, \vx_i).
\end{equation}

Note that this unnormalized density $W_t(\vx)$ exactly coincides with the denominator of the KR estimator \eqref{eqn:KR}.
Intuitively, a small value of $W_t(\vx)$ indicates higher uncertainty at $\vx$.
For illustration of KDE with different bandwidths, $20$ samples are independently drawn from the beta mixture model $p(x) = 0.7 \Beta(3, 8) + 0.3 \Beta(10, 2)$. 
The corresponding KDE is calculated and compared with the analytical beta mixture model, see \cref{fig:kde}.
The bandwidth in panel B, selected by maximum-likelihood cross-validation \cite{habbema1974stepwise,duin1976choice}, captures the two modes of the analytical model. 
In panel A, an excessively small bandwidth $\ell$ yields a highly localized density resembling a Dirac comb, failing to generalize beyond the immediate vicinity of the observations.
Generally, regions with higher density are likely to be sampled, and thus have lower uncertainty, whereas regions with lower density exhibit higher uncertainty and need more exploration.

\begin{figure}[!htb]
    \centering
    \includegraphics[width=.98\textwidth]{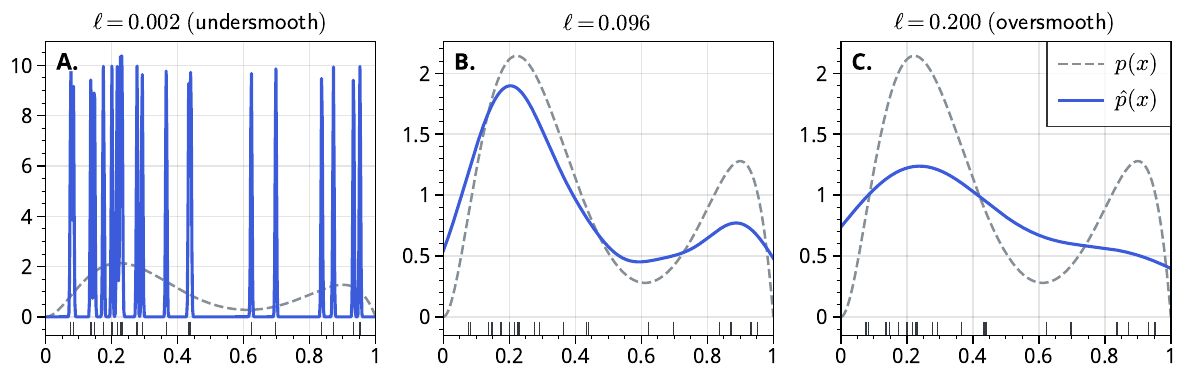}
    \caption{
        Kernel density estimation with Gaussian kernels at various bandwidths, where $\hat{p}(x) = (2 \pi)^{-1/2} (\ell t)^{-1} W_t(x)$. The $20$ samples are generated from $p(x) = 0.7 \Beta(3, 8) + 0.3 \Beta(10, 2)$.
    }
    \label{fig:kde}
\end{figure}

The \textit{fill distance}, also called the covering radius, is a measure of how well a set of points covers the decision set \cite{santner2003design}.
For a decision set $\fX$ and a design $\mX \subset \fX$ (a finite set of distinct points), the fill distance is defined as
\begin{equation}
    \label{eqn:fill distance}
    \filldist_{\fX, \mX} \coloneqq \sup_{\vx \in \fX} \min_{\vx_i \in \mX} \| \vx - \vx_i \|.
\end{equation}

Here, we propose the \textit{density-based exploration} (DE) strategy for space-filling design, which sequentially generates query points by greedily minimizing the KDE function:
\begin{equation}
    \label{eqn:min KDE}
    \vx_{t+1} \in \argmin_{\vx \in \fX} W_t(\vx).
\end{equation}
By comparison, the P-greedy algorithm \cite{de2005near}, a GP-based exploration strategy, selects query points by greedily maximizing the posterior variance at each iteration.
This produces a query sequence whose fill distance decays at a rate of $\bigTheta(t^{-\frac{1}{d}})$, matching the optimal asymptotic rate for space-filling designs \cite{wenzel2021novel}.
Rather than pursuing strict space-filling properties, alternative approaches such as EIGF \cite{lam2008sequential}, VIGF \cite{mohammadi2022cross}, and ES-LOO \cite{mohammadi2025sequential} employ response-dependent criteria to adaptively improve the global fit of GP models. In \cref{fig:fill distance}, we plot the fill distances of DE and P-greedy alongside traditional non-greedy baselines, including uniform sampling and Latin hypercube sampling (LHS) \cite{mckay1979comparison}.
As observed, the two greedy exploratory methods exhibit closely aligned decay curves that significantly outperform traditional baselines in moderate dimensions ($d \in \set{2, 5}$).
Furthermore, as dimensionality increases ($d \ge 10$), DE maintains its robust efficiency and even slightly outperforms P-greedy.
This demonstrates that KDE serves as a highly effective proxy for spatial uncertainty in space-filling designs.

\begin{figure}[!htb]
    \centering
    \includegraphics[width=.98\textwidth]{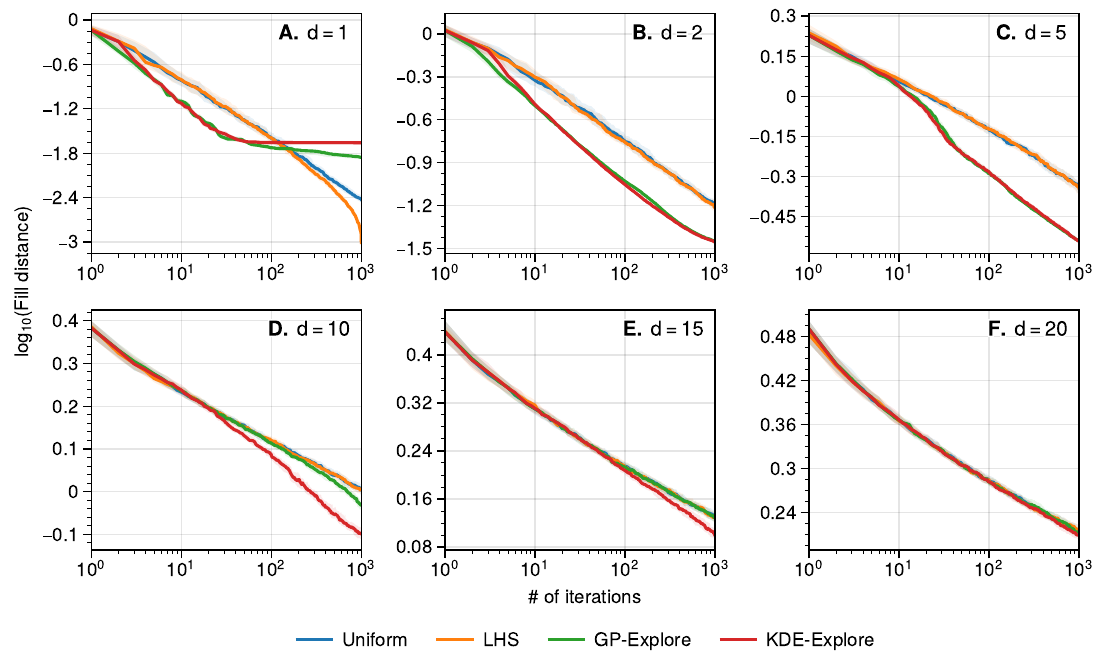}
    \caption{
        Fill distance and its distribution for different space-filling designs on $\fX = [0, 1]^d$ with $d \in \set{1, 2, 5, 10, 15, 20}$.
        In the legend, GP-Explore and KDE-Explore correspond to the P-greedy and DE strategies, respectively. Both methods use a Gaussian kernel with a fixed bandwidth $\ell = 0.1$.
        Each plot displays the median and interquartile range over $100$ runs.
        Both axes are on logarithmic scales.
    }
    \label{fig:fill distance}
\end{figure}

\subsection{Confidence bound criteria}
\label{sec:confidence bound criteria}

In multi-armed bandit problems, the \textit{upper confidence bound} (UCB) policy is a widely used allocation strategy to select the most optimistic arm at each step. 
Let $X \subset \fX$ denote the finite decision set of $M$ arms. 
The UCB1 policy \cite{auer2002finite} for the $M$-armed bandit is
\begin{equation}
    \label{eqn:max UCB1}
    \vx_{t+1} \in \argmax_{1 \le i \le M} \bar{y}_{t, i} + \sqrt{\frac{2 \ln t}{n_{t, i}}},
\end{equation}
where $\bar{y}_{t, i}$ is the empirical mean reward of arm $i$, and $n_{t, i}$ is the number of times arm $i$ has been pulled up to iteration $t$ (with the convention $1/0 = \infty$ for unpulled arms).
When extended to global optimization, confidence bound criteria are used to determine the next query point based on Gaussian process modeling, where the posterior variance of GP serves as the exploration term \cite{srinivas2010gaussian}. Additionally, kernel regression has been employed to share information between arms, thereby extending UCB1 to continuous spaces \cite{yee2016monte}.  For more details, see \cref{app:extended background}.

Motivated by UCB1, we define $\hsigma_t(\vx)$ as the density-based uncertainty function, which is monotonically decreasing in the kernel density:
\begin{equation}
    \label{eqn:DE}
    \hsigma_t(\vx) = \hsigma(\vx; \mX_t) \coloneqq (W(\vx; \mX_t) + \varrho)^{-\frac{1}{2}},
\end{equation}
where $\varrho > 0$ is a small constant added for numerical stability, and the exponent $-1/2$ is chosen to be analogous to the UCB1 scaling.
Hence, the DE strategy \eqref{eqn:min KDE} is equivalent to
\[
    \vx_{t+1} \in \argmax_{\vx \in \fX} \hsigma_t(\vx).
\]
The vanishing-bandwidth limit of $\hsigma_t(\vx)$ follows directly from its definition, namely $\lim_{\ell \to 0^+} \hsigma_t^{-2}(\vx) = \varrho + \Psi(\bm{0}) \sum_{i=1}^{t} \1_{\set{\vx = \vx_i}}$.
Similarly, the GP posterior variance also converges to the inverse of a Dirac comb (see \cref{eqn:degenerate GP variance} in \cref{app:vanishing bandwidth limits}).
Thus, both $\hsigma_t(\vx)$ and the GP posterior standard deviation induce similar exploratory behavior, and the parameter $\varrho$ plays the same role as the additive noise variance in GP (see \cref{rmk:degenerate GP vs KDE} in \cref{app:vanishing bandwidth limits}).

Using KR for exploitation and DE for exploration, we propose a novel acquisition function called the \textit{improved kernel regression upper confidence bound} (IKR-UCB).
The next query point is selected by maximizing this acquisition function, which is defined as follows:
\begin{equation}
    \label{eqn:IKR-UCB}
    a_t(\vx) = a(\vx; \mD_t) \coloneqq m(\vx; \mD_t) + \beta_t \hsigma(\vx; \mX_t),
\end{equation}
where $t$ denotes the number of queried points, and $\beta_t > 0$ is a tuning parameter that varies with $t$.

\begin{figure}[!htb]
    \centering
    \includegraphics[width=.9\textwidth]{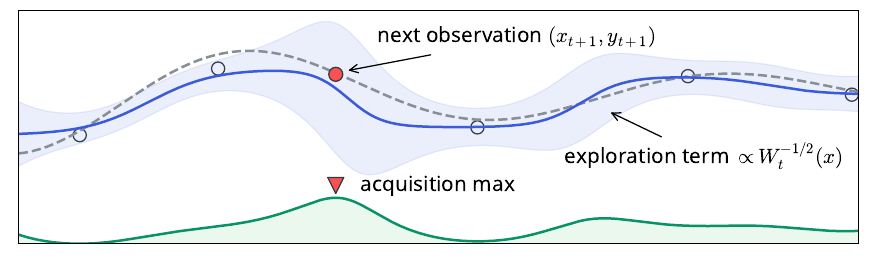}
    \caption{
        Illustration of the IKR-UCB acquisition function. 
        The gray dashed line represents the analytical function, while the blue curve shows the KR approximation of the objective function. 
        The upper light blue shaded area indicates the uncertainty quantification based on kernel density estimation. 
        The lower shaded plot visualizes the acquisition function.
    }
    \label{fig:acquisition}
\end{figure}

In \cref{fig:acquisition}, we conduct a simple test for optimizing the scalar function defined in \cref{eqn:test}, using the same settings described therein. 
The sixth query point is chosen by maximizing the IKR-UCB acquisition function. 
As shown, the acquisition function attains high values in regions where the kernel regression predicts a high surrogate mean (exploitation) and where the predictive uncertainty is substantial (exploration).

With the proposed IKR-UCB, the global optimization algorithm can be directly formulated. 
Starting with an initial design, the next query point is determined by maximizing the IKR-UCB acquisition function in \eqref{eqn:IKR-UCB}. 
This iterative process continues until a predefined iteration tolerance $T$ is reached. 
The complete BOKE algorithm is detailed in \cref{alg:BOKE}.

\begin{algorithm}[!htb]
    \caption{BOKE algorithm}
    \label{alg:BOKE}
    \KwIn{initial dataset $\mD_{T_0}$, budget $T$, kernel $k$, parameters $\set{\beta_i}_{i=T_0}^{T-1}$, $\varrho > 0$.}
    \For{$t = T_0, \dots, T - 1$}{
        Compute $m_t$ and $\hsigma_t$ using \cref{eqn:KR,eqn:DE}. \;
        Select $\vx_{t+1}$ by solving
        \begin{equation}
            \label{eqn:max IKR-UCB}
              \vx_{t+1} \in \argmax_{\vx \in \fX}  a_t(\vx).
        \end{equation} \;
        \vspace*{-1em}
        Evaluate $f$ at $\vx_{t+1}$ to obtain $y_{t+1}$ and update $\mD_{t+1} \gets \mD_{t} \cup \set{(\vx_{t+1}, y_{t+1})}$. \;
    }
\end{algorithm}

While UCB-style acquisition functions theoretically balance exploration and exploitation, their reliance on an increasing $\beta_t$ \cite{auer2002finite,srinivas2010gaussian} often leads to over-exploration in practice \cite{de2021greed}.
This empirical drawback is exacerbated in our setting: because the inherent smoothing bias of KR limits its ability to sharply resolve local optima compared to standard GPs, IKR-UCB suffers from a diminished exploitative capacity (see \cref{sec:prediction error bound}).
To improve its performance, we consider the pure-exploitation acquisition function defined in \eqref{eqn:max KR}:
\begin{equation}
    \label{eqn:max KR}
    \vx_{t+1} \in \argmax_{\vx \in \fX} m_t(\vx).
\end{equation}
By executing IKR-UCB with probability $q \in (0, 1]$ and \cref{eqn:max KR} with probability $1 - q$, we formulate a modified algorithm denoted as BOKE+ (summarized in \cref{alg:BOKE+}). This $\epsilon$-greedy scheme is consistent with recent BO literature \cite{de2021greed}, which shows that occasional pure exploitation can accelerate empirical convergence.
The optimization results of BOKE and BOKE+ are tested on the toy problem defined in \cref{eqn:test}, with the corresponding results shown in \cref{fig:iteration}. Both methods begin with the same initial set of five points. Notably, BOKE+ focuses sampling on regions with high objective values, while BOKE maintains a more balanced trade-off between exploration and exploitation.

\begin{algorithm}[!htb]
    \caption{BOKE+ algorithm}
    \label{alg:BOKE+}
    \KwIn{initial dataset $\mD_{T_0}$, budget $T$, kernel $k$, parameters $\set{\beta_i}_{i=T_0}^{T-1}$, $\varrho > 0$, and $q \in (0, 1]$.}
    \For{$t = T_0, \dots, T-1$}{
        Compute $m_t$ and $\hsigma_t$ by \cref{eqn:KR,eqn:DE}. \;
        Sample $U \sim \mathrm{Uniform}(0, 1)$. \;
        Select $\vx_{t+1}$ by solving \cref{eqn:max IKR-UCB} if $U < q$, and by solving \cref{eqn:max KR} otherwise. \;
        Evaluate $f$ at $\vx_{t+1}$ to obtain $y_{t+1}$ and update $\mD_{t+1} \gets \mD_{t} \cup \set{(\vx_{t+1}, y_{t+1})}$.
    }
\end{algorithm}

\begin{figure}[!htb]
    \centering    
    \includegraphics[width=.98\textwidth]{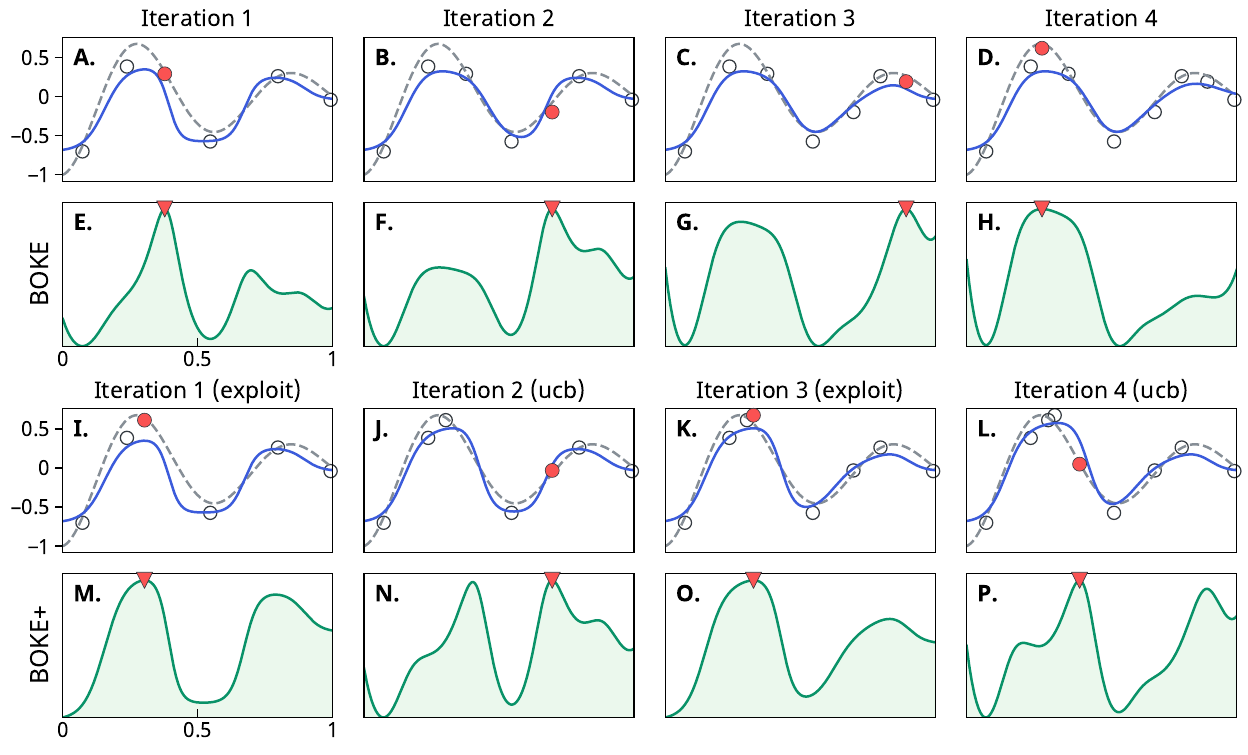}
    \caption{Comparison of BOKE (top two rows) and BOKE+ (bottom two rows) on the one-dimensional problem defined in \cref{eqn:test}. The gray dashed line represents the analytical function, while the blue curve shows the KR approximation of the objective function. The shaded plot visualizes the acquisition function. The red triangles indicate the maximizers of the acquisition functions, which determine the subsequent query points (red circles) evaluated on the noisy objective.
    }
    \label{fig:iteration}
\end{figure}

\subsection{Computational complexity}
\label{sec:computational complexity}

This section analyzes the computational time complexity of the proposed BOKE framework.
Since the modified BOKE+ algorithm shares the same asymptotic complexity, our analysis focuses on the base BOKE procedure.
We first compare BOKE against the UCB-based baselines GP-UCB and KR-UCB (see \cref{app:extended background}) before discussing scalable GP-BO variants.
At the $t$-th iteration, the computational cost of BO comprises an iteration-level \textit{fitting} phase and a per-query \textit{prediction} phase. 
The fitting phase encompasses updating the surrogate model and performing any necessary preprocessing steps, whereas the prediction phase evaluates the acquisition function at a candidate query point.

Focusing first on exact methods, GP-UCB requires computing a Cholesky decomposition for fitting (an $\bigO(t^3)$ cost) and solving a triangular system for prediction (an $\bigO(t^2)$ cost per query; see \cref{app:GP-UCB}).
KR-UCB incurs an $\bigO(t^2)$ fitting cost to identify the historically most optimistic arm (see \cref{app:KR-UCB}) and requires $\bigO(t)$ operations per query for prediction.
In contrast, BOKE sidesteps explicit model training, yielding an $\bigO(1)$ fitting complexity, while its prediction cost remains strictly $\bigO(t)$ per query due to the explicit weighted averaging over historical observations.

To bypass the $\bigO(t^3)$ bottleneck, scalable GP-BO variants employ either global (e.g., inducing points \cite{titsias2009variational,hensman2015mcmc}, Fourier features \cite{mutny2018efficient}, Vecchia approximations \cite{jimenez2023scalable}) or local (e.g., aggregating local experts \cite{tautvaivsas2024scalable}) approximations.
Assuming a bounded number of inducing points or a bounded local expert size, these paradigms reduce the fitting complexity to $\bigO(t)$, while the prediction complexity per query is reduced to either $\bigO(1)$ or $\bigO(t)$, depending on the specific approximation scheme.
However, these approximations incur additional constant-factor overheads and rely on structural approximations to the GP posterior, whether through global distillation or through partitioning and aggregation, that can compromise the calibration of predictive uncertainty or introduce other limitations in model capability \cite{liu2020gaussian}.
BOKE avoids these structural approximations entirely. 
Furthermore, due to the spatial locality of the kernel, BOKE can be directly accelerated using tree-based data structures, as discussed in \cref{sec:conclusions}.

Let $T$ denote the total iteration budget and $\Nacq$ the number of acquisition function evaluations per iteration.
Summing over $T$ iterations, the overall time complexities for exact GP-UCB and its global and local approximate variants are $\bigO(T^4 + \Nacq T^3)$, $\bigO(T^2 + \Nacq T)$, and $\bigO(\Nacq T^2)$, respectively.
The overall time complexity of KR-UCB is $\bigO(T^3 + \Nacq T^2)$.
BOKE achieves $\bigO(\Nacq T^2)$, scaling significantly better than the exact baselines while matching the efficiency of sparse GP-UCB variants.
The computational time complexities of all aforementioned methods are summarized in \cref{tab:computational complexity}.

\begin{table}[!htb]
    \caption{
        Computational time complexity of UCB-based BO algorithms.
        Here, $t$ denotes the current iteration, $T$ the total budget, and $\Nacq$ the number of candidate evaluations per iteration.
        For the approximate GP-UCB variants, we assume a bounded number of inducing points or a bounded local expert size.
    }
    \label{tab:computational complexity}
    \centering
    \begin{tabular}{lcccc}
        \toprule
        \textbf{Algorithm} & \textbf{Fitting} & \textbf{Prediction} & \textbf{Per Iteration} & \textbf{Overall} \\
        \midrule
        GP-UCB & $\bigO(t^3)$ & $\bigO(t^2)$ & $\bigO(t^3 + \Nacq t^2)$ & $\bigO(T^4 + \Nacq T^3)$ \\
        GP-UCB (global approx.) & $\bigO(t)$ & $\bigO(1)$ & $\bigO(t + \Nacq)$ & $\bigO(T^2 + \Nacq T)$ \\
        GP-UCB (local approx.) & $\bigO(t)$ & $\bigO(t)$ & $\bigO(\Nacq t)$ & $\bigO(\Nacq T^2)$ \\
        KR-UCB & $\bigO(t^2)$ & $\bigO(t)$ & $\bigO(t^2 + \Nacq t)$ & $\bigO(T^3 + \Nacq T^2)$ \\
        BOKE (ours) & $\bigO(1)$ & $\bigO(t)$ & $\bigO(\Nacq t)$ & $\bigO(\Nacq T^2)$ \\
        \bottomrule
    \end{tabular}
\end{table}

\section{Convergence analysis}
\label{sec:convergence analysis}

To obtain our main results, we impose the following assumptions.
\begin{assumption}
    \label{asm:domain}
    The decision set $\fX \subset \RR^d$ is a compact, convex set with non-empty interior.
\end{assumption}

\begin{assumption}
    \label{asm:function}
    The objective function $f : \fX \to \RR$ is continuous.
\end{assumption}

\begin{assumption}
    \label{asm:kernel}
    The kernel takes the form $k(\vx, \vx') = \Psi\left( (\vx - \vx') / \ell \right)$ for all $\vx, \vx' \in \fX$, with bandwidth $\ell > 0$.
    Here, $\Psi : \RR^d \to [0, 1]$ satisfies $\Psi(\bm{0}) = 1$ and is compactly supported within a centered ball of radius $R_{\Psi} > 0$.
    Moreover, there exist constants $L_{\Psi}$, $\theta > 0$ such that $|\Psi(\vx) - \Psi(\vx')| \le L_{\Psi} \| \vx - \vx' \|^{\theta}$ for all $\vx, \vx' \in \RR^d$.
\end{assumption}

\begin{assumption}
    \label{asm:noise}
    The noise terms $\set{\varepsilon_t}_{t \ge 1}$ are independent sub-Gaussian random variables with zero mean and variance proxy $\varsigma^2 > 0$, such that $\EE \left[ \exp(\lambda \varepsilon_t) \right] \le \exp\left( \varsigma^2 \lambda^2 / 2 \right)$ for all $\lambda \in \RR$ and $t \ge 1$.
\end{assumption}

\begin{assumption}
    \label{asm:bandwidth}
    The bandwidth sequence $\set{\ell_t}_{t \ge 1}$ is polynomially lower bounded, i.e., there exist constants $A_{\Psi}, \alpha > 0$ such that $\ell_t \ge A_{\Psi} t^{-\alpha}$ for all $t \ge 1$.
\end{assumption}

In \cref{asm:domain}, we assume that $\fX$ is convex to simplify the analysis.
The results of this paper extend directly to any connected compact set $\fX$ satisfying the interior cone condition (see, e.g., \cite[Definition 3.6]{wendland2004scattered}).

The compact support and H\"older continuity in \cref{asm:kernel} are standard conditions ensuring a uniform bound on the KR prediction error \cite{linke2023towards}.
Both are natively satisfied by kernels such as the Epanechnikov and triangular kernels.
The Gaussian kernel, while requiring numerical truncation to achieve compact support in practice, inherently satisfies the H\"older continuity condition.

Additionally, \cref{asm:bandwidth} prevents the estimator from degenerating when the bandwidth decays too rapidly. 
The polynomial lower bound is compatible with common bandwidth selection rules, such as Scott's \cite{scott2015multivariate} and Silverman's \cite{silverman1998density} rule of thumb, which yield $\ell_t = \bigTheta( t^{-\frac{1}{d + 4}} )$.

\subsection{Prediction error bounds for KR}
\label{sec:prediction error bound}
To establish convergence guarantees in a frequentist setting (\cref{asm:function}), we derive non-asymptotic prediction error bounds for the KR estimator.
Specifically, we provide pointwise bounds to control the error at the fixed optimizer $\vx^*$, and uniform bounds to account for the allocation/recommendation sequence, which depends on the noise.
We first derive these bounds in terms of the kernel density (\cref{thm:pointwise error of KR by KDE,thm:uniform error of KR by KDE}), and then relax them into fill distance-based forms to facilitate the analysis of simple regret (\cref{thm:pointwise error of KR by fill distance,thm:uniform error of KR by fill distance}).

Under \cref{asm:kernel}, $W_t(\vx)$ may vanish when the bandwidth is small, yielding an indeterminate $0/0$ form.
To make the estimator globally well-defined, we extend it by setting $m_t(\vx) = m(\vx; \mD_t) \coloneqq (1 / t) \sum_{i=1}^{t} y_i$ whenever $W_t(\vx) = 0$.
This ensures $m_t(\vx)$ remains a convex combination of the observations, thus allowing the subsequent concentration analysis to apply uniformly across $\fX$.
\cref{lem:uniform error of extended KR} follows directly from Hoeffding's inequality (see, e.g., \cite[Proposition 2.5]{wainwright2019high}); its proof is given in \cref{app:proof of prediction error bound}.

\begin{lemma}
    \label{lem:uniform error of extended KR}
    Suppose \cref{asm:domain,asm:function,asm:kernel,asm:noise} hold. For any $t \ge 1$, $\delta \in (0, 1)$, and $\set{\vx_1, \dots, \vx_t} \subset \fX$, with probability at least $1 - \delta$,
    \begin{equation}\label{eqn:uniform error of extended KR}
        \left| f(\vx) - m_t(\vx) \right| \le 2 \| f \|_{\infty} + \sqrt{2 t^{-1} \varsigma^2 \ln(2 / \delta)},
    \end{equation}
    for all $\vx \in \fX$ such that $W_t(\vx) = 0$, where $\| f \|_{\infty} \coloneqq \sup_{\vx \in \fX} |f(\vx)| < \infty$.
\end{lemma}

\subsubsection{Error bounds in terms of the kernel density}
\label{sec:prediction error bound by KDE}

We first derive error bounds of KR that depend on the KDE at the query point $\vx$.
\cref{thm:pointwise error of KR by KDE} shows that the estimation error is controlled by $\hsigma_t(\vx)$ and the modulus of uniform continuity of $f$.

\begin{theorem}
    \label{thm:pointwise error of KR by KDE}
     Suppose \cref{asm:domain,asm:function,asm:kernel,asm:noise} hold. 
    There exists a constant $C_0 > 0$ depending only on $\fX$, $\Psi$, $f$, $\varrho$ and $\varsigma^2$.
    For any $t \ge 1$, $\delta \in (0, 1)$, and $\set{\vx_1, \dots, \vx_t} \subset \fX$, at any fixed $\vx \in \fX$ such that $W_t(\vx) > 0$, with probability at least $1 - \delta$,
    \begin{equation}\label{eqn:pointwise error of KR by KDE}
        \left| f(\vx) - m_t(\vx) \right| \le C_0 \left( \hsigma_t(\vx) \sqrt{\ln(2 / \delta)} + \omega_f(\ell_t) \right),
    \end{equation}
    where $\omega_f(z) \coloneqq \sup_{\substack{\vx, \vx' \in \fX \\ \| \vx - \vx' \| \le z}} |f(\vx) - f(\vx')|$ is the modulus of uniform continuity.
\end{theorem}

\begin{proof}
    Decompose the prediction error into bias and random error:
    \[
        f(\vx) - m_t(\vx) = \underbrace{f(\vx) - \frac{\sum_{i=1}^{t} k(\vx, \vx_i) f(\vx_i)}{\sum_{i=1}^{t} k(\vx, \vx_i)}}_{\text{Bias}} - \underbrace{\frac{\sum_{i=1}^{t} k(\vx, \vx_i) \varepsilon_i}{\sum_{i=1}^{t} k(\vx, \vx_i)}}_{\text{Random Error}}.
    \]
    Under \cref{asm:kernel}, the bias term satisfies
    \[
        \left| f(\vx) - \frac{\sum_{i=1}^{t} k(\vx, \vx_i) f(\vx_i)}{\sum_{i=1}^{t} k(\vx, \vx_i)} \right|
        = \left| \frac{\sum_{i=1}^{t} k(\vx, \vx_i) ( f(\vx) - f(\vx_i) )}{\sum_{i=1}^{t} k(\vx, \vx_i)} \right|
        \le \omega_f(R_{\Psi} \ell_t).
    \]
    By \cref{asm:domain,asm:function}, there exist $\vz', \vz'' \in \fX$ such that $\omega_f(R_{\Psi} \ell_t) = \left| f(\vz') - f(\vz'') \right|$. 
    Let $N \coloneqq \lceil R_{\Psi} \rceil + 1$.
    Since $\fX$ is convex, there exists a sequence $\set{\vz_0 = \vz', \vz_1, \dots, \vz_N = \vz''} \subset \fX$ such that $\| \vz_j - \vz_{j-1} \| \le \ell_t$ for all $1 \le j \le N$.
    Thus,
    \[
        \omega_f(R_{\Psi} \ell_t)
        = \left| \sum_{j=1}^{N} ( f(\vz_j) - f(\vz_{j-1}) ) \right|
        \le \sum_{j=1}^{N} \left| f(\vz_j) - f(\vz_{j-1}) \right|
        \le N \omega_f(\ell_t).
    \]
    For the random error, Hoeffding's inequality implies that for any $\lambda \ge 0$,
    \[
        \PP \left[ \left| \frac{\sum_{i=1}^{t} k(\vx, \vx_i) \varepsilon_i}{\sum_{i=1}^{t} k(\vx, \vx_i)} \right| \ge \lambda \right]
        \le 2 \exp \left( - \frac{\lambda^2 W_t^2(\vx)}{2 \varsigma^2 \sum_{i=1}^{t} k^2(\vx, \vx_i)} \right)
        \le 2 \exp \left( - \frac{\lambda^2}{2 \varsigma^2} \frac{W_t(\vx) + \varrho}{1 + \varrho} \right),
    \]
    where the last inequality follows from $0 \le k(\vx, \vx_i) \le 1$, which ensures
    \[
        \frac{\left( \sum_{i=1}^{t} k(\vx, \vx_i) \right)^2}{\sum_{i=1}^{t} k^2(\vx, \vx_i)}
        \ge \max\set{1, \sum_{i=1}^{t} k(\vx, \vx_i)}
        \ge \frac{\sum_{i=1}^{t} k(\vx, \vx_i) + \varrho}{1 + \varrho}.
    \]
    Setting $\lambda = \hsigma_t(\vx) \sqrt{2 \varsigma^2 (1 + \varrho) \ln(2 / \delta)}$, we obtain that with probability at least $1 - \delta$,
    \[
        |f(\vx) - m_t(\vx)| \le \left( \lceil R_{\Psi} \rceil + 1 \right) \omega_f(\ell_t) + \hsigma_t(\vx) \sqrt{2 \varsigma^2 (1 + \varrho) \ln(2 / \delta)}.
    \]
    The proof is completed by letting $C_0 = \max\set{\lceil R_{\Psi} \rceil + 1, \sqrt{2 \varsigma^2 (1 + \varrho)}}$.
\end{proof}

To control the error for the noise-dependent query sequence, we extend the preceding pointwise result to a uniform bound over $\fX$.

\begin{theorem}
    \label{thm:uniform error of KR by KDE}
    Suppose \cref{asm:domain,asm:function,asm:kernel,asm:noise,asm:bandwidth} hold.
    There exists a constant $C_0 > 0$ depending only on $\fX$, $\Psi$, $f$, $\varrho$, and $\varsigma^2$.
    For any $t \ge 1$, $\delta \in (0, 1)$, $\ell_t \le \diam(\fX)$, and $\set{\vx_1, \dots, \vx_t} \subset \fX$, with probability at least $1 - \delta$,
    \begin{equation}\label{eqn:uniform error of KR by KDE}
        \left| f(\vx) - m_t(\vx) \right| \le C_0 \left( \left( \hsigma_t(\vx) + t^{-1} W^{-1}_t(\vx) \right) \sqrt{d \ln(t + 1) + \ln(4 / \delta)} + \omega_f(\ell_t) \right),
    \end{equation}
    for all $\vx \in \fX$ such that $W_t(\vx) > 0$.
\end{theorem}

\begin{proof}
    For brevity, we denote $\nu_t(\vx) \coloneqq \sum_{i=1}^{t} k(\vx, \vx_i) \varepsilon_i$.
    Similar to \cref{thm:pointwise error of KR by KDE}, it suffices to bound the random error term $|\nu_t(\vx)| / W_t(\vx)$.
    For any $\delta \in (0, 1)$, applying the union bound to $\PP\left[ |\varepsilon_i| \ge \sqrt{2 \varsigma^2 \ln (4 t / \delta)} \right] \le \delta / (2 t)$ for $i \in \set{1, \dots, t}$ yields
    \[
        \PP\left[ \max_{1 \le i \le t} |\varepsilon_i| \le \sqrt{2 \varsigma^2 \ln \frac{4 t}{\delta}} \right] \ge 1 - \frac{\delta}{2}.
    \]
    Denote this event by $\fE_1$.
    Let $D_{\fX} \coloneqq \diam(\fX)$.
    For any $\epsilon \in (0, D_{\fX}]$, there exists an $\epsilon$-net of $\fX$, denoted by $\fC_t$, such that $N_t \coloneqq \card{\fC_t} \le (2 D_{\fX} / \epsilon + 1)^d \le (3 D_{\fX} / \epsilon)^d$ (see, e.g., \cite[Proposition 4.2.10]{vershynin2026high}).
    Applying the union bound on $\fC_t$, we obtain
    \[
        \PP\left[ \frac{|\nu_t(\vx')|}{W_t(\vx')} \le \hsigma_t(\vx') \sqrt{2 \varsigma^2 (1 + \varrho) \ln \frac{4 N_t}{\delta}}, \quad \forall \vx' \in \fC_t \right] \ge 1 - \frac{\delta}{2}.
    \]
    Denote this event by $\fE_2$.
    Conditioned on $\fE_1 \cap \fE_2$, which holds with probability at least $1 - \delta$,
    \[
        \begin{aligned}
            |W_t(\vx) - W_t([\vx]_t)| 
            &\le \sum_{i=1}^{t} |k(\vx, \vx_i) - k([\vx]_t, \vx_i)| \\
            &\le t L_{\Psi} \ell_t^{-\theta} \| \vx - [\vx]_t \|^{\theta}
            \le t L_{\Psi} \ell_t^{-\theta} \epsilon^{\theta} \eqqcolon \Delta_W,
        \end{aligned}
    \]
    where $[\vx]_t \in \argmin_{\vx' \in \fC_t} \| \vx - \vx' \|$ is the closest point in $\fC_t$ to $\vx$, and the second inequality follows from the H\"older continuity of $\Psi$.
    Then,
    \[
        \begin{aligned}
            |\nu_t(\vx) - \nu_t([\vx]_t)| 
            &\le \sum_{i=1}^{t} |\varepsilon_i| \cdot |k(\vx, \vx_i) - k([\vx]_t, \vx_i)| \\
            &\le \left( \max_{1 \le i \le t} |\varepsilon_i| \right) \Delta_W
            \le \sqrt{2 \varsigma^2 \ln \frac{4 t}{\delta}} \Delta_W \eqqcolon \Delta_Z.
        \end{aligned}
    \]
    Let $g(z) = \frac{z}{\sqrt{z + \varrho}}$, for any $z \ge 0$,
    \[
        |g'(z)| 
        = \left| \frac{\sqrt{z + \varrho} - \frac{z}{2 \sqrt{z + \varrho}}}{z + \varrho} \right|
        = \frac{z + 2 \varrho}{2 (z + \varrho)^{3/2}}
        \le \frac{1}{\sqrt{z + \varrho}}
        \le \frac{1}{\sqrt{\varrho}},
    \]
    and for any $\vx' \in \fC_t$,
    \[
        |\nu_t(\vx')| \le g( W_t(\vx') ) \sqrt{2 \varsigma^2 (1 + \varrho) \ln \frac{4 N_t}{\delta}}.
    \]
    Therefore,
    \[
        \begin{aligned}
            |\nu_t(\vx)| 
            &\le |\nu_t([\vx]_t)| + \Delta_Z \\
            &\le g( W_t([\vx]_t) ) \sqrt{2 \varsigma^2 (1 + \varrho) \ln \frac{4 N_t}{\delta}} + \Delta_Z \\
            &\le \left( g( W_t(\vx) ) + \frac{\Delta_W}{\sqrt{\varrho}} \right) \sqrt{2 \varsigma^2 (1 + \varrho) \ln \frac{4 N_t}{\delta}} + \Delta_Z.
        \end{aligned}
    \]
    Divided by $W_t(\vx)$ yields
    \[
        \begin{aligned}
            \frac{|\nu_t(\vx)|}{W_t(\vx)}
            &\le \hsigma_t(\vx) \sqrt{2 \varsigma^2 (1 + \varrho) \ln \frac{4 N_t}{\delta}} + \frac{\Delta_W}{W_t(\vx)} \left( \sqrt{2 \varsigma^2 (1 + \varrho^{-1}) \ln \frac{4 N_t}{\delta}} + \sqrt{2 \varsigma^2 \ln \frac{4 t}{\delta}} \right) \\
            &\le \hsigma_t(\vx) \sqrt{2 \varsigma^2 (1 + \varrho) \ln \frac{4 N_t}{\delta}} + \frac{\Delta_W}{W_t(\vx)} \sqrt{2 \varsigma^2 (1 + \varrho^{-1})} \left( \sqrt{\ln \frac{4 N_t}{\delta}} + \sqrt{\ln \frac{4 t}{\delta}} \right) \\
            &\le \hsigma_t(\vx) \sqrt{2 \varsigma^2 (1 + \varrho) \ln \frac{4 N_t}{\delta}} + \frac{\Delta_W}{W_t(\vx)} \sqrt{4 \varsigma^2 (1 + \varrho^{-1}) \left( \ln (t N_t) + 2 \ln (4 / \delta) \right)},
        \end{aligned}
    \]
    where the second inequality is given by $\varrho > 0$, and the last inequality uses the Cauchy--Schwarz inequality.
    Set $\epsilon = \ell_t t^{-2/\theta} \le D_{\fX}$, we have $\Delta_W = t^{-1} L_{\Psi}$ and $\ln N_t \le d \ln(3 D_{\fX} / \epsilon) = d \ln \left( 3 D_{\fX} \ell_t^{-1} t^{2/\theta} \right)$.
    Since $\ell_t^{-1} \le A_{\Psi}^{-1} t^{\alpha}$,
    \[
        \begin{aligned}
            \ln N_t 
            &\le d \ln \left( 3 D_{\fX} A_{\Psi}^{-1} t^{\alpha + 2/\theta} \right)
            = d \ln \left( 3 D_{\fX} A_{\Psi}^{-1} \right) + \left( \alpha + \frac{2}{\theta} \right) d \ln t \\
            &\le d \frac{3 D_{\fX} A_{\Psi}^{-1}}{\ln 2} \ln (t + 1) + \left( \alpha + \frac{2}{\theta} \right) d \ln (t + 1) \\
            &= \left( \frac{3 D_{\fX} A_{\Psi}^{-1}}{\ln 2} + \alpha + \frac{2}{\theta} \right) d \ln(t + 1) 
            \eqqcolon A_0 d \ln(t + 1).
        \end{aligned}
    \]
    where $A_0 > 0$ is a universal constant. 
    The last inequality uses $\ln z < z$ for $z > 0$ and $\ln(t + 1) \ge \ln 2$.
    Hence,
    \[
        \begin{aligned}
            \frac{|\nu_t(\vx)|}{W_t(\vx)}
            &\le \hsigma_t(\vx) \sqrt{2 \varsigma^2 (1 + \varrho) \left( A_0 d \ln(t + 1) + \ln(4 / \delta) \right)} \\
            &\phantom{=} + W_t^{-1}(\vx) \cdot \frac{L_{\Psi}}{t} \sqrt{4 \varsigma^2 (1 + \varrho^{-1}) \left( (A_0 + 1) d \ln(t + 1) + 2 \ln(4 / \delta) \right)} \\
            &\le A_1 \hsigma_t(\vx) \sqrt{d \ln(t + 1) + \ln(4 / \delta)} + A_2 W^{-1}_t(\vx) t^{-1} \sqrt{d \ln(t + 1) + \ln(4 / \delta)},
        \end{aligned}
    \]
    where $A_1 = \sqrt{2 \varsigma^2 (1 + \varrho) (A_0 + 1)}$, $A_2 = L_{\Psi} \sqrt{4 \varsigma^2 (1 + \varrho^{-1}) (A_0 + 2)}$.
    The proof is completed by setting $C_0 = \max\set{\lceil R_{\Psi} \rceil + 1, A_1, A_2}$, where $\lceil R_{\Psi} \rceil + 1$ is specified for the bias term.
\end{proof}

\subsubsection{Error bounds in terms of the fill distance}
\label{sec:prediction error bound by fill distance}

While the KDE-based bounds in \cref{thm:pointwise error of KR by KDE,thm:uniform error of KR by KDE} are sharp, they are intrinsically position-dependent.
To establish explicit convergence rates for the regret, it is necessary to translate these results into bounds based on the fill distance.
This transition is facilitated by \cref{lem:bound KDE by fill distance}, which provides a lower bound for the KDE in terms of the fill distance.

\begin{lemma}
    \label{lem:bound KDE by fill distance}
    Suppose \cref{asm:domain,asm:kernel} hold. 
    For any $\ellmax > 0$, there exist constants $\filldist_0$, $w_0 > 0$ depending only on $\fX$, $\Psi$, and $\ellmax$.
    For any $t \ge 1$, $\ell_t \le \ellmax$, and $\mX_t = \set{\vx_1, \dots, \vx_t} \subset \fX$.
    If $\filldist_{\fX, \mX_t} \le \filldist_0 \ell_t$, then for all $\vx \in \fX$,
    \begin{equation}
        W_t(\vx) \ge w_0 \ell_t^d \filldist_{\fX, \mX_t}^{-d}.
    \end{equation}
\end{lemma}

Substituting this lower bound into the KDE-based result of \cref{thm:pointwise error of KR by KDE} yields the pointwise prediction error in terms of the fill distance.

\begin{theorem}
    \label{thm:pointwise error of KR by fill distance}
    Suppose \cref{asm:domain,asm:function,asm:kernel,asm:noise} hold. 
    For any $\ellmax > 0$, there exist constants $C_0$, $\filldist_0 > 0$ depending only on $\fX$, $\Psi$, $f$, $\varrho$, $\varsigma^2$, and $\ellmax$.
    For any $t \ge 1$, $\delta \in (0, 1)$, $\ell_t \le \ellmax$, and $\mX_t = \set{\vx_1, \dots, \vx_t} \subset \fX$.
    If $\filldist_{\fX, \mX_t} \le \filldist_0 \ell_t$, then at any fixed $\vx \in \fX$, with probability at least $1 - \delta$,
    \begin{equation}\label{eqn:pointwise error of KR by fill distance}
        \left| f(\vx) - m_t(\vx) \right| \le C_0 \left( \ell_t^{-d/2} \filldist_{\fX, \mX_t}^{d/2} \sqrt{\ln(2 / \delta)} + \omega_f(\ell_t) \right).
    \end{equation}
\end{theorem}

\begin{proof}
    Under these conditions, \cref{lem:bound KDE by fill distance} implies that $W_t(\vx)$ is bounded away from zero.
    Since $\varrho > 0$, $\hsigma_t(\vx) \le W_t^{-1/2}(\vx) \le w_0^{-1/2} \ell_t^{-d/2} \filldist_{\fX, \mX_t}^{d/2}$.
    By \cref{thm:pointwise error of KR by KDE}, there exists $A_0 > 0$ such that for any $\delta \in (0, 1)$, with probability at least $1 - \delta$,
    \[
        \begin{aligned}
            \left| f(\vx) - m_t(\vx) \right| 
            &\le A_0 \left( \hsigma_t(\vx) \sqrt{\ln(2 / \delta)} + \omega_f(\ell_t) \right) \\
            &\le A_0 \left( w_0^{-1/2} \ell_t^{-d/2} \filldist_{\fX, \mX_t}^{d/2} \sqrt{\ln(2 / \delta)} + \omega_f(\ell_t) \right).
        \end{aligned}
    \]
    The proof is completed by setting $C_0 = A_0 \max\set{w_0^{-1/2}, 1}$.
\end{proof}

Following the same rationale as in the previous section, we extend this geometric pointwise guarantee to a uniform bound over $\fX$ to accommodate the noise-dependent query sequence.

\begin{theorem}
    \label{thm:uniform error of KR by fill distance}
    Suppose \cref{asm:domain,asm:function,asm:kernel,asm:noise,asm:bandwidth} hold. 
    For any $\ellmax \in (0, \diam(\fX)]$, there exist constants $C_0$, $\filldist_0 > 0$ depending only on $\fX$, $\Psi$, $f$, $\varrho$, $\varsigma^2$, and $\ellmax$.
    For any $t \ge 1$, $\delta \in (0, 1)$, $\ell_t \le \ellmax$, and $\mX_t = \set{\vx_1, \dots, \vx_t} \subset \fX$.
    If $\filldist_{\fX, \mX_t} \le \filldist_0 \ell_t$, then with probability at least $1 - \delta$,
    \begin{equation}\label{eqn:uniform error of KR by fill distance}
        \left| f(\vx) - m_t(\vx) \right| \le C_0 \left( \ell_t^{-d/2} \filldist_{\fX, \mX_t}^{d/2} \sqrt{d \ln(t + 1) + \ln(4 / \delta)} + \omega_f(\ell_t) \right).
    \end{equation}
\end{theorem}

\begin{proof}
    Under these conditions, \cref{lem:bound KDE by fill distance} implies that $W_t(\vx)$ is bounded away from zero.
    Since $\ell_t \le \ellmax \le \diam(\fX)$, by \cref{thm:uniform error of KR by KDE}, there exists $A_0 > 0$ such that for any $\delta \in (0, 1)$, with probability at least $1 - \delta$,
    \[
        \begin{aligned}
            \left| f(\vx) - m_t(\vx) \right|
            \le& A_0 \left( \left( \hsigma_t(\vx) + t^{-1} W_t^{-1}(\vx) \right) \cdot \sqrt{d \ln(t + 1) + \ln(4 / \delta)} + \omega_f(\ell_t) \right) \\
            \le& A_0 \left( W_t^{-1/2}(\vx) \left( 1 + W_t^{-1/2}(\vx) \right) \cdot \sqrt{d \ln(t + 1) + \ln(4 / \delta)} + \omega_f(\ell_t) \right) \\
            \le& A_0 \left( w_0^{-\frac{1}{2}} \left( 1 + w_0^{-\frac{1}{2}} \filldist_0^{\frac{d}{2}} \right) \cdot \ell_t^{-\frac{d}{2}} \filldist_{\fX, \mX_t}^{\frac{d}{2}} \sqrt{d \ln(t + 1) + \ln(4 / \delta)} + \omega_f(\ell_t) \right),
        \end{aligned}
    \]
    where the last inequality follows from $W_t^{-1/2}(\vx) \le w_0^{-1/2} \ell_t^{-d/2} \filldist_{\fX, \mX_t}^{d/2} \le w_0^{-1/2} \filldist_0^{d/2}$.
    The proof is completed by setting $C_0 = A_0 \max\set{w_0^{-1/2} \left( 1 + w_0^{-1/2} \filldist_0^{d/2} \right), 1}$.
\end{proof}

\subsection{Algorithmic consistency}
\label{sec:algorithmic consistency}
For an allocation strategy such as BOKE or DE, we aim to analyze the denseness of the query sequence generated by the optimization algorithm, which is referred to as the \textit{algorithmic consistency} and is widely used to establish the convergence of global optimization in noise-free settings (see, e.g., \cite[Theorem 1.3]{torn1989global}).
The absence of algorithmic consistency may lead to the risk of being trapped in local optima, as the algorithm's iterates might not sufficiently explore the decision set. 
We adopt the \textit{sequential no-empty-ball} (SNEB) property (see \cite[Assumption 3.1]{chen2023pseudo}) to establish algorithmic consistency for our proposed algorithms.

In this section, we extend the SNEB property and algorithmic consistency to the noisy setting. First, we present the SNEB property of IKR-UCB in \cref{lem:SNEB IKR-UCB}. Then, the algorithmic consistency of BOKE and BOKE+ are established in \cref{thm:consistency BOKE} and \cref{cor:consistency BOKE+}, respectively. Additionally, the algorithmic consistency of the DE strategy, which leads to a pure exploration design, is provided in \cref{cor:consistency DE}.

To show the SNEB property of the IKR-UCB acquisition function, we consider its rescaled form instead:
\begin{equation}
    \tilde{a}_t(\vx) = \tilde{a}(\vx; \mD_t) \coloneqq \beta_t^{-1} m(\vx; \mD_t) + \hsigma(\vx; \mX_t).
\end{equation}
where $\beta_t > 0$ is a tuning parameter introduced in the definition of acquisition function. The SNEB property for IKR-UCB is presented in \cref{lem:SNEB IKR-UCB}, and its proof is provided in \cref{app:proof of algorithmic consistency}.

To analyze the SNEB property, we introduce an auxiliary dataset $\tilde{\mD}_t = \set{(\vx_1, y_1), \ldots} \subset \fX \times \RR$ and its projection onto the first coordinate, $\tilde{\mX}_t = \set{\vx_1, \ldots} \subset \fX$. Here, the notation denotes an arbitrary finite set; the subscript $t$ is merely an index and does not necessarily equal its cardinality.

\begin{lemma}[SNEB property of IKR-UCB]
    \label{lem:SNEB IKR-UCB}
    Suppose \cref{asm:domain,asm:function,asm:kernel,asm:noise,asm:bandwidth} hold. 
    For any fixed $\ell_t \equiv \ell$, and any sequence $\set{\beta_t}_{t \ge 1}$ with $\lim_{t \to \infty} \beta_t = \infty$, the following statements hold:
    \begin{enumerate}[label=(\alph*)]
        \item\label{it1:SNEB IKR-UCB} For any sequence $\set{\vx_t}_{t \ge 1} \subset \fX$, let $\mX_t = \set{\vx_1, \dots, \vx_t}$ and $\mD_t = \set{(\vx_i, y_i)}_{i=1}^{t}$. If there exists $\vx \in \fX$ with $\inf_{t} d(\vx, \mX_t) > R_{\Psi} \ell$, then $\tilde{a}(\vx; \mD_t) \pto \varrho^{-1/2}$ as $t \to \infty$.

        \item\label{it2:SNEB IKR-UCB} For any convergent sequence $\set{\vx_t}_{t \ge 1} \subset \fX$ and any sequence of finite datasets $\tilde{\mD}_t$ containing $\set{(\vx_i, y_i)}_{i=1}^{t}$. If $\beta_t = \bigOmega(\sqrt{\ln t})$, then $\tilde{a}(\vx_t; \tilde{\mD}_{t-1}) \pto 0$ as $t \to \infty$.
    \end{enumerate}
\end{lemma}

Based on \cref{lem:SNEB IKR-UCB}, we establish the algorithmic consistency of both BOKE and BOKE+ in the sense of almost sure convergence, as shown in \cref{thm:consistency BOKE} and \cref{cor:consistency BOKE+}.

\begin{theorem}[Algorithmic consistency of BOKE]
    \label{thm:consistency BOKE}
    Suppose \cref{asm:domain,asm:function,asm:kernel,asm:noise,asm:bandwidth} hold. 
    For any sequence $\set{\beta_t}_{t \ge 1}$ with $\beta_t = \bigOmega(\sqrt{\ln t})$, let $\mX_t = \set{\vx_1, \dots, \vx_t}$ denote the BOKE iterates and $\mD_t = \set{(\vx_i, y_i)}_{i=1}^{t}$ be the corresponding datasets. Then,
    \begin{enumerate}[label=(\roman*)]
        \item\label{it1:consistency BOKE} For any fixed $\ell_t \equiv \ell$, $\fX$ is eventually populated by the BOKE iterates within $\bigO(\ell)$ gaps, i.e., $\inf_t \filldist_{\fX, \mX_t} \le R_{\Psi} \ell$ holds almost surely.

        \item\label{it2:consistency BOKE} There exists a data-dependent bandwidth schedule $\set{\ell_t}_{t \ge 1}$, where each $\ell_t$ is determined by $\mD_{t-1}$, such that $\ell_t \asto 0$, and $\fX$ is eventually populated by the BOKE iterates, i.e., $\filldist_{\fX, \mX_t} \asto 0$ as $t \to \infty$.
    \end{enumerate}
\end{theorem}

\begin{proof}
    \proofparagraph{\cref{it1:consistency BOKE}}
    Suppose for contradiction that the event $\fE \coloneqq \set{\inf_{t} \filldist_{\fX, \mX_t} > R_{\Psi} \ell}$ has a strictly positive probability, i.e., $\PP[\fE] > 0$. 
    Conditioned on $\fE$, the monotonicity of $\filldist_{\fX, \mX_t}$ with respect to $t$ implies $\lim_{t \to \infty} \filldist_{\fX, \mX_t} > R_{\Psi} \ell$. 
    Hence, almost surely on $\fE$, there exist constants $N$ and $\epsilon > 0$ such that for all $t \ge N$, $\filldist_{\fX, \mX_t} = \sup_{\vx \in \fX} d(\vx, \mX_t) \ge R_{\Psi} \ell + \epsilon$. 
    Let $F_t = \set{\vx \in \fX : d(\vx, \mX_t) \ge R_{\Psi} \ell + \epsilon}$. 
    The continuity of $d(\vx, \mX_t)$ with respect to $\vx$ and the compactness of $\fX$ imply that $F_t$ is closed and non-empty. 
    Since $d(\vx, \mX_t) \ge d(\vx, \mX_{t+1})$, we have $\fX \supset F_t \supset F_{t+1}$. 
    By Cantor's intersection theorem, $\cap_{t=N}^{\infty} F_t \neq \varnothing$. 
    Thus, there exists $\vx' \in \cap_{t=N}^{\infty} F_t$ such that for all $t \ge N$, $d(\vx', \mX_t) \ge R_{\Psi} \ell + \epsilon$.
    Hence, $\inf_t d(\vx', \mX_t) = \lim_{t \to \infty} d(\vx', \mX_t) > R_{\Psi} \ell$.
    By \cref{lem:SNEB IKR-UCB}\cref{it1:SNEB IKR-UCB}, $\tilde{a}(\vx'; \mD_t) \ptogiven{\fE} \varrho^{-1/2}$ as $t \to \infty$.

    By compactness of $\fX$, $\set{\vx_t}_{t \ge 1}$ has a convergent subsequence $\set{\vx_{t_n}}_{n \ge 1}$.
    By \cref{lem:SNEB IKR-UCB}\cref{it2:SNEB IKR-UCB}, $\tilde{a}(\vx_{t_n}; \mD_{t_n - 1}) \ptogiven{\fE} 0$ as $n \to \infty$.
    Then, by Slutsky's theorem,
    \[
        \tilde{a}(\vx'; \mD_{t_n - 1}) - \tilde{a}(\vx_{t_n}; \mD_{t_n - 1}) \ptogiven{\fE} \varrho^{-1/2}
        \quad \text{ as } n \to \infty.
    \]
    Hence, by the definition of conditional probability $\PP[\cdot \cap \fE] = \PP\left[ \cdot \given \fE \right] \PP[\fE]$, we equivalently have
    \begin{equation}\label{eqn:acq x' > x_tn}
        \lim_{n \to \infty} \PP\left[ \set{\tilde{a}(\vx'; \mD_{t_n - 1}) > \tilde{a}(\vx_{t_n}; \mD_{t_n - 1})} \cap \fE \right] = \PP[\fE].
    \end{equation}
    However, the IKR-UCB criterion \eqref{eqn:max IKR-UCB} gives $\vx_{t_n} \in \argmax_{\vx \in \fX} \tilde{a}(\vx; \mD_{t_n - 1})$, so
    \[
        \tilde{a}(\vx; \mD_{t_n - 1}) \le \tilde{a}(\vx_{t_n}; \mD_{t_n - 1}), \quad \forall \vx \in \fX, \enspace \forall n \ge 1
    \]
    holds almost surely. 
    This implies
    \begin{equation}\label{eqn:acq x' <= x_tn}
        \PP\left[ \set{\tilde{a}(\vx'; \mD_{t_n - 1}) > \tilde{a}(\vx_{t_n}; \mD_{t_n - 1})} \cap \fE \right] = 0, \quad \forall n \ge 1.
    \end{equation}
    Combining this with \cref{eqn:acq x' > x_tn} yields $\PP[\fE] = 0$, which directly contradicts our initial assumption that $\PP[\fE] > 0$.
    Therefore, the assumption must be false, and hence $\inf_{t} \filldist_{\fX, \mX_t} \le R_{\Psi} \ell$ almost surely.

    \proofparagraph{\cref{it2:consistency BOKE}}
    Let $\set{\ell'_n}_{n \ge 1}$ be a deterministic sequence chosen to satisfy \cref{asm:bandwidth}, such as $\ell'_n = 1 / n$ for all $n \ge 1$.
    Denote by $\fF_t$ the natural filtration generated by the dataset $\mD_t$.
    We recursively construct the bandwidth schedule $\set{\ell_t}_{t \ge 1}$ alongside a sequence of random times $\set{\tau_n}_{n \ge 0}$. 
    Let $\tau_0 = 0$. 
    For each $n \ge 1$, maintain the bandwidth $\ell_t = \ell'_n$ for all $t > \tau_{n-1}$ until the following stopping time is reached:
    \[
        \tau_n = \inf\set{t > \tau_{n-1} : \filldist_{\fX, \mX_t} < 2 R_{\Psi} \ell'_n}.
    \]
    Since $\filldist_{\fX, \mX_t}$ is $\fF_t$-measurable, for any $n \ge 1$, $\tau_n$ is a stopping time with respect to the filtration $\set{\fF_t}_{t \ge 1}$.
    Assume inductively that $\PP[\tau_{n-1} < \infty] = 1$ for some $n \ge 1$.
    Conditioned on $\fF_{\tau_{n-1}}$, the iterates $\set{\vx_t}_{t > \tau_{n-1}}$ can be seen as being generated by a restarted BOKE algorithm given the initial dataset $\mD_{\tau_{n-1}}$ and a fixed bandwidth $\ell'_n$.
    By part \cref{it1:consistency BOKE}, these iterates satisfy $\inf_{t > \tau_{n-1}} \filldist_{\fX, \mX_t} \le R_{\Psi} \ell'_n$ almost surely.
    Hence, $\PP\left[ \tau_n < \infty \given \fF_{\tau_{n-1}} \right] = 1$, a.s. 
    By taking the expectation over $\fF_{\tau_{n-1}}$, we obtain unconditionally $\PP[ \tau_n < \infty] = \EE\left[ \PP\left[ \tau_n < \infty \given \fF_{\tau_{n-1}} \right] \right] = 1$.
    By mathematical induction, $\PP[\tau_n < \infty] = 1$ holds for all $n \ge 1$.
    Since $\set{\tau_n}_{n \ge 0}$ is a sequence of almost surely finite stopping times, the recursively constructed bandwidth can be equivalently written as
    \[
        \ell_t = \sum_{n=1}^{\infty} \ell'_n \1_{\set{\tau_{n-1} < t \le \tau_n}}, \quad \forall t \ge 1.
    \]
    Note that the event $\set{\tau_{n-1} < t \le \tau_n} = \set{\tau_{n-1} \le t - 1} \cap \set{\tau_n \le t - 1}^{\complement}$ is $\fF_{t-1}$-measurable; therefore, $\set{\ell_t}_{t \ge 1}$ is a predictable sequence.
    Since $\filldist_{\fX, \mX_t}$ is monotonically non-increasing with respect to $t$, $\set{\tau_n}_{n \ge 0}$ is strictly increasing and $\tau_n \asto \infty$.
    Consequently, as $t \to \infty$, the active bandwidth $\ell_t$ tracks the tail of $\set{\ell'_n}_{n \ge 1}$, yielding $\ell_t \asto 0$. Furthermore, $\filldist_{\fX, \mX_{\tau_n}} \le 2 R_{\Psi} \ell'_n \to 0$ almost surely, which combined with the monotonicity of $\filldist_{\fX, \mX_t}$ gives $\filldist_{\fX, \mX_t} \asto 0$.
\end{proof}

The proof of algorithmic consistency for BOKE+ parallels that for BOKE, using the Borel--Cantelli lemma to replicate the contradiction argument in \cref{thm:consistency BOKE}.

\begin{corollary}[Algorithmic consistency of BOKE+]
    \label{cor:consistency BOKE+}
    Suppose \cref{asm:domain,asm:function,asm:kernel,asm:noise,asm:bandwidth} hold. 
    Then the conclusions of \cref{thm:consistency BOKE} also apply to the BOKE+ iterates.
\end{corollary}

\begin{proof}
    First, suppose for contradiction that the event $\fE \coloneqq \set{\inf_{t} \filldist_{\fX, \mX_t} > R_{\Psi} \ell}$ has a strictly positive probability, i.e., $\PP[\fE] > 0$. 
    Conditioned on $\fE$, following the same argument as in \cref{thm:consistency BOKE}\cref{it1:consistency BOKE}, there exists $\vx' \in \fX$ such that $\tilde{a}(\vx'; \mD_t) \ptogiven{\fE} \varrho^{-1/2}$ as $t \to \infty$.

    Denote by $I_t \in \set{0, 1}$ the indicator variable for the event that the IKR-UCB acquisition function is chosen at iteration $t$.
    Let $\fT = \set{t \ge 1 : I_t = 1}$ be the random index set of iterations where the IKR-UCB strategy is applied.
    Since $\set{I_t}_{t \ge 1}$ are independent Bernoulli trials with $\PP[I_t = 1] = q > 0$ for all $t \ge 1$, the second Borel--Cantelli lemma implies that $\card{\fT} = \sum_{t=1}^{\infty} I_t = \infty$, $\PP$-a.s., and consequently almost surely under the conditional measure $\PP[\cdot|\fE]$.

    By compactness of $\fX$, $\set{\vx_t}_{t \in \fT}$ has a convergent subsequence $\set{\vx_{t_n}}_{n \ge 1}$.
    By \cref{lem:SNEB IKR-UCB}\cref{it2:SNEB IKR-UCB}, $\tilde{a}(\vx_{t_n}; \mD_{t_n - 1}) \ptogiven{\fE} 0$ as $n \to \infty$.
    Then, the same argument as in \cref{thm:consistency BOKE}\cref{it1:consistency BOKE} yields
    \[
        \lim_{n \to \infty} \PP\left[ \set{\tilde{a}(\vx'; \mD_{t_n - 1}) > \tilde{a}(\vx_{t_n}; \mD_{t_n - 1})} \cap \fE \right] = \PP[\fE].
    \]
    However, since $\set{t_n}_{n \ge 1} \subset \fT$, the IKR-UCB criterion \eqref{eqn:max IKR-UCB} yields
    \[
        \PP\left[ \set{\tilde{a}(\vx'; \mD_{t_n - 1}) > \tilde{a}(\vx_{t_n}; \mD_{t_n - 1})} \cap \fE \right] = 0, \quad \forall n \ge 1.
    \]
    This contradiction shows $\PP[\fE] = 0$.
    Hence $\inf_{t} \filldist_{\fX, \mX_t} \le R_{\Psi} \ell$ almost surely, which proves the first conclusion for BOKE+.
    Moreover, since the proof of \cref{thm:consistency BOKE}\cref{it2:consistency BOKE} relies only on part \cref{it1:consistency BOKE} and not on the acquisition strategy, the second conclusion extends to BOKE+ as well, completing the proof.
\end{proof}

The algorithmic consistency of DE follows as a special case of BOKE by setting $f \equiv 0$ and assuming a noise-free setting (i.e., $\varsigma^2 = 0$). 
A self-contained proof of \cref{cor:consistency DE} is given in \cref{app:proof of algorithmic consistency}.
The same argument extends to BOKE and BOKE+, implying convergence of simple regret in the noise-free setting (i.e., $\lim_{t \to \infty} \SimReg_t = 0$).

\begin{corollary}[Algorithmic consistency of DE]
    \label{cor:consistency DE}
    Suppose \cref{asm:domain,asm:function,asm:kernel} hold. Let $\mX_t = \set{\vx_1, \dots, \vx_t}$ denote the DE iterates. Then,
    \begin{enumerate}[label=(\roman*)]
        \item\label{it1:consistency DE} For any fixed $\ell_t \equiv \ell$, $\fX$ is eventually populated by the DE iterates within $\bigO(\ell)$ gaps, i.e., $\inf_t \filldist_{\fX, \mX_t} \le R_{\Psi} \ell$.

        \item\label{it2:consistency DE} There exists a data-dependent bandwidth schedule $\set{\ell_t}_{t \ge 1}$, where each $\ell_t$ is determined by $\mX_{t-1}$, such that $\lim_{t \to \infty} \ell_t = 0$, and $\fX$ is eventually populated by the DE iterates, i.e., $\lim_{t \to \infty} \filldist_{\fX, \mX_t} = 0$.
    \end{enumerate}
\end{corollary}

\begin{remark}
    As an algorithmic consistency result in the noise-free setting, the key difference between \cref{cor:consistency DE} and \cite[Theorem 3.9]{chen2023pseudo} lies in our use of the fill distance to characterize the denseness of query points. This improvement leads to the statement in part \cref{cor:consistency DE}\cref{it2:consistency DE}, showing that selecting an appropriate bandwidth can achieve better space-filling performance compared to using a fixed bandwidth.
\end{remark}

\subsection{Regret analysis}
\label{sec:regret analysis}

The theoretical regret of BOKE admits two complementary analytical perspectives: kernel density and fill distance. 
While the density-based approach adapts standard bandit techniques to yield explicit bounds for finite sets, its guarantees for continuous domains remain implicitly dependent on the empirical distribution of the iterates; we therefore defer this analysis to \cref{app:regret bound by KDE}.

To obtain explicit convergence rates for general decision sets, this section focuses exclusively on the latter approach: a novel, GP-free regret analysis governed by the fill distance. 
Our core analytical strategy bridges the prediction error bounds (\cref{thm:pointwise error of KR by fill distance,thm:uniform error of KR by fill distance}) with the algorithmic consistency of BOKE (\cref{thm:consistency BOKE}). 
This coupling allows us to establish simple regret bounds and asymptotic convergence in probability under a data-dependent bandwidth schedule (\cref{thm:simple regret BOKE}); these guarantees further extend to a sublinear cumulative regret bound via the modified BOKE+ algorithm (\cref{rmk:cumulative regret BOKE+}).

Before presenting the theoretical guarantees, we formally specify the recommendation strategy. 
For our simple regret analysis, we adopt the surrogate-based empirical best arm (EBA) strategy \cite{bubeck2011pure,vakili2021optimal}. 
Specifically, this strategy recommends the maximizer of a KR estimator $\hm_t$, i.e., $\hvx^*_t \in \argmax_{\vx \in \fX} \hm_t(\vx)$.

\begin{theorem}
    \label{thm:simple regret BOKE}
    Suppose \cref{asm:domain,asm:function,asm:kernel,asm:noise,asm:bandwidth} hold.
    For any sequence $\set{\beta_t}_{t \ge 1}$ with $\beta_t = \bigOmega(\sqrt{\ln t})$, let $\mX_t = \set{\vx_1, \dots, \vx_t}$ denote the BOKE iterates. Then,
    \begin{enumerate}[label=(\roman*)]
        \item\label{it1:simple regret BOKE} For any $\gamma \in (0, 1)$, there exist constants $C_0, C_1, \ellmax > 0$ depending only on $\fX$, $\Psi$, $f$, $\varrho$, $\varsigma^2$, and $\gamma$.
        For any $\delta \in (0, 1)$ and fixed $\ell_t \equiv \ell$, if $\ell \le \ellmax$, then
        \begin{equation}
            \liminf_{t \to \infty} \PP \left[ \SimReg_t \le C_0 \left( \filldist_{\fX, \mX_t}^{\gamma d / 2} \sqrt{d \ln \left( C_1 \filldist_{\fX, \mX_t}^{-d} \right) + \ln(8 / \delta)} + \omega_f \left( \filldist_{\fX, \mX_t}^{1 - \gamma} \right) \right) \right] 
            \ge 1 - \delta.
        \end{equation}

        \item\label{it2:simple regret BOKE} There exists a data-dependent bandwidth schedule $\set{\ell_t}_{t \ge T_0}$, where each $\ell_t$ is determined by $\mD_{t-1}$, such that $\ell_t \asto 0$ and $\SimReg_t \pto 0$ as $t \to \infty$.
    \end{enumerate}
\end{theorem}

\begin{proof}
    Recall the simple regret $\SimReg_t = f(\vx^*) - f(\hvx^*_t)$ evaluated under the aforementioned EBA strategy.
    We first specify the construction of the recommendation surrogate $\hm_t$.
    Assume $\hm_t$ has bandwidth $\hell_t$ and is fitted on a set of inducing points $\hmX_t \subset \mX_t$.
    Write $\hmX_t = \set{\hvx_i}_{i=1}^{s}$ with $s \coloneqq \card{\hmX_t} \le t$.
    The set $\hmX_t$ can be constructed greedily as a maximal $\filldist_{\fX, \mX_t}$-separated set, so that $\| \hvx_i - \hvx_j \| > \filldist_{\fX, \mX_t}$ for all $i \neq j$ (see, e.g., \cite[Remark 4.2.7]{vershynin2026high}).
    Consequently, this construction yields a $2 \filldist_{\fX, \mX_t}$-net of $\fX$.
    By the separation property, the balls $B \left( \hvx_i, \filldist_{\fX, \mX_t} / 2 \right)$ are disjoint.
    Since $\filldist_{\fX, \mX_t} \le \diam(\fX)$, each ball is contained in $\fX + (\diam(\fX) / 2) \unitball_d$, where $\unitball_d$ denotes the $d$-dimensional unit ball and $+$ is the Minkowski sum.
    Comparing the volumes gives 
    \[
        s \left( \filldist_{\fX, \mX_t} / 2 \right)^d \vol(\unitball_d)
        = \vol \left( \bigcup_{i=1}^{s} B \left( \hvx_i, \filldist_{\fX, \mX_t} / 2 \right) \right)
        \le \vol \left( \fX + (\diam(\fX) / 2) \unitball_d \right).
    \]
    Hence,
    \begin{equation}\label{eqn:size of inducing points}
        s \le \frac{\vol \left( \fX + (\diam(\fX) / 2) \unitball_d \right)}{(\filldist_{\fX, \mX_t} / 2)^d \vol(\unitball_d)} \eqqcolon A_0 \filldist_{\fX, \mX_t}^{-d}.
    \end{equation}

    \proofparagraph{\cref{it1:simple regret BOKE}}
    In order to apply the prediction error bound of $\hm_t$ in terms of the fill distance $\filldist_{\fX, \mX_t}$, it is required that $\hell_t \le \hellmax$ and $\filldist_{\fX, \mX_t} \le \filldist_0 \hell_t$, where $\hellmax \in (0, \diam(\fX)]$ is a universal constant specified in \cref{thm:uniform error of KR by fill distance}.
    Set $\hell_t = \filldist_{\fX, \mX_t}^{1 - \gamma}$ with $\gamma \in (0, 1)$.
    Then the above conditions are equivalent to $\filldist_{\fX, \mX_t} \le \min \set{\hellmax^{1 / (1 - \gamma)}, \filldist_0^{1 / \gamma}}$.

    Choose a universal constant $\ellmax \in \left( 0, (2 R_{\Psi})^{-1} \min \set{1, \hellmax^{1 / (1 - \gamma)}, \filldist_0^{1 / \gamma}} \right)$ such that the function $g(z) = z^{\gamma d / 2} \sqrt{\ln (1 / z)}$ is monotonically increasing on $(0, 2 R_{\Psi} \ellmax]$.
    By \cref{thm:consistency BOKE}\cref{it1:consistency BOKE}, define the random time
    \[
        \tau = \inf\set{t \ge T_0 : \filldist_{\fX, \mX_t} < 2 R_{\Psi} \ellmax}.
    \]
    Since $\ell_t \equiv \ell \le \ellmax$, the same argument as in \cref{thm:consistency BOKE}\cref{it2:consistency BOKE} shows that $\tau$ is a stopping time for the natural filtration generated by $\mD_t$, and that $\PP[\tau < \infty] = 1$.
    The monotonicity of $\filldist_{\fX, \mX_t}$ with respect to $t$ implies that the above conditions also hold for every $t \ge \tau$.
    Conditioned on the event $\set{\tau \le t}$, applying \cref{thm:pointwise error of KR by fill distance} at the optimizer $\vx^*$, there exists $A_1 > 0$ such that for any $\delta \in (0, 1)$, with probability at least $1 - \delta / 2$,
    \[
        \left| f(\vx^*) - \hm_t(\vx^*) \right| \le A_1 \left( \hell_t^{-d/2} \filldist_{\fX, \mX_t}^{d/2} \sqrt{\ln(4 / \delta)} + \omega_f(\hell_t) \right).
    \]
    Moreover, applying \cref{thm:uniform error of KR by fill distance} at $\vx_t$, there exists $A_2 > 0$ such that with probability at least $1 - \delta / 2$,
    \[
        \left| f(\hvx^*_t) - \hm_t(\hvx^*_t) \right| \le A_2 \left( \hell_t^{-d/2} \filldist_{\fX, \mX_t}^{d/2} \sqrt{d \ln(s + 1) + \ln(8 / \delta)} + \omega_f(\hell_t) \right).
    \]
    Let $C_0 \coloneqq 2 \max\set{A_1, A_2}$ and define $\hbeta_t = (C_0 / 2) \sqrt{d \ln(s + 1) + \ln(8 / \delta)}$. 
    Since $\hm_t(\hvx^*_t) \ge \hm_t(\vx^*)$, conditioning on the intersection of the above events yields, with probability at least $1 - \delta$ under $\PP\left[ \cdot \given \tau \le t \right]$,
    \[
        \begin{aligned}
             \SimReg_t
             &= f(\vx^*) - f(\hvx^*_t) \\
             &\le \hm_t(\vx^*) + \hbeta_t \hell_t^{-d/2} \filldist_{\fX, \mX_t}^{d/2} + (C_0 / 2) \omega_f(\hell_t) - f(\hvx^*_t) \\
             &\le \hm_t(\hvx^*_t) + \hbeta_t \hell_t^{-d/2} \filldist_{\fX, \mX_t}^{d/2} + (C_0 / 2) \omega_f(\hell_t) - f(\hvx^*_t) \\
             &\le 2 \hbeta_t \hell_t^{-d/2} \filldist_{\fX, \mX_t}^{d/2} + C_0 \omega_f(\hell_t) \\
             &\le C_0 \filldist_{\fX, \mX_t}^{\gamma d / 2} \sqrt{d \ln \left( A_0 \filldist_{\fX, \mX_t}^{-d} + 1 \right) + \ln(8 / \delta)} + C_0 \omega_f \left( \filldist_{\fX, \mX_t}^{1 - \gamma} \right) \\
             &\le C_0 \left( \filldist_{\fX, \mX_t}^{\gamma d / 2} \sqrt{d \ln \left( C_1 \filldist_{\fX, \mX_t}^{-d} \right) + \ln(8 / \delta)} + \omega_f \left( \filldist_{\fX, \mX_t}^{1 - \gamma} \right) \right) \eqqcolon U_t(\delta),
        \end{aligned}
    \]
    where the fourth inequality follows from \cref{eqn:size of inducing points}, and the constant $C_1 > 0$ in the last inequality is given by $\filldist_{\fX, \mX_t} \le \filldist_{\fX, \mX_{\tau}} < 2 R_{\Psi} \ellmax$.
    By the law of total probability,
    \[
        \PP[\SimReg_t \le U_t(\delta)] 
        \ge \PP\left[ \SimReg_t \le U_t(\delta) \given \tau \le t \right] \PP[\tau \le t]
        \ge (1 - \delta) \PP[\tau \le t].
    \]
    By the continuity of probability, we have
    \[
        \lim_{t \to \infty} \PP[\tau \le t] 
        = \PP\left[ \bigcup_{t=1}^{\infty} \set{\tau \le t} \right]
        = \PP[\tau < \infty] = 1.
    \]
    Hence,
    \[
        \liminf_{t \to \infty} \PP[\SimReg_t \le U_t(\delta)] 
        \ge (1 - \delta) \lim_{t \to \infty} \PP[\tau \le t]
        = 1 - \delta.
    \]

    \proofparagraph{\cref{it2:simple regret BOKE}}    
    Let $\set{\ell_n'}_{n \ge 1} \subset (0, \ellmax]$ be a deterministic sequence chosen to satisfy \cref{asm:bandwidth}, such as $\ell_n' = \ellmax / n$ for all $n \ge 1$.
    Denote by $\fF_t$ the natural filtration generated by the dataset $\mD_t$.

    As in \cref{thm:consistency BOKE}\cref{it2:consistency BOKE}, we recursively construct the bandwidth schedule $\set{\ell_t}_{t \ge T_0}$ alongside a sequence of random times $\set{\tau_n}_{n \ge 0}$.
    Let $\tau_0 = T_0 - 1$. 
    For each $n \ge 1$, maintain the bandwidth $\ell_t = \ell'_n$ for all $t > \tau_{n-1}$ until the following stopping time is reached:
    \[
        \tau_n = \inf\set{t > \tau_{n-1} : \filldist_{\fX, \mX_t} < 2 R_{\Psi} \ell'_n}.
    \]
    It can be proved that $\set{\tau_n}_{n \ge 0}$ is a sequence of almost surely finite stopping times, then the recursively constructed bandwidth can be written as
    \[
        \ell_t = \sum_{n=1}^{\infty} \ell'_n \1_{\set{\tau_{n-1} < t \le \tau_n}}, \quad \forall t \ge T_0.
    \]
    By part \cref{it1:simple regret BOKE}, for any fixed $\delta \in (0, 1)$ and each $n \ge 1$,
    \[
        \EE\left[ \1_{\set{\SimReg_t \ge U_t(\delta)}} \1_{\set{\tau_n < t \le \tau_{n+1}}} \given \fF_{\tau_n} \right] \le \delta \EE\left[ \1_{\set{\tau_n < t \le \tau_{n+1}}} \given \fF_{\tau_n} \right], \quad \text{a.s.}
    \]
    Since $\filldist_{\fX, \mX_{\tau_n}} \le 2 R_{\Psi} \ell'_n \le 2 R_{\Psi} \ellmax < 1$ and $\filldist_{\fX, \mX_{\tau_n}}$ is $\fF_{\tau_n}$-measurable, substituting $\delta$ with $\filldist_{\fX, \mX_{\tau_n}}$ yields
    \[
        \EE\left[ \1_{\set{\SimReg_t \ge U_t(\filldist_{\fX, \mX_{\tau_n}})}} \1_{\set{\tau_n < t \le \tau_{n+1}}} \given \fF_{\tau_n} \right] \le \filldist_{\fX, \mX_{\tau_n}} \EE\left[ \1_{\set{\tau_n < t \le \tau_{n+1}}} \given \fF_{\tau_n} \right], \quad \text{a.s.}
    \]
    By the choice of $\ellmax$ in part \cref{it1:simple regret BOKE}, the mapping $z \mapsto z^{\gamma d / 2} \sqrt{d \ln(C_1 z^{-d}) + C}$ is monotonically increasing on $(0, 2 R_{\Psi} \ellmax]$ for any constant $C \ge 0$.
    For all $t \in (\tau_n, \tau_{n+1}]$ with $n \ge 1$, the upper bound can be relaxed as
    \[
        \begin{aligned}
            & U_t(\filldist_{\fX, \mX_{\tau_n}}) 
            = C_0 \left( \filldist_{\fX, \mX_t}^{\gamma d / 2} \sqrt{d \ln \left( C_1 \filldist_{\fX, \mX_t}^{-d} \right) + \ln \left(8 \filldist_{\fX, \mX_{\tau_n}}^{-1} \right)} + \omega_f \left( \filldist_{\fX, \mX_t}^{1 - \gamma} \right) \right) \\
            \le& C_0 \left( \filldist_{\fX, \mX_{\tau_n}}^{\gamma d / 2} \sqrt{d \ln \left( C_1 \filldist_{\fX, \mX_{\tau_n}}^{-d} \right) + \ln \left(8 \filldist_{\fX, \mX_{\tau_n}}^{-1} \right)} + \omega_f \left( \filldist_{\fX, \mX_{\tau_n}}^{1 - \gamma} \right) \right)
            \eqqcolon V_n.
        \end{aligned}
    \]
    Let $N(t) = \sup\set{n : \tau_n < t}$. 
    We have $\set{N(t) = n} = \set{\tau_n < t \le \tau_{n+1}}$ and hence
    \[
        \begin{aligned}
            & \PP[ \SimReg_t \ge V_{N(t)}, N(t) \ge 1 ]
            = \sum_{n=1}^{\infty} \EE \left[ \1_{\set{\SimReg_t \ge V_n}} \cdot \1_{\set{N(t) = n}} \right] \\
            =& \sum_{n=1}^{\infty} \EE \left[ \EE \left[ \1_{\set{\SimReg_t \ge V_n}} \cdot \1_{\set{N(t) = n}} \given \fF_{\tau_n} \right] \right]
            \le \sum_{n=1}^{\infty} \EE \left[ \filldist_{\fX, \mX_{\tau_n}} \EE \left[ \1_{\set{N(t) = n}} \given \fF_{\tau_n} \right] \right] \\
            =& \sum_{n=1}^{\infty} \EE \left[ \filldist_{\fX, \mX_{\tau_n}} \1_{\set{N(t) = n}} \right]
            = \EE \left[ \filldist_{\fX, \mX_{\tau_{N(t)}}} \1_{\set{N(t) \ge 1}} \right].
        \end{aligned}
    \]
    Since $\lim_{t \to \infty} N(t) = \infty$, a.s., we have $\lim_{t \to \infty} \PP[N(t) = 0] = 0$.
    Moreover, since $\filldist_{\fX, \mX_{\tau_n}} \le 2 R_{\Psi} \ell_n' \asto 0$ as $n \to \infty$, the dominated convergence theorem gives
    \[
        \lim_{t \to \infty} \PP[ \SimReg_t \ge V_{N(t)}, N(t) \ge 1 ]
        = \lim_{t \to \infty} \EE \left[ \filldist_{\fX, \mX_{\tau_{N(t)}}} \1_{\set{N(t) \ge 1}} \right] = 0.
    \]
    Since $V_{N(t)} \asto 0$, for any $\epsilon > 0$ we have
    \[
        \begin{aligned}
            &\PP[\SimReg_t \ge \epsilon] 
            \le \PP[N(t) = 0] + \PP[\SimReg_t \ge \epsilon, N(t) \ge 1] \\
            \le& \PP[N(t) = 0] + \PP[V_{N(t)} \ge \epsilon, N(t) \ge 1] + \PP[\SimReg_t \ge V_{N(t)}, N(t) \ge 1].
        \end{aligned}
    \]
    Taking $t \to \infty$ yields $\SimReg_t \pto 0$, which completes the proof.
\end{proof}

\begin{remark}\label{rmk:cumulative regret BOKE+}
    By \cite[Lemma 2]{bubeck2011pure}, any algorithm that achieves expected simple regret $\EE[\SimReg_t] = \littleo(1)$ can be converted, by reordering its allocation and recommendation steps, into an algorithm with expected cumulative regret $\EE[\CumReg_T] = \littleo(T)$.
    Hence, from \cref{thm:simple regret BOKE}\cref{it2:simple regret BOKE}, one can construct a BOKE+ algorithm with sublinear cumulative regret that alternates between IKR-UCB and pure exploitation according to a fixed switching rule.
\end{remark}

\section{Numerical results}
\label{sec:numerics}

In this section, we conduct extensive experiments across a range of synthetic and real-world optimization tasks.
Our proposed algorithms, BOKE and BOKE+, are configured with hyperparameters $\varrho = 10^{-4}$ and $q = 0.5$.
The baseline algorithms considered in our experiments including random search (RS) \cite{bergstra2012random}, covariance matrix adaptation evolution strategy (CMA-ES) \cite{hansen2001completely}, GP-EI \cite{mockus1978application,jones1998efficient} (implemented as LogEI \cite{ament2023unexpected}), GP-UCB \cite{srinivas2010gaussian}, TPE \cite{bergstra2011algorithms}, BORE \cite{tiao2021bore}, LFBO \cite{song2022general}, KR-UCB \cite{yee2016monte}, and PseudoBO \cite{chen2023pseudo}.
Among the BO methods, the acquisition functions are specified as follows: GP-UCB, KR-UCB, and BOKE use UCB; TPE and BORE use PI; while GP-EI, LFBO, and PseudoBO use EI.
In particular, CMA-ES is implemented using \texttt{pycma} \cite{hansen2019pycma}, GP-EI and GP-UCB rely on \texttt{BoTorch} \cite{balandat2020botorch}, and TPE is executed via \texttt{Optuna} \cite{akiba2019optuna}.

We use Gaussian kernel throughout all methods and all tasks.
For KR-UCB, PseudoBO, BOKE, and BOKE+, the kernel bandwidth is chosen according to Silverman's rule of thumb \cite{silverman1998density}, which yields $\ell_t \propto \left( t \cdot \left( d + 2 \right) / 4 \right)^{-\frac{1}{d + 4}}$.
The tuning parameter for GP-UCB, BOKE, and BOKE+ is set to $\beta_t \propto 1 + \sqrt{d \ln(t + 1)}$.
All algorithms are initialized with the same set of points generated by LHS \cite{mckay1979comparison} and are executed on a single CPU core to assess algorithmic runtime.
The acquisition functions are optimized using either the L-BFGS-B algorithm \cite{byrd1995limited} with multiple starting points or Sobol sampling, ensuring an identical optimization budget of $1024$ total acquisition evaluations across all BO methods.

\begin{figure}[!htb]
    \centering
    \includegraphics[width=.98\textwidth]{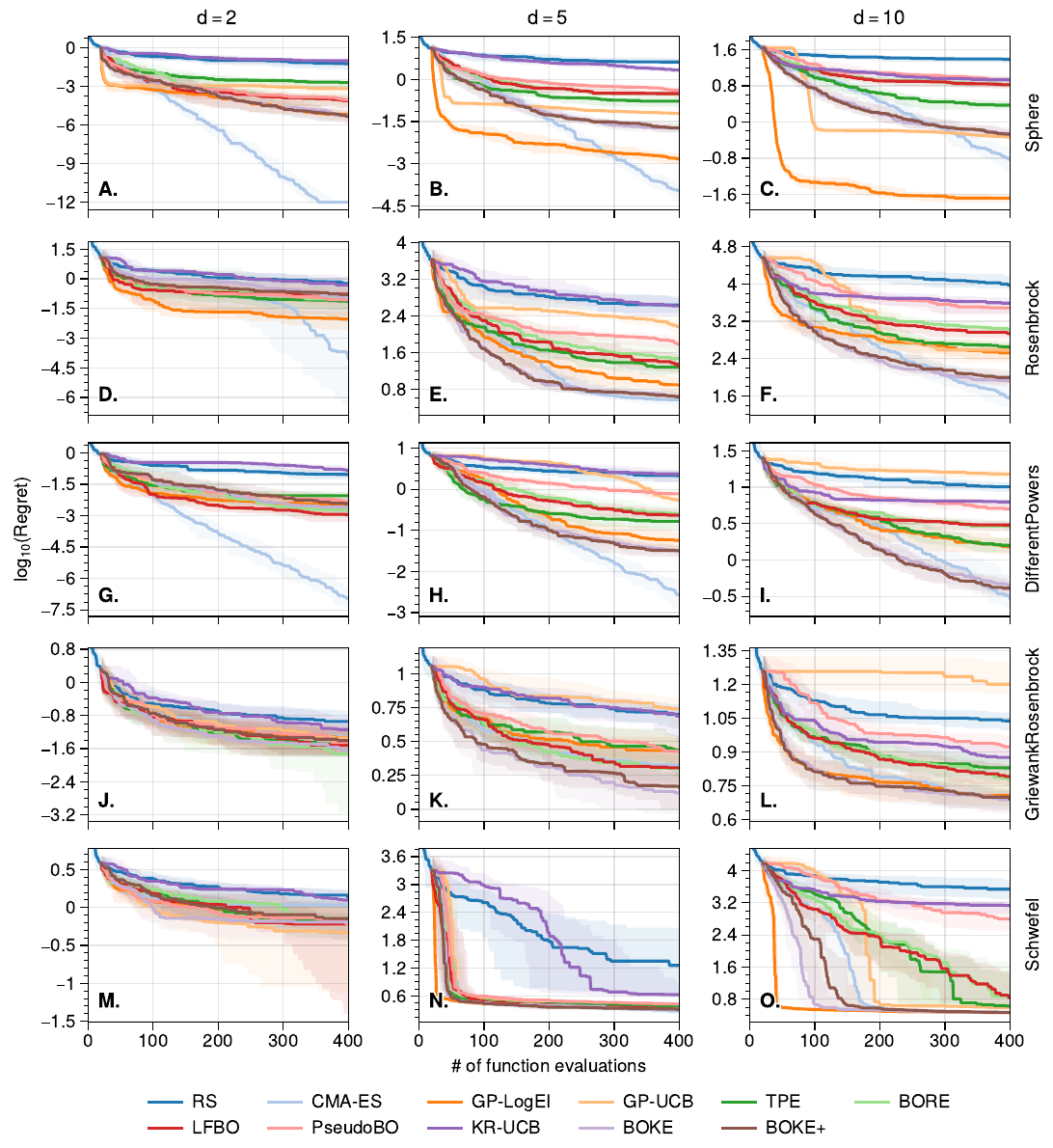}
    \caption{
        Logarithm of the simple regret and its distribution for noise-free evaluations on benchmark functions.
        Each column corresponds to a dimension $d \in \set{2, 5, 10}$; each row corresponds to a synthetic function, with the function name shown at the right of the row.
        Each plot displays the median and interquartile range over $50$ runs.
        All algorithms share the same $20$ initial sampling points generated by LHS.
    }
    \label{fig:benchmark noisefree batch 1}
\end{figure}

\subsection{Synthetic benchmark functions}
\label{sec:synthetic benchmark functions}

We empirically evaluate our method on a set of synthetic benchmark optimization tasks drawn from the black-box optimization benchmarking (BBOB) suite \cite{hansen2010comparing}, using the implementation of \texttt{IOHexperimenter} \cite{de2024iohexperimenter}.
To comprehensively evaluate algorithmic versatility across diverse landscapes while maintaining a concise presentation, we select two representative functions from each of the five official BBOB structural classes.
We present the convergence profiles for the first batch of five selected functions (one per class) in the main text, under both noise-free (\cref{fig:benchmark noisefree batch 1}) and noisy (\cref{fig:benchmark noisy batch 1}) settings. 
In the noisy setting, sampling noise is Gaussian with variance equal to $5\%$ of the signal variance.
The rows correspond to the structural classes: separable (A--C), low or moderate conditioning (D--F), high conditioning unimodal (G--I), multi-modal with adequate global structure (J--L), and multi-modal with weak global structure (M--O).
The comprehensive results for the remaining five functions (\cref{fig:benchmark noisefree batch 2,fig:benchmark noisy batch 2}) are deferred to \cref{app:additional numerics}.

\begin{figure}[!htb]
    \centering
    \includegraphics[width=.98\textwidth]{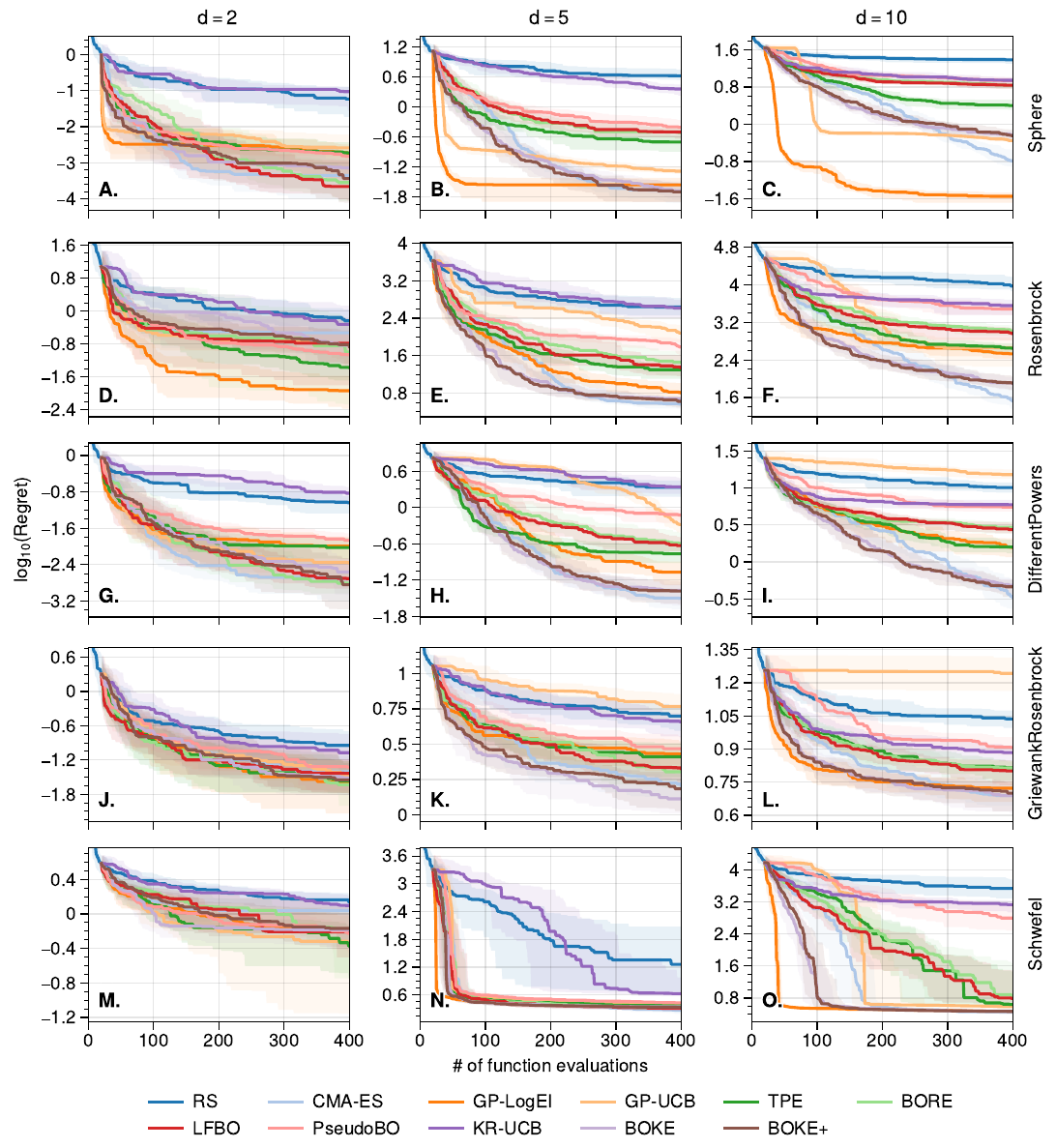}
    \caption{
        Logarithm of the simple regret and its distribution for noisy evaluations on benchmark functions.
        Each column corresponds to a dimension $d \in \set{2, 5, 10}$; each row corresponds to a synthetic function, with the function name shown at the right of the row.
        Each plot displays the median and interquartile range over $50$ runs.
        All algorithms share the same $20$ initial sampling points generated by LHS.
    }
    \label{fig:benchmark noisy batch 1}
\end{figure}

As shown in \cref{fig:benchmark noisefree batch 1,fig:benchmark noisy batch 1}, BOKE and BOKE+ achieve competitive performance compared to GP-based baselines across most scenarios and generally outperform other non-GP methods under both noise-free and noisy settings.
Among the UCB-based baselines, while KR-UCB is constrained by its local exploration strategy, BOKE employs a global exploration mechanism that effectively mitigates such limitations, yielding more robust convergence.

Nevertheless, BOKE and BOKE+ underperform in low-dimensional settings ($d = 2$) and on the Sphere function (panels A--C). 
In these scenarios, particularly under noise-free conditions, rapid convergence depends mainly on precise exploitation rather than broad exploration.
This suggests that, compared to GP-based methods, our algorithms exhibit a relatively limited exploitation capability, which is a common limitation of non-GP BO but a defining strength of GPs.
Additionally, while CMA-ES exhibits exceptionally strong asymptotic convergence on well-structured landscapes, particularly in noise-free settings (e.g., \cref{fig:benchmark noisefree batch 1}~(A--C, G--I)), its sample efficiency typically degrades as dimensionality, landscape complexity, or observation noise increases.
This performance shift underscores the robustness of BO methods.

\subsection{Sprinkler computer model}
\label{sec:sprinkler computer model}

The garden sprinkler computer model simulates the water consumption of a garden sprinkler system, taking into account factors such as the rotation speed and spray range \cite{siebertz2010statistische}.
We evaluate optimization algorithms using the implementation in the \texttt{CompModels} R package \cite{pourmohamad2021compmodels}, with a single-objective function that maximizes spray range by adjusting eight physical parameters of the sprinkler.

\begin{figure}[!htb]
    \centering
    \includegraphics[width=.98\textwidth]{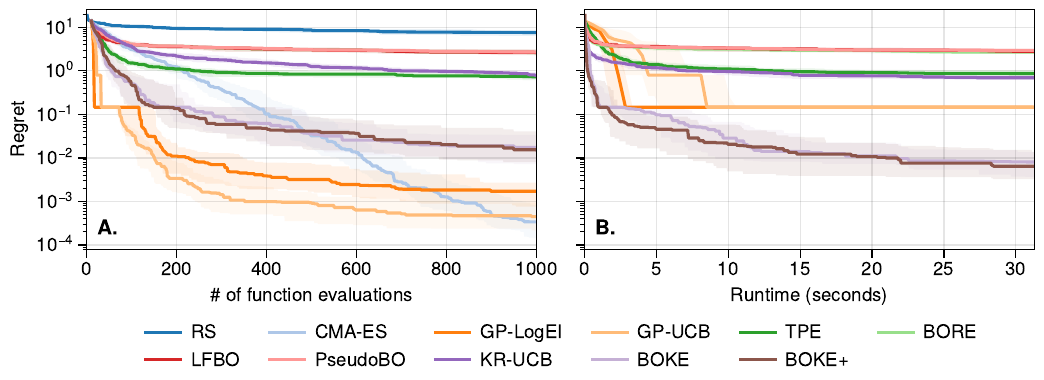}
    \caption{
        Simple regret and its distribution with respect to the number of function evaluations (panel A) and runtime (panel B) on the sprinkler computer problem (8D).
        Each plot displays the median and interquartile range over $50$ runs.
        All algorithms share the same $10$ initial sampling points generated by LHS.
    }
    \label{fig:sprinkler}
\end{figure}

Simulation results in \cref{fig:sprinkler} illustrate the algorithmic performance in terms of sample efficiency and runtime.
When plotted against the number of function evaluations (panel A), GP-based methods exhibit the highest sample efficiency. 
Although BOKE and BOKE+ ultimately lag behind GP-based methods and CMA-ES in asymptotic performance, they exhibit faster early-stage convergence than CMA-ES (prior to approximately $500$ evaluations) and maintain a competitive advantage over most other non-GP BO baselines.
However, when comparing BO algorithms by runtime (panel B; excluding RS and CMA-ES, which use paradigms distinct from the budget-constrained BO framework), BOKE and BOKE+ significantly outperform both GP-based and other non-GP methods, demonstrating their practical computational advantage for time-sensitive applications.
Notably, BOKE+ achieves greater runtime efficiency than BOKE in panel B, as the pure exploitation steps in its $\epsilon$-greedy scheme incur lower computational costs than optimizing the IKR-UCB acquisition function.

\subsection{Hyperparameter tuning}
\label{sec:hyperparameter tuning}
We empirically evaluate our method on hyperparameter optimization for random forest (RF) and XGBoost (XGB) regression tasks using datasets including California Housing \cite{pace1997sparse} and Wine Quality \cite{cortez2009modeling}.
For both problems, we utilize the implementation from the scikit-learn package and employ 5-fold cross-validation to obtain the goodness-of-fit measure, $R^2$, where higher values indicate better performance.
The parameters of interest are bounded, and for the RF problem, there are three integers: \texttt{n\_estimators}, \texttt{max\_depth}, and \texttt{min\_samples\_split}, and two real numbers: \texttt{max\_features} and \texttt{min\_impurity\_decrease}.
For the XGB problem, the hyperparameters include two integers: \texttt{n\_estimators} and \texttt{max\_depth}, and five real numbers: \texttt{learning\_rate}, \texttt{subsample}, \texttt{colsample\_bytree}, \texttt{gamma}, and \texttt{min\_child\_weight}.
To facilitate computation, we convert the integer variables to real numbers and apply rounding truncation.

\begin{figure}[!htb]
    \centering
    \includegraphics[width=.98\textwidth]{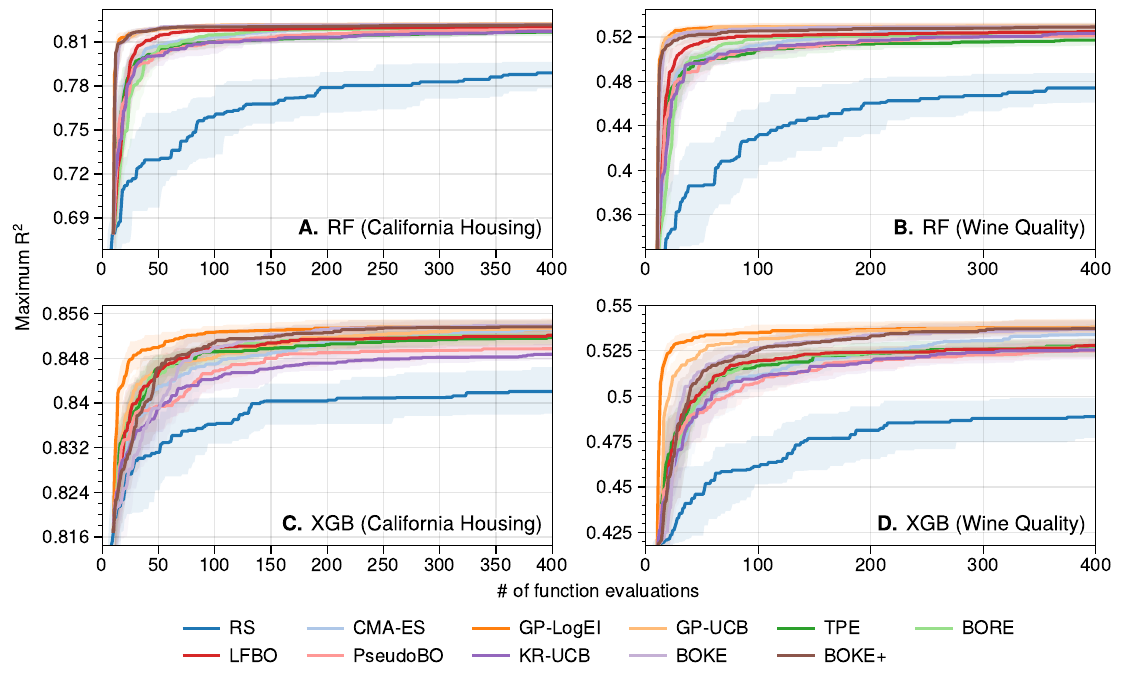}
    \caption{
        Maximum goodness-of-fit ($R^2$) and its distribution for the random forest tuning (5D; panels A and B) and XGBoost tuning (7D; panels C and D) problems.
        Each plot displays the median and interquartile range over $50$ runs.
        All algorithms share the same $10$ initial sampling points generated by LHS.
    }
    \label{fig:ml tuning}
\end{figure}

Across all hyperparameter tuning tasks, GP-based methods (particularly GP-EI) exhibit the highest sample efficiency. 
BOKE and BOKE+ follow closely, outperforming most other non-GP BO baselines, while CMA-ES shows moderate sample efficiency and lags behind these BO methods.
Overall, BOKE+ achieves better accuracy than BOKE (notably in panel C).
This is because the initial $R^2$ values are already close to their optimal values, making the objective landscapes favor exploitation.

\subsection{Runtime comparison}
\label{sec:runtime comparison}

We conclude the empirical evaluation by reporting the total wall-clock runtime of the compared algorithms in \cref{tab:runtime}.
The representative tasks selected for this comparison include the Rosenbrock synthetic functions (columns 2--4; see \cref{fig:benchmark noisefree batch 1}~(A--C)), the sprinkler computer model (column 5; see \cref{fig:sprinkler}), and the hyperparameter tuning tasks on the California Housing dataset (columns 6--7; see \cref{fig:ml tuning}~(A, C)).

\begin{table}[!htb]
    \caption{
        Comparison of average wall-clock runtimes (in seconds) across the selected tasks.
        To isolate the internal computational overhead of each optimization method, the reported times strictly exclude the objective function evaluation time. 
        Results are averaged over $50$ runs, with standard deviations provided.
    }
    \label{tab:runtime}
    \footnotesize 
    \setlength{\tabcolsep}{3pt}
    \centering
\begin{tabular}{@{}lcccccc@{}}
\toprule
\multirow{2}{*}{\textbf{Method}} & \multicolumn{3}{c}{\textbf{Synthetic}} & \multirow{2}{*}{\textbf{Sprinkler} (8D)} & \multirow{2}{*}{\textbf{RF} (5D)} & \multirow{2}{*}{\textbf{XGB} (7D)} \\ \cmidrule(lr){2-4}
 & 2D & 5D & 10D &  &  &  \\ \midrule
RS & $0.00_{ \pm 0.00 }$ & $0.00_{ \pm 0.00 }$ & $0.00_{ \pm 0.00 }$ & $0.00_{ \pm 0.00 }$ & $0.01_{ \pm 0.00 }$ & $0.00_{ \pm 0.00 }$ \\
CMA-ES & $0.04_{ \pm 0.00 }$ & $0.03_{ \pm 0.00 }$ & $0.03_{ \pm 0.00 }$ & $0.10_{ \pm 0.01 }$ & $0.07_{ \pm 0.06 }$ & $0.06_{ \pm 0.02 }$ \\
GP-EI & $156.67_{ \pm 20.90 }$ & $590.57_{ \pm 101.96 }$ & $1417.63_{ \pm 247.97 }$ & $11241.09_{ \pm 1958.88 }$ & $516.24_{ \pm 223.26 }$ & $578.26_{ \pm 188.83 }$ \\
GP-UCB & $160.62_{ \pm 27.26 }$ & $348.12_{ \pm 53.35 }$ & $590.73_{ \pm 76.54 }$ & $5032.61_{ \pm 1559.44 }$ & $365.08_{ \pm 117.77 }$ & $532.44_{ \pm 156.97 }$ \\
TPE & $10.74_{ \pm 1.23 }$ & $18.09_{ \pm 1.33 }$ & $30.41_{ \pm 1.20 }$ & $127.74_{ \pm 13.14 }$ & $17.50_{ \pm 1.30 }$ & $23.42_{ \pm 1.59 }$ \\
BORE & $6.02_{ \pm 0.71 }$ & $7.60_{ \pm 0.61 }$ & $9.96_{ \pm 0.36 }$ & $40.36_{ \pm 2.53 }$ & $5.84_{ \pm 0.42 }$ & $8.28_{ \pm 0.33 }$ \\
LFBO & $7.47_{ \pm 0.78 }$ & $9.74_{ \pm 0.58 }$ & $13.87_{ \pm 0.44 }$ & $53.65_{ \pm 2.84 }$ & $10.85_{ \pm 2.57 }$ & $12.38_{ \pm 0.94 }$ \\
PseudoBO & $9.51_{ \pm 0.59 }$ & $10.55_{ \pm 0.43 }$ & $12.41_{ \pm 0.27 }$ & $61.42_{ \pm 2.30 }$ & $15.17_{ \pm 10.77 }$ & $19.97_{ \pm 15.32 }$ \\
KR-UCB & $1.35_{ \pm 0.08 }$ & $1.70_{ \pm 0.04 }$ & $2.02_{ \pm 0.06 }$ & $15.01_{ \pm 0.54 }$ & $4.36_{ \pm 4.34 }$ & $8.15_{ \pm 13.15 }$ \\
BOKE & $1.96_{ \pm 0.06 }$ & $2.80_{ \pm 0.05 }$ & $3.31_{ \pm 0.03 }$ & $13.13_{ \pm 0.21 }$ & $3.56_{ \pm 1.48 }$ & $12.19_{ \pm 24.86 }$ \\
BOKE+ & $1.89_{ \pm 0.06 }$ & $2.72_{ \pm 0.07 }$ & $3.27_{ \pm 0.05 }$ & $13.12_{ \pm 0.37 }$ & $5.68_{ \pm 5.86 }$ & $11.14_{ \pm 7.70 }$ \\ \bottomrule
\end{tabular}
\end{table}

As expected, RS and CMA-ES exhibit substantially lower runtimes than all BO algorithms.
While KR-UCB is relatively fast for small iteration budgets due to its localized exploration, its overall $\bigO(T^3 + \Nacq T^2)$ time complexity degrades its runtime efficiency at larger scales (e.g., the sprinkler task in column 5).
Owing to their overall $\bigO(\Nacq T^2)$ complexity, BOKE and BOKE+ maintain consistently low computational overhead across all tasks, outperforming other non-GP BO baselines such as BORE, LFBO, PseudoBO, and TPE.
Finally, the exact GP-based algorithms (GP-UCB and GP-EI) prove to be the most computationally demanding methods, which is consistent with the theoretical complexity analysis in \cref{sec:computational complexity}.

\section{Conclusions}
\label{sec:conclusions}

We proposed BOKE, a Bayesian optimization algorithm that employs kernel regression as the surrogate model and kernel density estimation for exploration.
The combination of the kernel regression surrogate and the space-filling property of density-based exploration yields an effective trade-off between exploitation and exploration.
We established a rigorous analysis of the algorithm's global convergence under noisy evaluations.
Numerical experiments show that BOKE attains competitive convergence performance with GP-based and other BO baselines, while substantially reducing the overall time complexity from $\bigO(T^4 + \Nacq T^3)$ to $\bigO(\Nacq T^2)$.

While the proposed approach demonstrates competitive empirical performance and enjoys theoretical convergence guarantees, it presents several limitations that motivate promising directions for future work:
\begin{itemize}
    \item \textbf{Absence of finite-time guarantees:}
    Unlike classical GP-based analyses, the algorithmic consistency established in our theoretical framework guarantees only asymptotic convergence. 
    Extending this framework to a rigorous non-asymptotic regime remains an important open problem.
    One possible approach is to derive finite-time regret bounds by characterizing the decay rate of the fill distance, potentially building upon existing analyses of space-filling designs, such as those for the P-greedy algorithm \cite{wenzel2021novel}.

    \item \textbf{Sensitivity to hyperparameter selection:}
    The algorithm's performance is intrinsically tied to the choice of the bandwidth parameter. 
    While our experiments indicate that the currently adopted Silverman's rule of thumb effectively promotes exploration, it can oversmooth the surrogate, leading to biased KR predictions. 
    Future work will explore adaptive, data-driven bandwidth selection methods, such as cross-validation \cite{clark1977non}, to improve exploitation, alongside the necessary theoretical extensions to accommodate data-dependent bandwidths.

    \item \textbf{Lack of probabilistic grounding:}
    A fundamental limitation of replacing GPs with KR is the loss of a rigorous probabilistic framework for principled uncertainty quantification. 
    Furthermore, because KR inherently introduces smoothing bias, its exploitation capability is relatively limited compared to GPs.
    Future research should investigate alternative surrogates, such as local polynomial regression \cite{fan1996local} or scalable GPs \cite{liu2020gaussian}, to recover stronger predictive accuracy and calibrated uncertainty estimates while preserving the $\bigO(\Nacq T^2)$ computational scalability.

    \item \textbf{Computational bottlenecks for large $\Nacq$:} 
    Since maximizing the acquisition function exactly is intractable, guaranteeing sublinear regret bounds via random search requires the number of acquisition evaluations $\Nacq$ to grow with $t$ \cite{kim2025bayesian}.
    In this regime ($\Nacq = \bigOmega(t)$), BOKE's $\bigO(t)$ per-query evaluation cost emerges as a critical bottleneck.
    To mitigate this, tree-based data structures (e.g., ball-trees \cite{omohundro1989five}) can exploit the spatial locality of the kernel to accelerate KR and KDE evaluations, theoretically reducing the overall time complexity to $\bigO(\Nacq T \log T)$ and offering a compelling direction for large-scale deployment.
\end{itemize}

\section*{Acknowledgements}
L. Ji acknowledges support from Shanghai Academy of Spaceflight Technology Funded project USCAST2022-18. The authors also acknowledge the support of the HPC center of Shanghai Jiao Tong University.

\section*{Declarations}

\subparagraph{Conflict of interest.} 
The authors have no competing interests to declare that are relevant to the content of this article.

\subparagraph{Data and code availability.} 
The authors confirm that all data generated or analyzed during this study are included in this article or publicly available; see the corresponding references in this article. Codes are available from the corresponding author on reasonable request.

\appendix

\crefalias{section}{appendix}
\crefalias{subsection}{appendix}

\numberwithin{equation}{section}
\renewcommand{\theequation}{\thesection\arabic{equation}}
\renewcommand{\theHsection}{\Alph{section}}
\renewcommand{\theHequation}{\theHsection\arabic{equation}}

\section{Extended background on UCB-based baseline algorithms}
\label{app:extended background}

This section presents extended background on two UCB-based baseline algorithms: the Gaussian process upper confidence bound (GP-UCB, \cref{alg:GP-UCB}) and the kernel regression upper confidence bound (KR-UCB, \cref{alg:KR-UCB}). 
We describe their respective surrogate models, acquisition functions, and optimization procedures.

\subsection{GP-UCB}
\label{app:GP-UCB}

Gaussian processes (GPs) have been the most widely used surrogate model in the literature of Bayesian optimization, which views the objective function $f$ as a random realization of a Gaussian process. Suppose the mean function of the GP is zero, and its covariance function is $k: \fX \times \fX \to \RR$. That is, for any $t \ge 1$ and a finite set $\set{\vx_1, \dots, \vx_t} \subset \fX$, the vector $( f(\vx_1), \dots, f(\vx_t) )^\transpose$ follows a multivariate normal distribution with mean vector $\bm{0}$ and a covariance matrix whose $(i, j)$-entry is $k(\vx_i, \vx_j)$. 
Let $y_i = f(\vx_i) + \varepsilon_i$, where the additive noise terms $\varepsilon_i$ are i.i.d. random variables following $\fN(0, \sigma^2)$ (note that $\sigma^2$ is the model's assumed noise variance, distinct from the true variance proxy $\varsigma^2$). Then the posterior distribution of $f(\vx)$ at each $\vx \in \fX$ is also a normal distribution with its mean and variance
\begin{align}
    \mu_t(\vx) &= \vk_t^\transpose(\vx) (\mK_t + \sigma^2 \mI_t)^{-1} \vy_t, \label{eqn:GPR} \\
    \sigma_t^2(\vx) &= k(\vx, \vx) - \vk_t^\transpose(\vx) (\mK_t + \sigma^2 \mI_t)^{-1} \vk_t(\vx), \label{eqn:GP variance}
\end{align}
where $\mI_t$ is the $t$-dimensional identity matrix, $\vy_t = ( y_1, \dots, y_t )^\transpose \in \RR^t$, $\vk_t(\vx) = ( k(\vx, \vx_1), \dots, k(\vx, \vx_t) )^\transpose \in \RR^t$, and $\mK_t = ( k(\vx_i, \vx_j) )_{i,j=1}^{t} \in \RR^{t \times t}$.
Typically, the kernel function $k$ is chosen as a Gaussian or Mat\'ern kernel (see, e.g., \cite[Chapter 4]{rasmussen2006gaussian}).
The GP-UCB acquisition function takes the form of a weighted sum of the posterior mean and standard deviation \cite{srinivas2010gaussian}. 
The next evaluation point $\vx_{t+1}$ is selected by
\begin{equation}\label{eqn:max GP-UCB}
    \vx_{t+1} \in \argmax_{\vx \in \fX} \mu_t(\vx) + \beta_t \sigma_t(\vx).
\end{equation}
The tuning parameter $\beta_t$ balances exploitation and exploration at iteration $t$. 
Theoretical analyses typically set $\beta_t = \bigTheta(\sqrt{\ln t})$ to obtain a bound on the cumulative regret \cite{srinivas2010gaussian}.

In practice, GP regression uses the Cholesky decomposition $\mK_t + \sigma^2 \mI_t = \mL_t \mL_t^\transpose$ instead of direct matrix inversion to achieve improved numerical stability and computational efficiency.
For a dataset of size $t$, computing the Cholesky factor requires $\bigO(t^3)$ operations. 
Once cached, predicting the posterior mean $\mu_t(\vx)$ for a new test point $\vx$ takes $\bigO(t)$ time.
This is because the term $(\mK_t + \sigma^2 \mI_t)^{-1} \vy_t$ is precomputed as a fixed vector during training, reducing the evaluation to a simple inner product with $\vk_t(\vx)$. 
In contrast, computing the predictive variance $\sigma_t^2(\vx)$ involves evaluating $\| \mL_t^{-1} \vk_t(\vx) \|_2^2$, which requires solving a lower triangular linear system for $\mL_t \vv = \vk_t(\vx)$ at a cost of $\bigO(t^2)$ (see, e.g., \cite[Algorithm 2.1]{rasmussen2006gaussian}).
If the kernel hyperparameters are held fixed, the GP posterior update can be accelerated to $\bigO(t^2)$ \cite{chevalier2013corrected}; 
however, because the kernel hyperparameters are re-optimized by maximizing the marginal likelihood in standard Bayesian optimization workflows, the cost of each such re-optimization is $\bigO(t^3)$ and typically dominates the per-iteration complexity, though the amortized cost can be reduced by less frequent re-optimization as the estimation stabilizes.
The overall procedure of the GP-UCB algorithm is summarized in \cref{alg:GP-UCB}.

\begin{algorithm}[!htb]
    \caption{GP-UCB algorithm}
    \label{alg:GP-UCB}
    \KwIn{initial dataset $\mD_{T_0}$, budget $T$, kernel $k$, parameters $\set{\beta_i}_{i=T_0}^{T-1}$.}
    \For{$t = T_0, \dots, T-1$}{
        Compute $\mu_t$ and $\sigma_t$ using \cref{eqn:GPR,eqn:GP variance}. \;
        Select $\vx_{t+1}$ by solving \cref{eqn:max GP-UCB}. \;
        Evaluate $f$ at $\vx_{t+1}$ to obtain $y_{t+1}$ and update $\mD_{t+1} \gets \mD_{t} \cup \set{(\vx_{t+1}, y_{t+1})}$.
    }
\end{algorithm}

\subsection{KR-UCB}
\label{app:KR-UCB}

The KR-UCB policy stated in \cref{eqn:max KR-UCB select} is proposed as a smoothed version of the vanilla UCB1 policy \eqref{eqn:max UCB1} to select among already explored arms.
It is derived by replacing the empirical mean with the KR estimate $m_t(\vx)$ \eqref{eqn:KR}, and the discrete pull counts with the kernel density $W_t(\vx)$ \eqref{eqn:KDE}, yielding
\begin{equation}
    \label{eqn:max KR-UCB select}
    \mathring{\vx}_{t+1} \in \argmax_{\vx_i \in \mX_t} \left( m_t(\vx_i) + C \sqrt{\frac{\ln \sum_{j = 1}^{t} W_t(\vx_j)}{W_t(\vx_i)}} \right),
\end{equation}
where $C > 0$ is a tuning parameter.
To extend this approach to a general decision set $\fX$, a progressive widening strategy \cite{couetoux2011continuous} is employed. 
This strategy generalizes \cref{eqn:max KR-UCB select} by adding new points into the candidate arm set whenever the number of active arms $M$ does not exceed $t^{\eta}$, where $\eta \in (0, 1]$ is a constant controlling the widening rate.

When adding a new point, the algorithm explores the search space $\fX$ via a density-based exploration mechanism, as defined in \cref{eqn:min KR-UCB widen}.
\begin{equation}
    \label{eqn:min KR-UCB widen}
    \vx_{t+1} \in \argmin_{\vx \in \fX, k(\vx, \mathring{\vx}_{t+1}) > \kappa} W_t(\vx),
\end{equation}
where $\kappa > 0$ ensures that $\vx_{t+1}$ remains in the neighborhood of $\mathring{\vx}_{t+1}$ while avoiding regions densely populated by existing observations $\mX_t$.
Together, the discrete selection \cref{eqn:max KR-UCB select} and the continuous exploration \cref{eqn:min KR-UCB widen} constitute the complete KR-UCB allocation strategy.

The KR-UCB algorithm is summarized in \cref{alg:KR-UCB}.
The computational complexity of KR-UCB is $\bigO(t^2)$ per iteration, primarily because evaluating \cref{eqn:max KR-UCB select} requires computing pairwise kernel values among $\mX_t$.
When combined with Monte Carlo tree search (MCTS), this policy yields the KR-UCB for trees (KR-UCT) algorithm \cite{yee2016monte}, the original context for which KR-UCB was developed.

\begin{algorithm}[!htb]
    \caption{KR-UCB algorithm}
    \label{alg:KR-UCB}
    \KwIn{initial dataset $\mD_{T_0}$, budget $T$, kernel $k$, parameters $C, \kappa > 0$, and $\eta \in (0, 1]$.}
    Initialize the number of arms $M \gets T_0$. \;
    \For{$t = T_0, \dots, T-1$}{
        Compute $m_t$ and $W_t$ using \cref{eqn:KR,eqn:KDE}. \;
        Obtain $\mathring{\vx}_{t+1}$ by solving \cref{eqn:max KR-UCB select}. \;
        If $t^{\eta} < M$, set $\vx_{t+1} \gets \mathring{\vx}_{t+1}$; otherwise select $\vx_{t+1}$ by solving \cref{eqn:min KR-UCB widen} and update $M \gets M + 1$. \;
        Evaluate $f$ at $\vx_{t+1}$ to obtain $y_{t+1}$ and update $\mD_{t+1} \gets \mD_{t} \cup \set{(\vx_{t+1}, y_{t+1})}$.
    }
\end{algorithm}

\section{Regret bounds in terms of the kernel density}
\label{app:regret bound by KDE}

As discussed in \cref{sec:regret analysis}, this appendix details the density-based regret analysis of BOKE by adapting standard UCB theory \cite{lattimore2020bandit,srinivas2010gaussian}.
We first establish explicit expected cumulative regret bounds for finite decision sets, recovering the standard gap-dependent and gap-independent rates (\cref{thm:cumulative regret BOKE MAB}).
For general decision sets, we derive a high-probability cumulative regret bound (\cref{thm:cumulative regret BOKE}), demonstrating that it inherently depends on the kernel density evaluated at the query points.

\subsection{Finite decision set}

We first consider a finite decision set $X \subset \fX$ with $M \coloneqq \card{X}$ discrete arms.
For each arm $\vx \in X$, define the suboptimality gap $\Delta(\vx) = f(\vx^*) - f(\vx)$.
In \cref{thm:cumulative regret BOKE MAB}, we follow a standard approach (see, e.g., \cite[Section 8.1]{lattimore2020bandit}) and use $W_t(\vx) \ge \sum_{i=1}^{t} \1_{\set{\vx_i = \vx}}$ to show that BOKE achieves a cumulative regret bound comparable to UCB1 \eqref{eqn:max UCB1}.
If the bandwidth is sufficiently small, BOKE reduces exactly to UCB1.

\begin{theorem}
    \label{thm:cumulative regret BOKE MAB}
    Suppose \cref{asm:domain,asm:function,asm:kernel,asm:noise} hold, and every arm is played at least once during the initial phase.
    There exist constants $C_0$, $C_1$, $C_2$, $T_1$, $T_2$, $\ellmax > 0$ depending only on $X$, $\Psi$, $f$, $\varrho$, and $\varsigma^2$.
    Let $\beta_t = C_0 \sqrt{\ln (t + 1)}$.
    If $\ell_t < \ellmax$ for all $t \ge T_1$, then,
    \begin{equation}
        \EE[ \CumReg_T ] \le C_1 \min\set{\sqrt{M T \ln T}, \frac{M \ln T}{\Delta_{\min}}} + C_2 \sum_{\vx \in X} \Delta(\vx), \quad \forall T \ge T_2,
    \end{equation}
    where $\Delta_{\min} \coloneqq \min_{\vx \in X : \Delta(\vx) > 0} \Delta(\vx)$ is the minimal suboptimality gap.
\end{theorem}

\begin{proof}
    For each arm $\vx \in X$, let $n_t(\vx) = \sum_{i=1}^{t} \1_{\set{\vx_i = \vx}}$ denote the number of times $\vx$ has been pulled up to iteration $t$.
    After the initial phase, $n_{T_0}(\vx) \ge 1$ for all $\vx \in X$.
    The expected cumulative regret can be written as
    \[
        \EE[\CumReg_T] = \EE \left[ \sum_{t=1}^{T} ( f(\vx^*) - f(\vx_t) ) \right] = \sum_{\vx \in X} \Delta(\vx) \EE[n_T(\vx)].
    \]
    Since $\Delta(\vx^*) = 0$, it suffices to consider suboptimal arms $\vx$ with $\Delta(\vx) > 0$.
    Fix a constant $\epsilon > 0$ to be chosen later.
    We decompose $n_T(\vx)$ as
    \[
        \begin{aligned}
            n_T(\vx) 
            &= \sum_{t=0}^{T-1} \1\set{\vx_{t+1} = \vx} \\
            &\le \sum_{t=0}^{T-1} \1\set{f(\vx^*) \ge m_t(\vx^*) + A_0 \left( \hsigma_t(\vx^*) \sqrt{\ln(2 / \delta_t)} + \omega_f(\ell_t) + \epsilon \right)} \\
            &\phantom{=} + \sum_{t=0}^{T-1} \1\set{f(\vx^*) \le m_t(\vx) + A_0 \left( \hsigma_t(\vx) \sqrt{\ln(2 / \delta_t)} + \omega_f(\ell_t) + \epsilon \right), \vx_{t+1} = \vx},
        \end{aligned}
    \]
    where $A_0 > 0$ corresponds to the universal constant defined in \cref{eqn:pointwise error of KR by KDE}.

    For the first sum, taking expectation and applying \cref{thm:pointwise error of KR by KDE} at $\vx^*$ yields
    \[
        \begin{aligned}
            J_1 
            &\coloneqq \EE\left[ \sum_{t=0}^{T-1} \1\set{f(\vx^*) \ge m_t(\vx^*) + A_0 \left( \hsigma_t(\vx^*) \sqrt{\ln(2 / \delta_t)} + \omega_f(\ell_t) + \epsilon \right)} \right] \\
            &\le \sum_{t=0}^{T-1} \sum_{s=0}^{T-1} \PP\left[ f(\vx^*) - m_t(\vx^*) \ge A_0 \left( \hsigma_t(\vx^*) \sqrt{\ln(2 / \delta_t)} + \omega_f(\ell_t) + \epsilon \right) \given n_t(\vx^*) = s \right] \\
            &\le \sum_{t=0}^{T-1} \sum_{s=0}^{T-1} 2 \exp \left. \left( - \hsigma_t^{-2}(\vx^*) \left( \hsigma_t(\vx^*) \sqrt{\ln(2 / \delta_t)} + \epsilon \right)^2 \right) \right|_{n_t(\vx^*) = s} \\
            &\le \sum_{t=0}^{T-1} \sum_{s=0}^{T-1} 2 \exp \left. \left( \ln(\delta_t / 2) - \hsigma_t^{-2}(\vx^*) \epsilon^2 \right) \right|_{n_t(\vx^*) = s} \\
            &\le \sum_{t=0}^{T-1} \sum_{s=0}^{T-1} 2 \exp \left( \ln(\delta_t / 2) - \epsilon^2 s \right)
            = \sum_{t=0}^{T-1} \delta_t \sum_{s=0}^{T-1} e^{- \epsilon^2 s}
            \le \left(1 + \frac{1}{\epsilon^2} \right) \sum_{t=0}^{T-1} \delta_t.
        \end{aligned}
    \]
    The first inequality follows from the union bound over all possible values of $n_t(\vx^*)$. 
    The fourth inequality uses $\hsigma_t^{-2}(\vx^*) = W_t(\vx^*) + \varrho \ge n_t(\vx^*)$.
    The final inequality follows from $\sum_{s=0}^{T-1} e^{- \epsilon^2 s} \le 1 / (1 - e^{- \epsilon^2})$ and $e^{\epsilon^2} \ge \epsilon^2 + 1$.

    For the second sum, assume that $\delta_t$ is non-increasing in $t$.
    Taking expectation yields
    \[
        \begin{aligned}
            J_2 
            &\coloneqq \EE\left[ \sum_{t=0}^{T-1} \1\set{f(\vx^*) \le m_t(\vx) + A_0 \left( \hsigma_t(\vx) \sqrt{\ln(2 / \delta_t)} + \omega_f(\ell_t) + \epsilon \right), \vx_{t+1} = \vx} \right] \\
            &\le \EE\left[ \sum_{t=0}^{T-1} \1\set{f(\vx) - m_t(\vx) \le A_0 \left( \hsigma_t(\vx) \sqrt{\ln(2 / \delta_{T-1})} + \omega_f(\ell_t) + \epsilon \right) - \Delta(\vx), \vx_{t+1} = \vx} \right] \\
            &\le \EE\left[ \sum_{t=0}^{T-1} \1\set{f(\vx) - m_t(\vx) \le A_0 \left( \hsigma_t(\vx) \sqrt{\ln(2 / \delta_{T-1})} + \omega_f(\ell_t) + \epsilon \right) - \Delta(\vx)} \given n_t(\vx) = t \right] \\
            &\le \sum_{t=0}^{T-1} \PP\left[ f(\vx) - m_t(\vx) \le A_0 \left( t^{-1/2} \sqrt{\ln(2 / \delta_{T-1})} + \omega_f(\ell_t) + \epsilon \right) - \Delta(\vx) \given n_t(\vx) = t \right] \\
            &\le u + \sum_{t = \lceil u \rceil}^{T-1} \PP\left[ f(\vx) - m_t(\vx) \le A_0 \left( t^{-1/2} \sqrt{\ln(2 / \delta_{T-1})} + \omega_f(\ell_t) + \epsilon \right) - \Delta(\vx) \given n_t(\vx) = t \right].
        \end{aligned}
    \]
    The second inequality holds because each arm can be selected at most $n_T(\vx) \le T$ times.
    The third inequality uses $\hsigma_t^{-2}(\vx) \ge n_t(\vx)$.
    The role of constant $u > 0$ in the last line is to guarantee
    \[
        A_0^{-1} \Delta(\vx) - t^{-1/2} \sqrt{\ln(2 / \delta_{T-1})} - \omega_f(\ell_t) - \epsilon \ge 0, \quad \forall t \ge \lceil u \rceil.
    \]
    Since $\omega_f(\ell_t) \le \omega_f(\ellmax)$ for all $t \ge T_1$, it is required that
    \[
        \ellmax < \inf \set{\ell > 0 : \omega_f(\ell) \le A_0^{-1} \Delta_{\min}},
    \]
    where $\Delta_{\min} = \min_{\vx \in X : \Delta(\vx) > 0} \Delta(\vx)$.
    Hence, it suffices to choose
    \[
        u \ge \max\set{T_1, \frac{\ln(2 / \delta_{T-1})}{\left( A_0^{-1} \Delta(\vx) - \omega_f(\ellmax) - \epsilon \right)^2}}.
    \]
    Applying \cref{thm:pointwise error of KR by KDE} at $\vx$, the sum is further bounded by
    \[
        \begin{aligned}
            J_2 &\le u + \sum_{t = \lceil u \rceil}^{T-1} 2 \exp \left. \left( - \hsigma_t^{-2}(\vx) \left( A_0^{-1} \Delta(\vx) - t^{-1/2} \sqrt{\ln(2 / \delta_{T-1})} - \omega_f(\ell_t) - \epsilon \right)^2 \right) \right|_{n_t(\vx) = t} \\
            &\le u + \sum_{t = \lceil u \rceil}^{T-1} 2 \exp \left( - \left( \sqrt{t} \left( A_0^{-1} \Delta(\vx) - \omega_f(\ellmax) - \epsilon \right) - \sqrt{\ln(2 / \delta_{T-1})} \right)^2 \right) \\
            &\le 2 + u + \int_{u}^{\infty} 2 \exp \left( - \left( \sqrt{t} \left( A_0^{-1} \Delta(\vx) - \omega_f(\ellmax) - \epsilon \right) - \sqrt{\ln(2 / \delta_{T-1})} \right)^2 \right) \dif t \\
            &\le 2 + \frac{\ln(2 / \delta_{T-1}) + 2 \sqrt{\pi \ln(2 / \delta_{T-1})} + 2}{\left( A_0^{-1} \Delta(\vx) - \omega_f(\ellmax) - \epsilon \right)^2}
            \le 2 + \frac{2 \ln(2 / \delta_{T-1}) + 6}{\left( A_0^{-1} \Delta(\vx) - \omega_f(\ellmax) - \epsilon \right)^2}.
        \end{aligned}
    \]
    The second inequality $\hsigma_t^{-2}(\vx) \ge n_t(\vx)$ and $\omega_f(\ell_t) \le \omega_f(\ellmax)$ for $t \ge \lceil u \rceil$.
    The last line follows from the change of variable
    \[
        s = \sqrt{t} \left( A_0^{-1} \Delta(\vx) - \omega_f(\ellmax) - \epsilon \right) - \sqrt{\ln(2 / \delta_{T-1})},
    \]
    and by taking
    \[
        u = \frac{\ln(2 / \delta_{T-1})}{\left( A_0^{-1} \Delta(\vx) - \omega_f(\ellmax) - \epsilon \right)^2} 
        \ge \frac{\ln(2 / \delta_{T_2 - 1})}{\left( A_0^{-1} \Delta(\vx) - \omega_f(\ellmax) - \epsilon \right)^2} 
        \ge T_1,
    \]
    which holds for sufficiently large $T_2$.
    Let $\delta_t = 2 / (t+1)^2$, so that $\beta_t \coloneqq A_0 \sqrt{\ln (2 / \delta_t)} = A_0 \sqrt{2 \ln (t + 1)}$.
    Define $A_1 \coloneqq (2 A_0)^{-1} - \omega_f(\ellmax) / (2 \Delta_{\min}) > 0$, and set $\epsilon = A_1 \Delta(\vx)$.
    Then, $A_0^{-1} \Delta(\vx) - \omega_f(\ellmax) - \epsilon \ge A_1 \Delta(\vx)$.
    Let $\Delta_{\max} = \max_{\vx \in X} \Delta(\vx)$.
    If
    \[
        T_2 \ge \exp \left( \left( (A_0^{-1} - A_1) \Delta_{\max} - \omega_f(\ellmax) \right)^2 T_1 / 2 \right),
    \]
    we obtain
    \[
        J_2 \le 2 + \frac{2 \ln(2 / \delta_{T-1}) + 6}{\left( A_0^{-1} \Delta(\vx) - \omega_f(\ellmax) - \epsilon \right)^2}
        \le 2 + \frac{4 \ln T + 6}{A_1^2 \Delta^2(\vx)}.
    \]

    Finally, since $T \ge T_0 > 1$, combining the two parts gives
    \[
        \EE[ n_T(\vx) ] 
        \le J_1 + J_2
        \le \frac{\pi^2}{3} \left(1 + \frac{1}{A_1^2 \Delta^2(\vx)} \right) + 2 + \frac{4 \ln T + 6}{A_1^2 \Delta^2(\vx)}
        \le A_2 \left( 1 + \frac{\ln T}{\Delta^2(\vx)} \right),
    \]
    where $A_2 > 0$ is a universal constant.
    Hence, 
    \[
        \EE[\CumReg_T] 
        \le \sum_{\vx \in X : \Delta(\vx) > 0} A_2 \left( \Delta(\vx) + \frac{\ln T}{\Delta(\vx)} \right)
        \le \frac{A_2 M \ln T}{\Delta_{\min}} + A_2 \sum_{\vx \in X} \Delta(\vx).
    \]
    Moreover, for any $\Delta_0 > 0$,
    \[
        \begin{aligned}
            \EE[\CumReg_T]
            &= \sum_{\vx \in X : \Delta(\vx) < \Delta_0} \Delta(\vx) \EE[n_T(\vx)] + \sum_{\vx \in X : \Delta(\vx) \ge \Delta_0} \Delta(\vx) \EE[n_T(\vx)] \\
            &\le \Delta_0 T + \sum_{\vx \in X : \Delta(\vx) \ge \Delta_0} A_2 \left( \Delta(\vx) + \frac{\ln T}{\Delta(\vx)} \right) \\
            &\le \Delta_0 T + \frac{A_2 M \ln T}{\Delta_0} + A_2 \sum_{\vx \in X} \Delta(\vx) \\
            &\le 2 \sqrt{A_2 M T \ln T} + A_2 \sum_{\vx \in X} \Delta(\vx),
        \end{aligned}
    \]
    where the first inequality uses $\sum_{\vx \in X} \EE[n_T(\vx)] = T$, and the last inequality follows by choosing $\Delta_0 = \sqrt{A_2 M \ln T / T}$.
    The proof is completed by setting $C_1 = \max\set{ A_2, 2 \sqrt{A_2} }$ and $C_2 = A_2$.
\end{proof}

\subsection{General decision set}

For a general decision set $\fX$, the lack of a probabilistic prior (e.g., GPs) complicates the characterization of asymptotic regret rates.
Consequently, unlike GP-UCB (see, e.g., \cite[Theorem 2 and 3]{srinivas2010gaussian}), we cannot infer from \cref{thm:cumulative regret BOKE} that $\CumReg_T = \littleo(T)$ (i.e., that the algorithm achieves sublinear regret).
Nevertheless, analysis of cumulative regret remains useful for algorithm design, for example in selecting $\beta_t$.

\begin{theorem}
    \label{thm:cumulative regret BOKE}
    Suppose \cref{asm:domain,asm:function,asm:kernel,asm:noise,asm:bandwidth} hold.
    There exist a constant $C_0 > 0$ depending only on $\fX$, $\Psi$, $f$, $\varrho$, and $\varsigma^2$.
    For any $\delta \in (0, 1)$, let $\beta_t = C_0 ( 1 + \sqrt{d \ln(t + 1) + \ln(2 \pi^2 t^2 / \delta)} )$.
    Then, with probability at least $1 - \delta$,
    \begin{equation}
        \CumReg_T \le 4 \sqrt{T \beta_T^2 \hvarSigma_T} + 2 C_0 \varOmega_T + 2 T_0 \| f \|_{\infty}, \quad \forall T \ge T_0,
    \end{equation}
    where $\hvarSigma_T \coloneqq \sum_{t=T_0}^{T-1} \left( \hsigma_t^2(\vx_{t+1}) + t^{-2} W^{-2}_t(\vx_{t+1}) \right)$ and $\varOmega_T \coloneqq \sum_{t=T_0}^{T-1} \omega_f(\ell_t)$.
\end{theorem}

\begin{remark}
    Under additional assumptions on the decay of the fill distance, one can obtain sublinear regret for BOKE.
    For example, let $\ell_t = \bigTheta(t^{-\alpha})$ with $\alpha > 0$. 
    If $f$ is Lipschitz, then $\varOmega_T \lesssim \sum_{t=T_0}^{T-1} \ell_t \lesssim \sum_{t=1}^{T} t^{-\alpha} \lesssim T^{1 - \alpha}$. 
    Assume $\filldist_{\fX, \mX_t} \lesssim \ell_t$ and decays at the near-optimal rate $\bigO(t^{-\frac{1}{d} + \epsilon})$ for some $\epsilon > 0$. 
    Then \cref{lem:bound KDE by fill distance} implies $W_t(\vx_{t+1}) \gtrsim \ell_t^d \filldist_{\fX, \mX_t}^{-d} \gtrsim t^{1 - (\epsilon + \alpha) d}$, and consequently $\hvarSigma_T \lesssim \sum_{t=T_0}^{T-1} W^{-1}_t(\vx_{t+1}) \lesssim \sum_{t=1}^{T} t^{(\epsilon + \alpha) d - 1} \lesssim T^{(\epsilon + \alpha) d}$.
    Therefore, a sufficient condition for $\CumReg_T = \littleo(T)$ is $\epsilon + \alpha < \frac{1}{d}$.
\end{remark}

\begin{proof}
    Applying \cref{thm:pointwise error of KR by KDE} at the optimizer $\vx^*$, there exists $A_0 > 0$ such that for any $\delta \in (0, 1)$, with probability at least $1 - \delta$, if $W_t(\vx^*) > 0$ then
    \[
        \left| f(\vx^*) - m_t(\vx^*) \right| 
        \le A_0 \left( \hsigma_t(\vx^*) \sqrt{\ln(2 / \delta)} + \omega_f(\ell_t) \right).
    \]
    If $W_t(\vx^*) = 0$, then $\hsigma_t(\vx^*) = \varrho^{-1/2}$, and by \cref{lem:uniform error of extended KR},
    \[
        \begin{aligned}
            \left| f(\vx^*) - m_t(\vx^*) \right| 
            &\le 2 \| f \|_{\infty} + \sqrt{2 t^{-1} \varsigma^2 \ln(2 / \delta)} \\
            &\le \hsigma_t(\vx^*) \cdot \sqrt{\varrho} \left( 2 \| f \|_{\infty} + \sqrt{2 \varsigma^2 \ln(2 / \delta)} \right).
        \end{aligned}
    \]
    Hence, there exists $A_1 > 0$ such that
    \[
        \PP \left[ \left| f(\vx^*) - m_t(\vx^*) \right| \le A_1 \left( \hsigma_t(\vx^*) \left( 1 + \sqrt{\ln(2 / \delta)} \right) + \omega_f(\ell_t) \right) \right] \ge 1 - \delta.
    \]
    Set $\pi_t = \pi^2 t^2 / 6$, so $\sum_{t=1}^{\infty} 1 / \pi_t = 1$.
    By the union bound, with probability at least $1 - \delta / 2$, for all $t \ge T_0$,
    \[
        \left| f(\vx^*) - m_t(\vx^*) \right| \le A_1 \left( \hsigma_t(\vx^*) \left( 1 + \sqrt{\ln(4 \pi_t / \delta)} \right) + \omega_f(\ell_t) \right).
    \]
    Applying the union bound together with \cref{thm:uniform error of KR by KDE} to the iterates $\set{\vx_t}_{t > T_0}$, there exists $A_2 > 0$ such that with probability at least $1 - \delta / 2$, for all $t \ge T_0$,
    \[
        \begin{aligned}
            \left| f(\vx_{t+1}) - m_t(\vx_{t+1}) \right|
            \le A_2 \left( \left( \hsigma_t(\vx_{t+1}) + t^{-1} W^{-1}_t(\vx_{t+1}) \right) \sqrt{d \ln(t + 1) + \ln(8 \pi_t / \delta)} + \omega_f(\ell_t) \right).
        \end{aligned}
    \]
    Let $C_0 \coloneqq \max\set{A_1, A_2}$ and define $\beta_t = C_0 ( 1 + \sqrt{d \ln(t + 1) + \ln(2 \pi^2 t^2 / \delta)} )$. 
    Conditioning on the intersection of the above events yields, with probability at least $1 - \delta$, for all $t \ge T_0$,
    \[
        \begin{aligned}
            \left| f(\vx^*) - m_t(\vx^*) \right| &\le \beta_t \hsigma_t(\vx^*) + C_0 \omega_f(\ell_t), \\
            \left| f(\vx_{t+1}) - m_t(\vx_{t+1}) \right| &\le \beta_t \left( \hsigma_t(\vx_{t+1}) + t^{-1} W^{-1}_t(\vx_{t+1}) \right) + C_0 \omega_f(\ell_t).
        \end{aligned}
    \]
    Using $a_t(\vx_{t+1}) \ge a_t(\vx^*)$ and the two bounds above gives
    \[
        \begin{aligned}
             r_{t+1} 
             &= f(\vx^*) - f(\vx_{t+1}) \\
             &\le m_t(\vx^*) + \beta_t \hsigma_t(\vx^*) + C_0 \omega_f(\ell_t) - f(\vx_{t+1}) \\
             &\le m_t(\vx_{t+1}) + \beta_t \hsigma_t(\vx_{t+1}) + C_0 \omega_f(\ell_t) - f(\vx_{t+1}) \\
             &\le 2 \beta_t \hsigma_t(\vx_{t+1}) + \beta_t t^{-1} W^{-1}_t(\vx_{t+1}) + 2 C_0 \omega_f(\ell_t).
        \end{aligned}
    \]
    Summing from $t = T_0$ to $T - 1$ yields
    \[
        \begin{aligned}
            &\CumReg_T - \CumReg_{T_0} = \sum_{t=T_0+1}^{T} r_t \\
            \le& 2 \sum_{t=T_0}^{T-1} \beta_t \hsigma_t(\vx_{t+1}) + \sum_{t=T_0}^{T-1} \beta_t t^{-1} W^{-1}_t(\vx_{t+1}) + 2 C_0 \sum_{t=T_0}^{T-1} \omega_f(\ell_t) \\
            \le& 4 \sqrt{T \beta_T^2 \sum_{t=T_0}^{T-1} \left( \hsigma_t^2(\vx_{t+1}) + t^{-2} W^{-2}_t(\vx_{t+1}) \right)} + 2 C_0 \sum_{t=T_0}^{T-1} \omega_f(\ell_t),
        \end{aligned}
    \]
    where the last line follows from Cauchy--Schwarz and the monotonicity of $t \mapsto \beta_t$.
    Finally, note
    \[
        \CumReg_{T_0} \le \sum_{t=1}^{T_0} \left| f(\vx^*) - f(\vx) \right| \le 2 T_0 \| f \|_{\infty},
    \]
    which completes the proof.
\end{proof}

\section{Omitted proofs}
\label{app:proofs}

\subsection{Vanishing-bandwidth limits of Gaussian processes}
\label{app:vanishing bandwidth limits}

Before analyzing the limiting behavior as $\ell \to 0^+$, we transform \cref{eqn:GPR,eqn:GP variance} so that the covariance matrix $\mK_t$ is nonsingular (see \cref{lem:reduced GP}).
Such singularities can arise for finite decision sets and complicate the evaluation of the limit $\lim_{\ell \to 0^+} (\mK_t + \sigma^2 \mI_t)^{-1}$.

\begin{lemma}
    \label{lem:reduced GP}
    Let $f \sim \mathcal{GP}( \bm{0}, k(\cdot, \cdot) )$ be a Gaussian process. 
    For any $\mD_t = \set{(\vx_i, y_i)}_{i=1}^{t}$, let $\bar{\mX}_s = \set{\bar{\vx}_1, \dots, \bar{\vx}_s}$ represent the set of $s$ unique locations in $\mX_t$.
    For each $\bar{\vx}_{\lambda}$ with $\lambda \in \set{1, \dots, s}$, define its corresponding index set as $\fI_{\lambda} = \set{1 \le i \le t : \vx_i = \bar{\vx}_{\lambda}}$, and let $n_{\lambda} = \card{\fI_{\lambda}}$ be the number of its occurrences. 
    Then $\sum_{\lambda=1}^{s} n_{\lambda} = t$. 
    For any $\vx \in \fX$, the posterior statistics of $f(\vx)$ defined in \cref{eqn:GPR,eqn:GP variance} can be equivalently rewritten as
    \begin{align}
        \mu_t(\vx) & = \bar{\vk}_s^\transpose(\vx) (\bar{\mK}_s + \bar{\bm{\Sigma}}_s)^{-1} \bar{\vy}_s, \label{eqn:reduced GPR} \\
        \sigma^2_t(\vx) & = k(\vx, \vx) - \bar{\vk}_s^\transpose(\vx) (\bar{\mK}_s + \bar{\bm{\Sigma}}_s)^{-1} \bar{\vk}_s(\vx), \label{eqn:reduced GP variance}
    \end{align}
    where $\bar{\vk}_s(\vx) = ( k(\vx, \bar{\vx}_1), \dots, k(\vx, \bar{\vx}_s) )^\transpose$, $\bar{\mK}_s = ( k(\bar{\vx}_{\lambda}, \bar{\vx}_{\nu}) )_{\lambda, \nu=1}^{s}$, $\bar{\bm{\Sigma}}_s = \sigma^2 \diag\set{1/n_1, \dots, 1/n_s}$, and $\bar{\vy}_s = ( \bar{y}_1, \dots, \bar{y}_s )^\transpose$ with coordinates $\bar{y}_{\lambda} = \frac{1}{n_{\lambda}} \sum_{i \in \fI_{\lambda}} y_i$.
\end{lemma}

\begin{proof}
    Let $\mX = \set{\vx_1, \dots, \vx_n} \subset \fX$ be arbitrary $n$ sample points.
    Denote $f(\mX) = ( f(\vx_1), \dots, f(\vx_n) )^\transpose \in \RR^n$, $k(\mX, \vx) = k(\vx, \mX)^\transpose = ( k(\vx, \vx_1), \dots, k(\vx, \vx_n) )^\transpose \in \RR^n$, and $k(\mX, \mX) = ( k(\vx_i, \vx_j) )_{i,j=1}^{n} \in \RR^{n \times n}$. 
    Consider the following two cases:
    \[
        \begin{aligned}
            \left.\begin{bmatrix}
                f(\mX) \\ y_a \\ y_b
            \end{bmatrix} \right| \mX, \vx 
            &\sim \fN\left( \bm{0},
            \begin{bmatrix}
                k(\mX, \mX) & k(\mX, \vx) & k(\mX, \vx) \\
                k(\vx, \mX) & k(\vx, \vx) + \frac{\sigma^2}{n_a} & k(\vx, \vx) \\
                k(\vx, \mX) & k(\vx, \vx) & k(\vx, \vx) + \frac{\sigma^2}{n_b}
            \end{bmatrix}
            \right), \\
            \left.\begin{bmatrix}
                f(\mX) \\ \frac{n_a y_a + n_b y_b}{n_a + n_b}
            \end{bmatrix} \right| \mX, \vx 
            &\sim \fN\left( \bm{0},
            \begin{bmatrix}
                k(\mX, \mX) & k(\mX, \vx) \\
                k(\vx, \mX) & k(\vx, \vx) + \frac{\sigma^2}{n_a + n_b}
            \end{bmatrix}
            \right),
        \end{aligned}
    \]
    where $y_a$ and $y_b$ are the sample averages of independent observations of $f$ at the same location $\vx$, and $n_a, n_b \ge 1$ are the respective numbers of observations.
    By direct calculation, these two cases lead to the same Gaussian posterior for $f(\mX)$ with
    \[
        \begin{aligned}
            \EE \left[ f(\mX) \given \mX, \vx, y_a, y_b \right]
            &= \frac{n_a y_a + n_b y_b}{(n_a + n_b) k(\vx, \vx) + \sigma^2} k(\mX, \vx), \\
            \Var \left[ f(\mX) \given \mX, \vx, y_a, y_b \right]
            &= k(\mX, \mX) - \frac{n_a + n_b}{(n_a + n_b) k(\vx, \vx) + \sigma^2} k(\mX, \vx) k(\vx, \mX).
        \end{aligned}
    \]
    Thus, any two values $y_a$ and $y_b$ at the same location $\vx$ with occurrence numbers $n_a, n_b \ge 1$ can be reduced to their weighted average $\frac{n_a y_a + n_b y_b}{n_a + n_b}$ with a combined occurrence number of $n_a + n_b$. 
    Since the original dataset consists of individual observations with initial occurrence numbers of $1$, repeatedly applying this reduction rule to any overlapping points eventually merges all observations into the unique locations $\bar{\mX}_s$. 
    This yields the reduced posterior formulations in \cref{eqn:reduced GPR,eqn:reduced GP variance}, completing the proof.
\end{proof}

Based on \cref{lem:reduced GP}, we derive the limiting behavior of GP (\cref{pro:degenerate GP}) and use it to analyze the relationship between the GP posterior variance $\sigma_t^2$ and the density-based uncertainty function $\hsigma_t^2$ (\cref{rmk:degenerate GP vs KDE}).

\begin{proposition}
    \label{pro:degenerate GP}
    Suppose $\Psi$ is continuous at the origin. For any $\mD_t = \set{(\vx_i, y_i)}_{i=1}^{t}$ and $\vx \in \fX$,
    \begin{align}
        \lim_{\ell \to 0^+} \mu_t(\vx) &= \frac{\sum_{i=1}^{t} y_i \1_{\set{\vx = \vx_i}}}{\sigma^2 / \Psi(\bm{0}) + \sum_{i=1}^{t} \1_{\set{\vx = \vx_i}}}, \label{eqn:degenerate GPR} \\
        \lim_{\ell \to 0^+} \sigma^2_t(\vx) &= \frac{\sigma^2}{\sigma^2 / \Psi(\bm{0}) + \sum_{i=1}^{t} \1_{\set{\vx = \vx_i}}}. \label{eqn:degenerate GP variance}
    \end{align}
\end{proposition}

\begin{proof}
    By \cref{lem:reduced GP}, the GP posterior can be rewritten as
    \[
        \begin{aligned}
            \mu_t(\vx) &= \bar{\vk}_s^\transpose(\vx) (\bar{\mK}_s + \bar{\bm{\Sigma}}_s)^{-1} \bar{\vy}_s, \\
            \sigma^2_t(\vx) &= \Psi(\bm{0}) - \bar{\vk}_s^\transpose(\vx) (\bar{\mK}_s + \bar{\bm{\Sigma}}_s)^{-1} \bar{\vk}_s(\vx),
        \end{aligned}
    \]
    where $\bar{\vk}_s(\vx)$, $\bar{\mK}_s$, $\bar{\bm{\Sigma}}_s$ and $\bar{\vy}_s$ are defined as in \cref{lem:reduced GP}.
    As $\ell \to 0^+$, $\bar{\mK}_s$ converges to $\Psi(\bm{0}) \mI_{s}$, and $\bar{\vk}_s(\vx)$ converges to $( \Psi(\bm{0}) \1_{\set{\vx = \bar{\vx}_1}}, \dots, \Psi(\bm{0}) \1_{\set{\vx = \bar{\vx}_s}} )^\transpose$.
    Then
    \[
        \lim_{\ell \to 0^+} \mu_t(\vx)
        = \sum_{\lambda=1}^{s} \frac{\bar{y}_{\lambda} \Psi(\bm{0}) \1_{\set{\vx = \bar{\vx}_{\lambda}}}}{\Psi(\bm{0}) + \frac{\sigma^2}{n_{\lambda}}}
        = \frac{\sum_{i=1}^{t} y_i \1_{\set{\vx = \vx_i}}}{\frac{\sigma^2}{\Psi(\bm{0})} + \sum_{i=1}^{t} \1_{\set{\vx = \vx_i}}},
    \]
    and
    \[
        \begin{aligned}
            \lim_{\ell \to 0^+} \sigma^2_t(\vx)
            &= \Psi(\bm{0}) - \sum_{\lambda=1}^{s} \frac{\Psi(\bm{0})^2 \1_{\set{\vx = \bar{\vx}_{\lambda}}}}{\Psi(\bm{0}) + \frac{\sigma^2}{n_{\lambda}}} \\
            &= \Psi(\bm{0}) \left( 1 - \frac{\sum_{i=1}^{t} \1_{\set{\vx = \vx_i}}}{\frac{\sigma^2}{\Psi(\bm{0})} + \sum_{i=1}^{t} \1_{\set{\vx = \vx_i}}} \right) \\
            &= \frac{\sigma^2}{\frac{\sigma^2}{\Psi(\bm{0})} + \sum_{i=1}^{t} \1_{\set{\vx = \vx_i}}},
        \end{aligned}
    \]
    which completes the proof.
\end{proof}

\begin{remark}\label{rmk:degenerate GP vs KDE}
    By comparing \cref{eqn:degenerate GP variance} with the vanishing-bandwidth limit of the density-based uncertainty function, given by
    \[
        \lim_{\ell \to 0^+} \hsigma_t^2(\vx) 
        = \frac{1}{\varrho + \Psi(\bm{0}) \sum_{i=1}^{t} \1_{\set{\vx = \vx_i}}}
        = \frac{1 / \Psi(\bm{0})}{\varrho / \Psi(\bm{0}) + \sum_{i=1}^{t} \1_{\set{\vx = \vx_i}}},
    \]
    we observe that as $\ell \to 0^+$, $\hsigma_t^2$ shares a similar limiting structural form with $\sigma_t^2$.
    Thus $\varrho$ plays the role of the additive noise variance $\sigma^2$: both provide a regularizing constant in the denominator that prevents the Dirac-comb term $\sum_{i=1}^t \1_{\set{\vx=\vx_i}}$ from becoming singular.
\end{remark}

\subsection{Proofs for \texorpdfstring{\cref{sec:prediction error bound}}{Section 4.1}}
\label{app:proof of prediction error bound}

\begin{proof}[Proof of \cref{lem:uniform error of extended KR}]
    For any $\lambda \ge 0$, Hoeffding's inequality implies that the average noise satisfies
    \[
        \PP \left[ \left| \frac{1}{t} \sum_{i=1}^{t} \varepsilon_i \right| \ge \lambda \right]
        \le 2 \exp \left( - \frac{\lambda^2 t}{2 \varsigma^2} \right).
    \]
    Given $\delta \in (0, 1)$, let $\lambda = \sqrt{2 \varsigma^2 \ln(2 / \delta) / t}$.
    Then, with probability at least $1 - \delta$,
    \[
        \begin{aligned}
            \left| f(\vx) -  \frac{1}{t} \sum_{i=1}^{t} y_i \right|
            &\le \frac{1}{t} \sum_{i=1}^{t} \left| f(\vx) - f(\vx_i) \right| + \left| \frac{1}{t} \sum_{i=1}^{t} \varepsilon_i \right| \\
            &\le 2 \sup_{\vx' \in \fX} |f(\vx')| + \sqrt{2 t^{-1} \varsigma^2 \ln(2 / \delta)},
        \end{aligned}
    \]
    which completes the proof.
\end{proof}

\begin{lemma}
    \label{lem:lower bound volume}
    Suppose \cref{asm:domain} holds.
    For any $D > 0$, there exists a constant $v_0 > 0$ depending only on $\fX, D$.
    For any $\epsilon \in (0, D]$ and $\vx \in \fX$,
    \begin{equation}
        \vol(B(\vx, \epsilon) \cap \fX) \ge v_0 \epsilon^d,
    \end{equation}
    where $\vol(\cdot)$ denotes the Lebesgue measure in $\RR^d$.
\end{lemma}

\begin{proof}
    Since $\fX$ has non-empty interior, there exists a point $\vc \in \fX$ and a radius $r > 0$ such that $B(\vc, r) \subset \fX$.
    Let $t = \frac{\epsilon/2}{\| \vx - \vc \| + \epsilon/2} \in (0, 1)$. 
    Define $\vy \coloneqq \vx + t(\vc - \vx) \in \fX$ as a convex combination of $\vx$ and $\vc$, then $\| \vx - \vy \| = \frac{\| \vx - \vc \| \epsilon/2}{\| \vx - \vc \| + \epsilon/2} < \epsilon/2$.
    Let $s = \min\set{tr, \epsilon/2}$. The boundedness of $\fX$ yields
    \[
        s \ge \frac{\epsilon}{2} \min\set{\frac{r}{\diam(\fX) + \epsilon/2}, 1} \ge \epsilon \min\set{\frac{r}{2\diam(\fX) + D}, \frac{1}{2}}.
    \]
    For any $\vz \in B(\vy, s)$, we have $\| \vz - \vx \| \le \| \vz - \vy \| + \| \vx - \vy \| \le s + \epsilon/2 \le \epsilon$, then $\vz \in B(\vx, \epsilon)$. Moreover, by $\| \vz - \vy \| \le s \le t r$,
    \[
        \vz = \vy + \vz - \vy
        = (1-t)\vx + t\vc + t \cdot \frac{\vz-\vy}{t}
        = (1-t)\vx + t \left(\vc + \frac{\vz-\vy}{t}\right)
        \in \fX.
    \]
    Hence, $B(\vy, s) \subset B(\vx, \epsilon) \cap \fX$. The proof is completed by
    \[
        \vol( B(\vy, s) ) 
        = \vol(\unitball_d) s^d
        \ge \vol(\unitball_d) \left( \epsilon \min\set{\frac{r}{2\diam(\fX) + D}, \frac{1}{2}} \right)^d
        \eqqcolon v_0 \epsilon^d > 0,
    \]
    where $\unitball_d$ is the $d$-dimensional unit ball.
\end{proof}

\begin{proof}[Proof of \cref{lem:bound KDE by fill distance}]
    Since $\Psi(\bm{0}) = 1$ and $\Psi$ is continuous, there exists $0 < \epsilon \le \min\set{R_{\Psi} / 2, 2 \ellmax^{-1} D}$ such that $c_0 \coloneqq \inf_{\| \vx \| \le \epsilon} \Psi(\vx) > 0$, where $D$ is specified in \cref{lem:lower bound volume}. 
    Hence,
    \[
        W_t(\vx)
        = \sum_{i=1}^{t} \Psi\left( \frac{\vx - \vx_i}{\ell_t} \right)
        \ge \card{(\closure{B}(\vx, \epsilon \ell_t) \cap \mX_t)} c_0.
    \]
    By the definition of the fill distance, for any $\vx \in \fX$, there exists $\vx_i \in \mX_t$ such that $\| \vx - \vx_i \| \le \filldist_{\fX, \mX_t}$, so $\mX_t$ is a $\filldist_{\fX, \mX_t}$-covering of $\fX$. 
    If $\filldist_{\fX, \mX_t} \le \epsilon \ell_t / 2$, then $\closure{B}(\vx, \epsilon \ell_t) \cap \mX_t$ is a $\filldist_{\fX, \mX_t}$-covering of $B( \vx, \epsilon \ell_t / 2) \cap \fX$. 
    Write $\closure{B}(\vx, \epsilon \ell_t) \cap \mX_t = \set{\vz_i}_{i=1}^{N}$ with $N \coloneqq \card{(\closure{B}(\vx, \epsilon \ell) \cap \mX_t)}$. 
    The volume of $B(\vx, \epsilon \ell_t / 2) \cap \fX$ is bounded by the total volume of the covering balls, hence
    \[
        \vol(B(\vx, \epsilon \ell_t / 2) \cap \fX)
        \le \vol\left( \bigcup_{i=1}^{N} B(\vz_i, \filldist_{\fX, \mX_t}) \right)
        \le N \filldist_{\fX, \mX_t}^d \vol(\unitball_d).
    \]
    Since $\filldist_{\fX, \mX_t} \le \epsilon \ell_t / 2 \le D$, applying \cref{lem:lower bound volume} to $\vol(B(\vx, \epsilon \ell_t / 2) \cap \fX)$ and combining the above inequalities gives
    \[
        W_t(\vx) 
        \ge \frac{\vol(B(\vx, \epsilon \ell_t / 2) \cap \fX)}{\vol(\unitball_d) \filldist_{\fX, \mX_t}^d} c_0
        \ge \ell_t^d \filldist_{\fX, \mX_t}^{-d} \frac{(\epsilon/2)^d v_0 c_0}{\vol(\unitball_d)}
        \eqqcolon w_0 \ell_t^d \filldist_{\fX, \mX_t}^{-d}.
    \]
    This proof is completed by setting $\filldist_0 = \epsilon / 2$.
\end{proof}

\subsection{Proofs for \texorpdfstring{\cref{sec:algorithmic consistency}}{Section 4.2}}
\label{app:proof of algorithmic consistency}

\begin{lemma}[SNEB property of DE]
    \label{lem:SNEB DE}
    Suppose \cref{asm:domain,asm:function,asm:kernel} hold. For any fixed $\ell_t \equiv \ell$, the following statements hold:
    \begin{enumerate}[label=(\alph*)]    
        \item\label{it1:SNEB DE} For any sequence $\set{\vx_t}_{t \ge 1} \subset \fX$, let $\mX_t = \set{\vx_1, \dots, \vx_t}$. If there exists $\vx \in \fX$ with $\inf_{t} d(\vx, \mX_t) > R_{\Psi} \ell$, then $\hsigma(\vx; \mX_t) = \varrho^{-1/2}$ for all $t \ge 1$.

        \item\label{it2:SNEB DE} For any convergent sequence $\set{\vx_t}_{t \ge 1} \subset \fX$ and any sequence of finite sets $\tilde{\mX}_t$ containing $\set{\vx_1, \dots, \vx_t}$, we have $\hsigma(\vx_t; \tilde{\mX}_{t-1}) = \bigO(t^{-1/2})$ as $t \to \infty$.
    \end{enumerate}
\end{lemma}

\begin{proof}
    \proofparagraph{\cref{it1:SNEB DE}}
    Since $\inf_{t} d(\vx, \mX_t) > R_{\Psi} \ell$ implies $\| \vx - \vx_i \| > R_{\Psi} \ell$ for every $i \ge 1$.
    By \cref{asm:kernel}, we have $\Psi( (\vx - \vx_i) / \ell ) = 0$ for all $i$, hence
    \[
        W(\vx; \mX_t) = \sum_{i=1}^{t} \Psi\left( \frac{\vx - \vx_i}{\ell} \right) = 0,
    \]
    and therefore $\hsigma(\vx; \mX_t) = \varrho^{-1/2}$ by definition.

    \proofparagraph{\cref{it2:SNEB DE}}
    Since $\Psi(\bm{0}) = 1$ and $\Psi$ is continuous, there exists $\epsilon > 0$ such that $\inf_{\| \vx \| \le \epsilon} \Psi(\vx) > 0$. As $\set{\vx_t}$ converges, there exists $N = N(\epsilon)$ with $\| \vx_i - \vx_j \| \le \epsilon \ell$ for all $i, j \ge N$. Hence, for $t \ge N + 1$,
    \[
        W(\vx_t; \tilde{\mX}_{t-1}) \ge \sum_{i=N}^{t-1} \Psi\left(\frac{\vx_t - \vx_i}{\ell}\right) \ge (t - N) \inf_{\| \vx \| \le \epsilon} \Psi(\vx).
    \]
    Thus $W(\vx_t; \tilde{\mX}_{t-1}) = \Omega(t)$ and $\hsigma(\vx_t; \tilde{\mX}_{t-1}) = (\Omega(t) + \varrho)^{-1/2} = \bigO(t^{-1/2})$ as $t \to \infty$.
\end{proof}

\begin{proof}[Proof of \cref{lem:SNEB IKR-UCB}]
    \proofparagraph{\cref{it1:SNEB IKR-UCB}}
    By \cref{lem:SNEB DE}\cref{it1:SNEB DE}, $W(\vx; \mX_t) = 0$ for all $t \ge 1$.
    By \cref{lem:uniform error of extended KR}, for any $\lambda \ge 0$, with probability at least $1 - 2 \exp\left( -\frac{\lambda^2 t}{2 \varsigma^2} \right)$,
    \[
        |m(\vx; \mD_t)|
        \le |f(\vx)| + |f(\vx) - m(\vx; \mD_t)|
        \le |f(\vx)| + 2 \| f \|_{\infty} + \lambda
        \le 3 \| f \|_{\infty} + \lambda.
    \]
    Since $\lim_{t \to \infty} \beta_t = \infty$, for any $\epsilon > 0$, there exists $N = N(\epsilon)$ such that for all $t \ge N$, the choice $\lambda = \beta_t \epsilon - 3 \| f \|_{\infty} \ge 0$ is valid. 
    Hence,
    \[
        \PP\left[ |\beta_t^{-1} m(\vx; \mD_t)| > \epsilon \right]
        \le 2 \exp\left( -\frac{(\beta_t \epsilon - 3 \| f \|_{\infty})^2 t }{2 \varsigma^2 (1 + \varrho^{-1})} \right)
        \to 0 \quad \text{ as } t \to \infty.
    \]
    Thus $|\beta_t^{-1} m(\vx; \mD_t)| \pto 0$ as $t \to \infty$. Since $\hsigma(\vx; \mX_t) = \varrho^{-1/2}$, it follows that $\tilde{a}(\vx; \mD_t) \pto \varrho^{-1/2}$ as $t \to \infty$.

    \proofparagraph{\cref{it2:SNEB IKR-UCB}}
    By \cref{lem:SNEB DE}\cref{it2:SNEB DE}, $W(\vx_{t+1}; \tilde{\mX}_t) = \bigOmega(t)$ and $\hsigma(\vx_{t+1}; \tilde{\mX}_t) = \bigO(t^{-1/2})$ as $t \to \infty$. 
    For clarity we shift the index and denote $n = n(t) \coloneqq \card{\tilde{\mX}_t} \ge t$.
    Applying \cref{thm:uniform error of KR by KDE} to dataset $\tilde{\mD}_t$ of size $n$, there exists a universal constant $C_0 > 0$ such that, for any $\delta \in (0, 1)$, with probability at least $1 - \delta$,
    \[
        \begin{aligned}
            &\left| f(\vx_{t+1}) - m(\vx_{t+1}; \tilde{\mD}_t) \right| \\
            \le& C_0 \left( \left( \hsigma(\vx_{t+1}; \tilde{\mX}_t) + n^{-1} W^{-1}(\vx_{t + 1}; \tilde{\mX}_t) \right) \sqrt{d \ln(n + 1) + \ln(4 / \delta)} + \omega_f(\ell) \right).
        \end{aligned}
    \]
    Since $W(\vx_{t+1}; \tilde{\mX}_t) \to \infty$ as $t \to \infty$, there exists $N_1 > 0$ such that for all $t \ge N_1$, $W(\vx_{t+1}; \tilde{\mX}_t) + \varrho \le W^2(\vx_{t+1}; \tilde{\mX}_t)$. Hence for $t \ge N_1$,
    \[
        n^{-1} W^{-1}(\vx_{t + 1}; \tilde{\mX}_t) \le (W(\vx_{t+1}; \tilde{\mX}_t) + \varrho)^{-1/2} = \hsigma(\vx_{t+1}; \tilde{\mX}_t),
    \]
    and therefore
    \[
        \left| f(\vx_{t+1}) - m(\vx_{t+1}; \tilde{\mD}_t) \right| 
        \le C_0 \left( 2 \hsigma(\vx_{t+1}; \tilde{\mX}_t) \sqrt{d \ln(n + 1) + \ln(4 / \delta)} + \omega_f(\ell) \right).
    \]
    Choose $\lambda \ge 0$ and set
    \[
        \delta = 4 \exp\left( -\frac{\lambda^2}{4 C_0^2 \hsigma^2(\vx_{t+1}; \tilde{\mX}_t)} + d \ln(n + 1) \right) \in (0, 1).
    \]
    Combining this with $\omega_f(\ell) \le 2 \| f \|_{\infty}$ yields
    \[
            |m(\vx_{t+1}; \tilde{\mD}_t)|
            \le |f(\vx_{t+1})| + C_0 \omega_f(\ell) + \lambda
            \le (2 C_0 + 1) \| f \|_{\infty} + \lambda.
    \]
    As $\hsigma(\vx_{t+1}; \tilde{\mX}_t) = \bigO(t^{-1/2})$, there exist $N_2, A_1 > 0$ such that for all $t \ge N_2$, $( 4 C_0^2 \hsigma^2(\vx_{t+1}; \tilde{\mX}_t) )^{-1} \ge A_1 t$.
    Also there exist $N_3, A_2 > 0$ such that for all $t \ge N_3$, $d \ln(n + 1) \le A_2 \ln n$.
    Set $\lambda = \beta_n \epsilon_t - C_1 \ge 0$, where $\epsilon_t \coloneqq \beta_n^{-1} \max\set{6 C_1, \sqrt{\frac{3 A_2 \ln n}{A_1 t}}}$ and $C_1 \coloneqq (2 C_0 + 1) \| f \|_{\infty}$. 
    For all $t \ge N \coloneqq \max\set{N_1, N_2, N_3}$, we obtain
    \[
        \begin{aligned}
            \PP\left[ |\beta_n^{-1} m(\vx_{t+1}; \tilde{\mD}_t)| > \epsilon_t \right]
            &\le 4 \exp\left( -\frac{\left( \beta_n \epsilon_t - C_1 \right)^2}{4 C_0^2 \hsigma^2(\vx_{t+1}; \tilde{\mX}_t)} + d \ln(n + 1) \right) \\
            &\le 4 \exp\left( -A_1 t \left( \beta_n \epsilon_t - C_1 \right)^2 + A_2 \ln n \right) \\
            &= 4 \exp\left( -A_1 \beta_n^2 \epsilon_t^2 t + 2 A_1 C_1 \beta_n \epsilon_t t + A_2 \ln n \right) \\
            &\le 4 \exp\left( -A_1 \beta_n^2 \epsilon_t^2 t + \frac{A_1}{3} \beta_n^2 \epsilon_t^2 t + \frac{A_1}{3} \beta_n^2 \epsilon_t^2 t \right) \\
            &= 4 \exp\left( - \frac{A_1}{3} \beta_n^2 \epsilon_t^2 t \right) \\
            &= 4 \exp\left( -\max\set{12 A_1 C_1^2 t, A_2 \ln n} \right) 
            \to 0 \quad \text{ as } t \to \infty,
        \end{aligned}
    \]
    where the last inequality and the last equality follow from the definition of $\epsilon_t$.
    The choice $\delta \in (0, 1)$ is valid for these $t$.
    Since $\beta_n = \bigOmega(\sqrt{\ln n})$, we obtain
    \[
        \epsilon_t
        = \max\set{\bigO\left( \frac{1}{\sqrt{\ln n}} \right), \bigO\left( \frac{1}{\sqrt{t}} \right)}
        \to 0 \quad \text{ as } t \to \infty.
    \]
    Hence, $|\beta_n^{-1} m(\vx_{t+1}; \tilde{\mD}_t)| \pto 0$ as $t \to \infty$. 
    Combined with $\tilde{a}(\vx_{t+1}; \tilde{\mD}_t) \pto 0$, this implies $\hsigma(\vx_{t+1}; \tilde{\mX}_t) \to 0$ as $t \to \infty$, completing the proof.
\end{proof}

\begin{proof}[Proof of \cref{cor:consistency DE}]
    \proofparagraph{\cref{it1:consistency DE}}
    Suppose for contradiction that $\inf_{t} \filldist_{\fX, \mX_t} > R_{\Psi} \ell$. 
    Then, the monotonicity of $\filldist_{\fX, \mX_t}$ with respect to $t$ implies $\lim_{t \to \infty} \filldist_{\fX, \mX_t} > R_{\Psi} \ell$. 
    Hence, there exist constants $N$ and $\epsilon > 0$ such that for all $t \ge N$, $\filldist_{\fX, \mX_t} = \sup_{\vx \in \fX} d(\vx, \mX_t) \ge R_{\Psi} \ell + \epsilon$. 
    Let $F_t = \set{\vx \in \fX : d(\vx, \mX_t) \ge R_{\Psi} \ell + \epsilon}$. 
    The continuity of $d(\vx, \mX_t)$ with respect to $\vx$ and the compactness of $\fX$ imply that $F_t$ is closed and non-empty. 
    Since $d(\vx, \mX_t) \ge d(\vx, \mX_{t+1})$, we have $\fX \supset F_t \supset F_{t+1}$. 
    By Cantor's intersection theorem, $\cap_{t=N}^{\infty} F_t \neq \varnothing$. 
    Thus, there exists $\vx' \in \cap_{t=N}^{\infty} F_t$ such that for all $t \ge N$, $d(\vx', \mX_t) \ge R_{\Psi} \ell + \epsilon$. 
    Hence, $\inf_t d(\vx', \mX_t) = \lim_{t \to \infty} d(\vx', \mX_t) > R_{\Psi} \ell$. 
    By \cref{lem:SNEB DE}\cref{it1:SNEB DE}, $\hsigma(\vx'; \mX_t) = \varrho^{-1/2}$ for all $t \ge 1$.

    By compactness of $\fX$, $\set{\vx_t}_{t \ge 1}$ has a convergent subsequence $\set{\vx_{t_n}}_{n \ge 1}$.
    By \cref{lem:SNEB DE}\cref{it2:SNEB DE}, $\lim_{n \to \infty} \hsigma(\vx_{t_n}; \mX_{t_n - 1}) = 0$.
    Hence,
    \[
        \lim_{n \to \infty} \hsigma(\vx'; \mX_{t_n - 1}) - \hsigma(\vx_{t_n}; \mX_{t_n - 1}) = \varrho^{-1/2}.
    \]
    However, the DE criterion \eqref{eqn:min KDE} gives $\vx_{t_n} \in \argmax_{\vx \in \fX} \hsigma(\vx; \mX_{t_n - 1})$, so 
    \[
        \hsigma(\vx'; \mX_{t_n - 1}) \le \hsigma(\vx_{t_n}; \mX_{t_n - 1}) \quad \forall n \ge 1.
    \]
    This yields a contradiction.
    Therefore, $\inf_{t} \filldist_{\fX, \mX_t} \le R_{\Psi} \ell$.

    \proofparagraph{\cref{it2:consistency DE}}
    Let $\set{\ell'_n}_{n \ge 1}$ be a deterministic sequence chosen to satisfy $\lim_{n \to \infty} \ell'_n = 0$, such as $\ell'_n = 1 / n$ for all $n \ge 1$.
    We recursively construct the bandwidth schedule $\set{\ell_t}_{t \ge 1}$ alongside a sequence of times $\set{\tau_n}_{n \ge 0}$. 
    Let $\tau_0 = 0$. 
    For each $n \ge 1$, maintain the bandwidth $\ell_t = \ell'_n$ for all $t > \tau_{n-1}$ until the following stopping time is reached:
    \[
        \tau_n = \inf\set{t > \tau_{n-1} : \filldist_{\fX, \mX_t} < 2 R_{\Psi} \ell'_n}.
    \]
    Assume inductively that $\tau_{n-1} < \infty$ for some $n \ge 1$.
    The iterates $\set{\vx_t}_{t > \tau_{n-1}}$ can be seen as being generated by a restarted DE algorithm given the initial set $\mX_{\tau_{n-1}}$ and a fixed bandwidth $\ell'_n$.
    By part \cref{it1:consistency DE}, these iterates satisfy $\inf_{t > \tau_{n-1}} \filldist_{\fX, \mX_t} \le R_{\Psi} \ell'_n$, and therefore $\tau_n < \infty$. 
    By mathematical induction, $\tau_n < \infty$ holds for all $n \ge 1$.
    Since $\set{\tau_n}_{n \ge 0}$ is a sequence of finite times, the recursively constructed bandwidth can be equivalently written as
    \[
        \ell_t = \sum_{n=1}^{\infty} \ell'_n \1_{\set{\tau_{n-1} < t \le \tau_n}}, \quad \forall t \ge 1.
    \]
    Since $\filldist_{\fX, \mX_t}$ is monotonically non-increasing with respect to $t$, $\set{\tau_n}_{n \ge 0}$ is strictly increasing and $\tau_n \asto \infty$.
    Consequently, as $t \to \infty$, the active bandwidth $\ell_t$ tracks the tail of $\set{\ell'_n}_{n \ge 1}$, yielding $\lim_{t \to \infty} \ell_t = 0$ and $\lim_{n \to \infty} \filldist_{\fX, \mX_{\tau_n}} = \lim_{n \to \infty} 2 R_{\Psi} \ell'_n = 0$.
    Combining this with the monotonicity of $\filldist_{\fX, \mX_t}$ gives $\lim_{t \to \infty} \filldist_{\fX, \mX_t} = 0$.
\end{proof}

\newpage
\section{Additional numerical results}
\label{app:additional numerics}

The supplementary results in \cref{fig:benchmark noisefree batch 2,fig:benchmark noisy batch 2} generally corroborate the main text regarding the overall competitiveness of BOKE and BOKE+.

\begin{figure}[H]
    \centering
    \includegraphics[width=.98\textwidth]{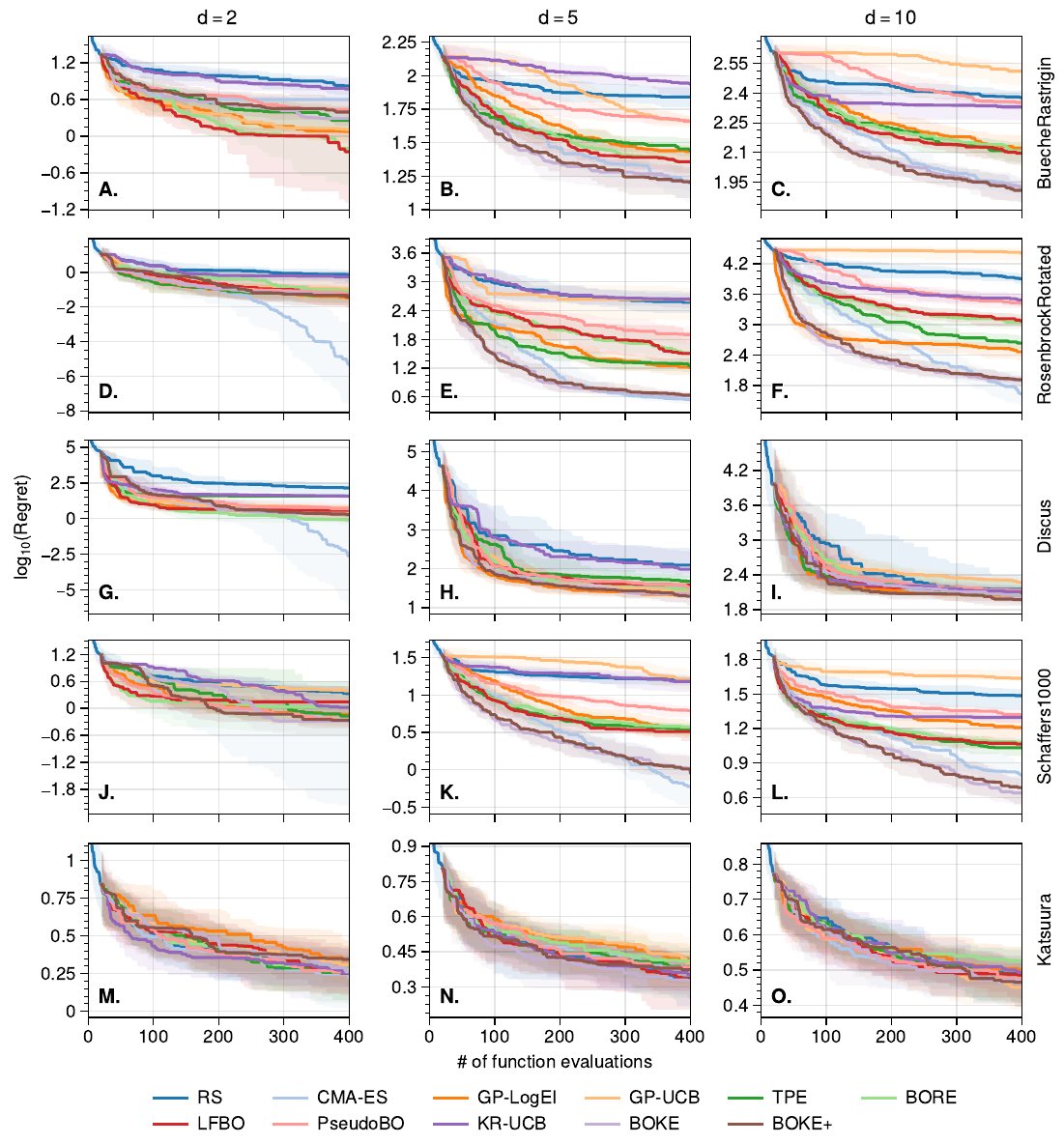}
    \caption{
        Logarithm of the simple regret and its distribution for noise-free evaluations on benchmark functions.
        Each column corresponds to a dimension $d \in \set{2, 5, 10}$; each row corresponds to a synthetic function, with the function name shown at the right of the row.
        Each plot displays the median and interquartile range over $50$ runs.
        All algorithms share the same $20$ initial sampling points generated by LHS.
    }
    \label{fig:benchmark noisefree batch 2}
\end{figure}

Our algorithms still underperform in certain low-dimensional settings ($d = 2$).
Notably, on the B\"uche-Rastrigin and Discus functions (panels A and G), density-ratio-based methods such as BORE and LFBO exhibit superior exploitation capabilities and achieve exceptionally rapid convergence.
Finally, these additional benchmarks further highlight the sample efficiency and robustness of BO methods compared to CMA-ES, particularly under increasing dimensionality and observation noise.

\begin{figure}[H]
    \centering
    \includegraphics[width=.98\textwidth]{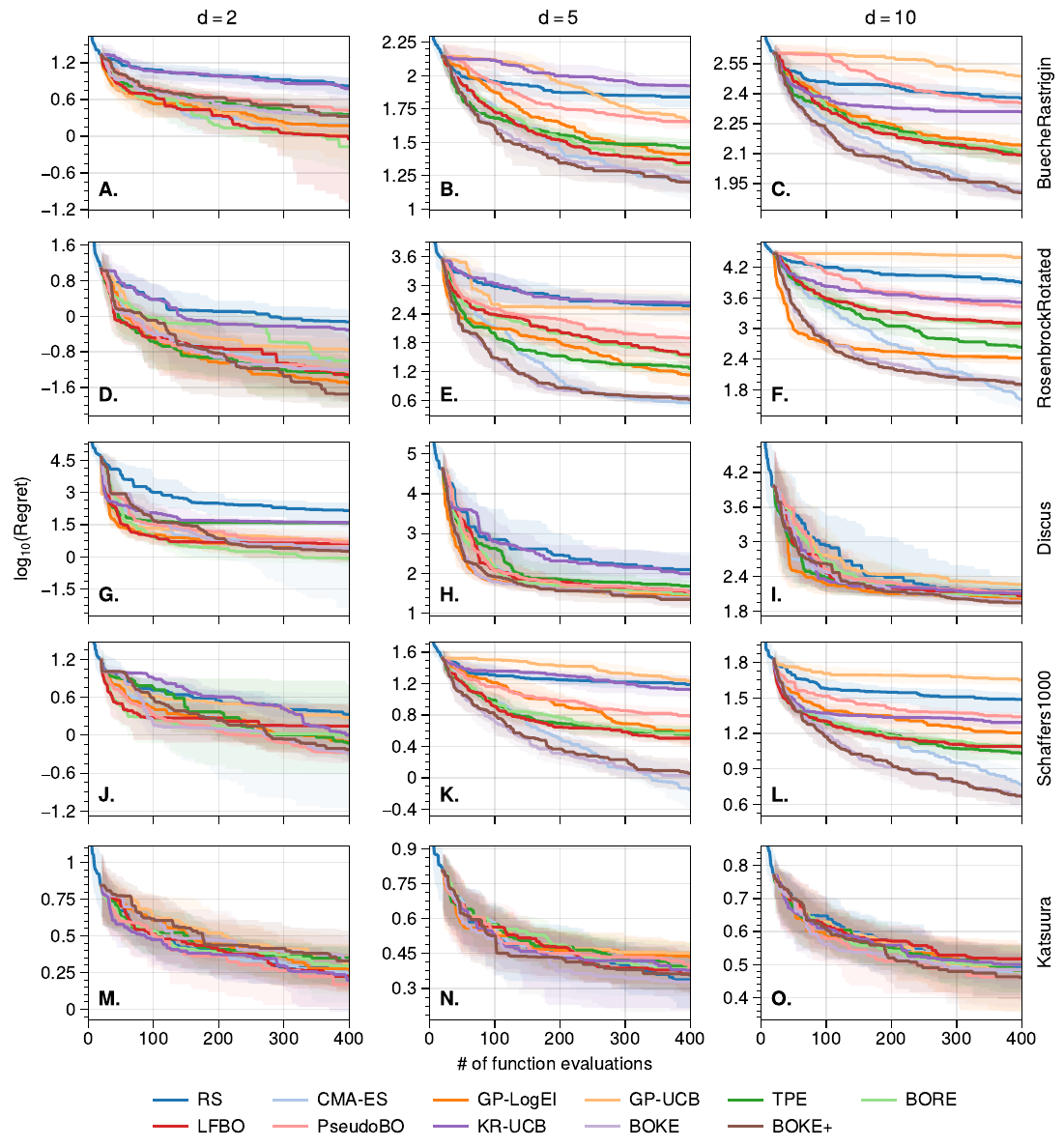}
    \caption{
        Logarithm of the simple regret and its distribution for noisy evaluations on benchmark functions.
        Each column corresponds to a dimension $d \in \set{2, 5, 10}$; each row corresponds to a synthetic function, with the function name shown at the right of the row.
        Each plot displays the median and interquartile range over $50$ runs.
        All algorithms share the same $20$ initial sampling points generated by LHS.
    }
    \label{fig:benchmark noisy batch 2}
\end{figure}

\printbibliography[heading=bibintoc, title=\ebibname]

\end{document}